%% file: new_necessary_condition.tex
\documentclass[a4paper]{article}

\usepackage{amsmath}
\usepackage{amsfonts}
\usepackage{amssymb}
\usepackage{amsthm}
\usepackage[draft]{todonotes}
\usepackage[titletoc,toc,title]{appendix}
\usepackage{array}
\usepackage{booktabs}
\usepackage{color}
\usepackage{enumerate}

\usepackage{bibentry}
\nobibliography*

\DeclareMathOperator{\tr}{tr}
\DeclareMathOperator{\diag}{diag}

\newtheorem{corollary}{Corollary}
\newtheorem{theorem}{Theorem}
\newtheorem{lemma}{Lemma}

\theoremstyle{definition}
\newtheorem{observation}{Observation}
\newtheorem{assumption}{Assumption}

\begin{document}
\bibliographystyle{abbrv}

\title{A Necessary Condition for the Spectrum of Nonnegative Symmetric $ 5 \times 5 $ Matrices}
\author{Raphael Loewy \and Oren Spector}
\date{\today}
\maketitle

\begin{abstract}
Let $A$ be a nonnegative symmetric $ 5 \times 5 $ matrix with eigenvalues $ \lambda_1 \geq \lambda_2 \geq \lambda_3 \geq \lambda_4 \geq \lambda_5 $. We show that if $ \sum_{i=1}^{5} \lambda_{i} \geq \frac{1}{2} \lambda_1 $ then $ \lambda_3 \leq \sum_{i=1}^{5} \lambda_{i} $. McDonald and Neumann showed that $ \lambda_1 + \lambda_3 + \lambda_4 \geq 0 $. Let $ \sigma = \left( \lambda_1, \lambda_2, \lambda_3, \lambda_4, \lambda_5 \right) $ be a list of decreasing real numbers satisfying:
\begin{enumerate}
\item $ \sum_{i=1}^{5} \lambda_{i} \geq \frac{1}{2} \lambda_1 $,
\item $ \lambda_3 \leq \sum_{i=1}^{5} \lambda_{i} $,
\item $ \lambda_1 + \lambda_3 + \lambda_4 \geq 0 $,
\item the Perron property, that is $ \lambda_1 = \max_{\lambda \in \sigma} \left| \lambda \right| $.
\end{enumerate}
We show that $ \sigma $ is the spectrum of a nonnegative symmetric $ 5 \times 5 $ matrix. Thus, we solve the symmetric nonnegative inverse eigenvalue problem for $ n = 5 $ in a region for which a solution has not been known before.
\end{abstract}

\input{new_necessary_condition_body}

\begin{appendices}

\input{new_necessary_condition_appendices}

\end{appendices}

\bibliography{../bib/sniep}

\end{document}

%% file: new_necessary_condition_body.tex
\section{Introduction}

Let $ \sigma = \left( \lambda_1, \lambda_2, \dots, \lambda_n \right) $ be a list of complex numbers. The problem of determining when $ \sigma $ is the spectrum of a nonnegative $ n \times n $ matrix is called the nonnegative inverse eigenvalue problem (NIEP). When $ \sigma $ consists of real numbers, the problem of determining when $ \sigma $ is the spectrum of a nonnegative (symmetric nonnegative) $ n \times n $ matrix is called RNIEP (SNIEP). All three problems are currently unsolved in the general case, more precisely for any list consisting of $ n \geq 5 $ numbers.

Loewy and London~\cite{RefWorks:39} have solved NIEP in the case of $ 3 \times 3 $ matrices and RNIEP in the case of $ 4 \times 4 $ matrices. Moreover, RNIEP and SNIEP are the same in the case of $ n \times n $ matrices for $ n \leq 4 $. This can be seen from \cite{RefWorks:39} and a paper by Fiedler~\cite{RefWorks:59}. However, it has been shown by Johnson, Laffey and Loewy~\cite{RefWorks:47} that RNIEP and SNIEP are different in general.

In this paper we consider SNIEP in the case $ n = 5 $. It is known that RNIEP and SNIEP are different for $ n = 5 $. Consider the list $ \sigma = \left( 3 + t, 3, -2, -2, -2 \right) $. It has been shown by Loewy and Hartwig~\cite{Refworks:100} that the smallest $t$ for which $ \sigma $ is the spectrum of a nonnegative symmetric $ 5 \times 5 $ matrix is $1$. On the other hand, it has been shown by Meehan~\cite{RefWorks:71} that there exists $ t < 0.52 $ such that $ \sigma $ is the spectrum of a nonnegative $ 5 \times 5 $ matrix.

The trace zero case, namely when $ \sum_{i=1}^{5} \lambda_{i} = 0 $, has been solved by Spector~\cite{RefWorks:79}. Other results about SNIEP in case $ n = 5 $ can be found in Egleston, Lenker and Narayan~\cite{RefWorks:16}, Loewy and McDonald~\cite{RefWorks:36}, and McDonald and Neumann~\cite{RefWorks:45}.

Let $A$ be a nonnegative symmetric $ 5 \times 5 $ matrix with eigenvalues $ \lambda_1 \geq \lambda_2 \geq \lambda_3 \geq \lambda_4 \geq \lambda_5 $. We show that if $ \sum_{i=1}^{5} \lambda_{i} \geq \frac{1}{2} \lambda_1 $ then $ \lambda_3 \leq \sum_{i=1}^{5} \lambda_{i} $. It has been shown in~\cite{RefWorks:45} that $ \lambda_1 + \lambda_3 + \lambda_4 \geq 0 $. Let $ \sigma = \left( \lambda_1, \lambda_2, \lambda_3, \lambda_4, \lambda_5 \right) $ be a list of decreasing real numbers satisfying:
\begin{enumerate}
\item $ \sum_{i=1}^{5} \lambda_{i} \geq \frac{1}{2} \lambda_1 $,
\item $ \lambda_3 \leq \sum_{i=1}^{5} \lambda_{i} $,
\item $ \lambda_1 + \lambda_3 + \lambda_4 \geq 0 $,
\item the Perron property, that is $ \lambda_1 = \max_{\lambda \in \sigma} \left| \lambda \right| $.
\end{enumerate}
We show that $ \sigma $ is the spectrum of a nonnegative symmetric $ 5 \times 5 $ matrix. Thus, we solve SNIEP for $ n = 5 $ in a region for which a solution has not been known before.

As we consider in this paper only SNIEP with $ n = 5 $, we will not list specifically the many results that appear in the literature concerning NIEP, RNIEP and SNIEP. The relevant results that we shall use in this paper will be cited in the appropriate places.

The paper is organized as follows. In Section~\ref{sec:preliminaries_nnc} we state some preliminary results. In Section~\ref{sec:main_theorem} we state our main result and show that it suffices to prove the result for two matrix patterns similar to those mentioned in~\cite{RefWorks:36}. In Section~\ref{sec:pattern_H} we prove the result for the first pattern and in Section~\ref{sec:pattern_C} we prove the result for the second pattern. In order to ease the reading of the paper some of the more technical details used in Sections~\ref{sec:pattern_H} and~\ref{sec:pattern_C} are deferred to the Appendices. In the Appendices roots of polynomials of degree at most $4$ and with integer coefficients are computed with an accuracy of $ 10^{-10} $, which is sufficient for our purposes.

\section{Preliminaries} \label{sec:preliminaries_nnc}

Let $M$ be an $ n \times n $ matrix and let $ 1 \leq c_1 < c_2 < \dots < c_k \leq n $, where $ k \leq n $, be natural numbers. We denote by $ M[ c_1, c_2, \dots, c_k ] $ the principal submatrix of $M$ generated by taking only rows and columns $ c_1, c_2, \dots, c_k $ of $M$. We denote the spectral radius of $M$ by $ \rho \left( M \right) $, the determinant of $M$ by $ \det \left( M \right) $, the trace of $M$ by $ \tr \left( M \right) $ and the characteristic polynomial of $M$ by $ P_M \left( \lambda \right) = \det \left( \lambda \cdot I_n - M \right) $, where $ I_n $ is the $ n \times n $ identity matrix. The diagonal matrix with elements $ x_1, x_2, \dots, x_n $ on its main diagonal is denoted by $ \diag \left( x_1, x_2, \dots, x_n \right) $. When $ \mathbf{x} = \left( x_1, x_2, \dots, x_n \right) \in {\mathbb{R}}^n $ we write $ \mathbf{x} \geq 0 $ to denote $ x_i \geq 0 $ for $ i = 1, 2, \dots, n $.

Let $M$ be a nonnegative symmetric $ n \times n $ matrix. Let $ \left( \lambda_1, \lambda_2, \dots, \lambda_n \right) $ be the spectrum of $M$, where $ \lambda_1 \geq \lambda_2 \geq \dots \geq \lambda_n $. By the eigenvalue interlacing property of symmetric matrices we have the following observations:

\begin{observation} \label{obs:spectral_radius}
Any principal submatrix $M'$ of $M$ has $ \rho \left( M' \right) \leq \rho \left( M \right) $.
\end{observation}

\begin{observation} \label{obs:3x3_sufficient_condition}
If the largest eigenvalue of an $ \left( n - 2 \right) \times \left( n - 2 \right) $ principal submatrix of $M$ is less than or equal to $ \alpha $ then $ \lambda_3 \leq \alpha $.
\end{observation}

\begin{observation} \label{obs:4x4_sufficient_condition}
If the second largest eigenvalue of an $ \left( n - 1 \right) \times \left( n - 1 \right) $ principal submatrix of $M$ is less than or equal to $ \alpha $ then $ \lambda_3 \leq \alpha $.
\end{observation}

Another observation we shall often use is:
\begin{observation} \label{obs:negative_cp}
If $ P_M \left( \lambda \right) < 0 $ then $ \rho \left( M \right) > \lambda $. Let $ \mu_1, \mu_2, \mu_3 $ be the three largest eigenvalues of $M$ and $ \alpha $ some number such that $ \mu_3 \leq \alpha \leq \mu_2 \leq \mu_1 $. If $ \mu_2 < \lambda < \mu_1 $ then $ P_M \left( \lambda \right) < 0 $. Moreover, if for some $ \lambda \geq \alpha $ we have $ P_M \left( \lambda \right) < 0 $, then $ \mu_2 < \lambda < \mu_1 $.
\end{observation}

Two results will be of great importance to us in proving our main results. To state the first result we need the concept of an {\em extreme matrix}, introduced by Laffey~\cite{RefWorks:23}. A nonnegative $ n \times n $ matrix is called an {\em extreme matrix} if its spectrum $ \left( \lambda_1, \lambda_2, \dots, \lambda_n \right) $ has the property that for all $ \varepsilon > 0 $, $ \left( \lambda_1 - \varepsilon, \lambda_2 - \varepsilon, \dots, \lambda_n - \varepsilon \right) $ is not the spectrum of a nonnegative matrix. In particular, an extreme matrix cannot be similar to a positive matrix. In the context of this paper an extreme matrix means a symmetric extreme matrix. Such a matrix cannot be orthogonally similar to a positive symmetric matrix.

\begin{theorem}[\cite{RefWorks:23}] \label{th:extreme_L}
Let $ M = \left( m_{ij} \right) $ be an $ n \times n $ nonnegative symmetric matrix that is not orthogonally similar to a positive symmetric matrix. Then there exists an $ n \times n $ nonnegative, symmetric, nonzero matrix $ Y = \left( y_{ij} \right) $ such that
\begin{enumerate}
\item $ M Y = Y M $,
\item $ m_{ij} y_{ij} = 0 $ for all $i$, $j$.
\end{enumerate}
\end{theorem}
Actually Theorem~\ref{th:extreme_L} is stated for extreme matrices, but their only property being used in the proof is that of not being orthogonally similar to a positive symmetric matrix.

The second result is due to Loewy and McDonald~\cite{RefWorks:36}. To state this result we need the following definitions. Let
\begin{center}
\begin{tabular}{ *{3}{ >{$} l <{$} } }
\mathbf{c} = \left( 1, 1, 1, -1, -1 \right) , & \mathbf{d} = \left( 1, 1, 0, -1, -1 \right) , & \mathbf{e} = \left( 1, 0, 0, 0, -1 \right) , \\
\addlinespace
\mathbf{i} = \left( 1, \frac{1}{2}, \frac{1}{2}, -1, -1 \right) , & \mathbf{l} = \left( 1, 0, 0, -\frac{1}{2}, -\frac{1}{2} \right) , \\
\end{tabular}
\end{center}
and let $ \mathbb{F}_1 $ ($ \mathbb{F}_2 $) be the convex hulls of $ \mathbf{c}, \mathbf{d}, \mathbf{e}, \mathbf{l} $, ($ \mathbf{d}, \mathbf{e}, \mathbf{i}, \mathbf{l} $ respectively). Let $ \mathbb{U} $ be the convex hull of $ \mathbf{c}, \mathbf{d}, \mathbf{e}, \mathbf{i}, \mathbf{l} $, and let $ \mathbb{U}_1 = \mathbb{U} \setminus \left( \mathbb{F}_1 \cup \mathbb{F}_2 \right) $.

\begin{theorem}[\cite{RefWorks:36}] \label{th:extreme_LM}
Let $A$ be a symmetric extreme matrix with eigenvalues in $ \mathbb{U}_1 $. Then the pattern of $A$ (up to permutation) is either
\[ H = \left(
\begin{array}{ccccc}
+ & + & + & 0 & 0 \\ \noalign{\medskip}
+ & 0 & 0 & + & + \\ \noalign{\medskip}
+ & 0 & + & 0 & + \\ \noalign{\medskip}
0 & + & 0 & * & + \\ \noalign{\medskip}
0 & + & + & + & 0
\end{array}
\right) \quad
\text{or} \quad
C_0 = \left(
\begin{array}{ccccc}
+ & + & + & 0 & 0 \\ \noalign{\medskip}
+ & * & 0 & 0 & + \\ \noalign{\medskip}
+ & 0 & 0 & + & 0 \\ \noalign{\medskip}
0 & 0 & + & * & + \\ \noalign{\medskip}
0 & + & 0 & + & *
\end{array}
\right) , \]
where $+$ indicates a positive element and $*$ indicates a zero or a positive element.
\end{theorem}

\section{Main Theorem} \label{sec:main_theorem}

The following are the two main Theorems.
\begin{theorem} \label{th:main_nnc}
Let $ \sigma = \left( \lambda_1, \lambda_2, \dots, \lambda_5 \right) $ be a list of monotonically decreasing real numbers. If a nonnegative symmetric $ 5 \times 5 $ matrix has a spectrum $ \sigma $ and $ \sum_{i=1}^{5} \lambda_{i} \geq \frac{1}{2} \lambda_1 $ then $ \lambda_3 \leq \sum_{i=1}^{5} \lambda_{i} $.
\end{theorem}

In the next Theorem we solve SNIEP for $ n = 5 $ under the additional assumption $ \sum_{i=1}^{5} \lambda_{i} \geq \frac{1}{2} \lambda_1 $.

\begin{theorem}
Let $ \sigma = \left( \lambda_1, \lambda_2, \dots, \lambda_5 \right) $ be a list of monotonically decreasing real numbers such that $ \sum_{i=1}^{5} \lambda_{i} \geq \frac{1}{2} \lambda_1 $. Necessary and sufficient conditions for $ \sigma $ to be the spectrum of a nonnegative symmetric $ 5 \times 5 $ matrix are:
\begin{enumerate}

\item $ \lambda_1 = \max_{\lambda \in \sigma} \left| \lambda \right| $, \label{eq:Perron_Frobenius_condition}

\item $ \lambda_2 + \lambda_5 \leq \sum_{i=1}^{5} \lambda_{i} $, \label{eq:McDonald_Neumann_condition}

\item $ \lambda_3 \leq \sum_{i=1}^{5} \lambda_{i} $. \label{eq:new_condition}

\end{enumerate}
\end{theorem}

\begin{proof}
Condition~\ref{eq:Perron_Frobenius_condition} is the Perron-Frobenius condition, and~\ref{eq:McDonald_Neumann_condition} is the McDonald-Neumann condition (Lemma~4.1 in~\cite{RefWorks:45}). Note that the latter condition was originally stated in the case of irreducible matrices, but this is not needed. Condition~\ref{eq:new_condition} is necessary by Theorem~\ref{th:main_nnc}. Therefore, the conditions are necessary. We show the conditions are sufficient by considering two cases:
\begin{enumerate}

\item $ \lambda_3 \leq 0 $.

By condition~\ref{eq:Perron_Frobenius_condition} $ \lambda_1 \geq 0 $, and by assumption $ \sum_{i=1}^{5} \lambda_{i} \geq \frac{1}{2} \lambda_1 \geq 0 $.

If $ \lambda_2 \leq 0 $, then by Theorem~2.4 in~\cite{RefWorks:59} there is a nonnegative symmetric $ 5 \times 5 $ matrix realizing $ \sigma $.

Otherwise, $ \lambda_2 \geq 0 $. By condition~\ref{eq:McDonald_Neumann_condition} we have $ \lambda_1 + \lambda_3 + \lambda_4 \geq 0 $. Therefore, by Theorem~2.4 and Theorem~2.5 in~\cite{RefWorks:59} there is a nonnegative symmetric $ 5 \times 5 $ matrix realizing $ \sigma $ (see~\cite{RefWorks:16} for details). Note, that
this can also be proved by applying Theorem~5 in~\cite{RefWorks:79} for the case $ n = 5 $.

\item $ \lambda_3 \geq 0 $.

Let $ \sigma' = \left( \lambda_1, \lambda_2, \lambda_4, \lambda_5 \right) $. By condition \ref{eq:new_condition} we have $ \lambda_1 + \lambda_2 + \lambda_4 + \lambda_ 5 \geq 0 $, and by condition~\ref{eq:Perron_Frobenius_condition} also $ \lambda_1 = \max_{\lambda \in \sigma'} \left| \lambda \right| $. Therefore, by Theorem~3 in~\cite{RefWorks:39} and as RNIEP and SNIEP are the same for $ n = 4 $, there is a nonnegative symmetric $ 4 \times 4 $ matrix $A$ realizing $ \sigma' $. Then $ M = \left( \lambda_3 \right) \bigoplus A $ is a nonnegative symmetric $ 5 \times 5 $ matrix realizing $ \sigma $.

\end{enumerate}
\end{proof}

Note that the condition $ \sum_{i=1}^{5} \lambda_{i} \geq \frac{1}{2} \lambda_1 $ in Theorem~\ref{th:main_nnc} is essential. For example, the matrix
\[ \left(
\begin{array}{ccccc}
\frac{2}{25} & \frac{ \sqrt{130} }{25} & \frac{ \sqrt{130} }{25} & 0 & 0 \\ \noalign{\medskip}
\frac{ \sqrt{130} }{25} & 0 & 0 & 0 & \frac{ 3 \sqrt{70} }{50} \\ \noalign{\medskip}
\frac{ \sqrt{130} }{25} & 0 & 0 & \frac{ 3 \sqrt{70} }{50} & 0 \\ \noalign{\medskip}
0 & 0 & \frac{ 3 \sqrt{70} }{50} & \frac{17}{200} & \frac{91}{200} \\ \noalign{\medskip}
0 & \frac{ 3 \sqrt{70} }{50} & 0 & \frac{91}{200} & \frac{17}{200}
\end{array}
\right) \]
has a spectrum $ \left( 1, \frac{35}{100}, \frac{34}{100}, - \frac{72}{100}, - \frac{72}{100} \right) $ and its trace is $ \frac{25}{100} $. Here $ \lambda_3 > \sum_{i=1}^{5} \lambda_{i} $.

Let $M$ be a nonnegative symmetric $ 5 \times 5 $ matrix with a spectrum $ \sigma = \left( \lambda_1, \lambda_2, \dots, \lambda_5 \right) $ and $ \lambda_1 \geq \lambda_2 \geq \dots \geq \lambda_5 $. In what follows we point out in detail several cases where Theorem~\ref{th:main_nnc} holds, thus allowing us to make additional assumptions when we get to the proof of this Theorem.

\paragraph{Sign of eigenvalues.} Obviously, when $ \lambda_3 \leq 0 $ Theorem~\ref{th:main_nnc} holds. {\em Therefore, we may assume that $ \lambda_3 > 0 $ and so $ \lambda_1 \geq \lambda_2 > 0 $}. If $ \lambda_4 \geq 0 $ then the four largest eigenvalues of $M$ are nonnegative. By the Perron-Frobenius Theorem $ \lambda_1 + \lambda_5 \geq 0 $ so $ \lambda_3 \leq \sum_{i=2}^{4} \lambda_{i} \leq \sum_{i=1}^{5} \lambda_{i} $ and Theorem~\ref{th:main_nnc} holds. {\em Therefore, we may assume that $ \lambda_4 < 0 $ and so $ \lambda_5 < 0 $}.


\paragraph{Trace.} If $ \tr \left( M \right) = 0 $ and $ \lambda_1 > 0 $ then $ \sum_{i=1}^{5} \lambda_{i} = 0 < \frac{1}{2} \lambda_1 $, so the conditions of Theorem~\ref{th:main_nnc} are not met. {\em Therefore, we may also assume that $ \tr \left( M \right) > 0 $}.

\paragraph{Normalization.} If $ \lambda_1 > 0 $ and $ t = \sum_{i=1}^{5} \lambda_{i} \geq \frac{1}{2} \lambda_1 $, then $ M' = \frac{1}{2 t} M $ is well-defined. Obviously $ M' $ is a nonnegative symmetric $ 5 \times 5 $ matrix with eigenvalues $ \frac{1}{2 t} \lambda_1 \geq \frac{1}{2 t} \lambda_2 \geq \dots \geq \frac{1}{2 t} \lambda_5 $. We know that $ t \geq \frac{1}{2} \lambda_1 $ and so $ \rho \left( M' \right) = \frac{1}{2 t} \lambda_1 \leq 1 $. Also, $ \tr \left( M' \right) = \frac{1}{2 t} t = \frac{1}{2} $. {\em Hence, it suffices to prove Theorem~\ref{th:main_nnc} for matrices whose spectral radius is at most $1$ and whose trace is $ \frac{1}{2} $}.

\paragraph{Reducible matrices.} Assume for the moment that $M$ is reducible. We have, up to permutation similarity, the following cases:
\begin{enumerate}

\item $ M = A \bigoplus B $, where $A$ is a $ 1 \times 1 $ matrix and $B$ is a $ 4 \times 4 $ nonnegative symmetric matrix. \label{case:reducible_1}

Therefore, $ m_{11} $ is a nonnegative eigenvalue of $M$, and by assumption $ m_{11} \in \left\{ \lambda_1, \lambda_2, \lambda_3 \right\} $. We consider the different cases:
\begin{enumerate}

\item $ m_{11} = \lambda_1 $.

$B$ is a nonnegative symmetric matrix with eigenvalues $ \lambda_2 \geq \lambda_3 \geq \lambda_4 \geq \lambda_5 $, so $ \lambda_3 \leq \lambda_1 \leq \lambda_1 + \tr \left( B \right) = \sum_{i=1}^{5} \lambda_{i} $.

\item $ m_{11} = \lambda_2 $.

$B$ is a nonnegative symmetric matrix with eigenvalues $ \lambda_1 \geq \lambda_3 \geq \lambda_4 \geq \lambda_5 $, so $ \lambda_3 \leq \lambda_2 \leq \lambda_2 + \tr \left( B \right) = \sum_{i=1}^{5} \lambda_{i} $.

\item $ m_{11} = \lambda_3 $.

$B$ is a nonnegative symmetric matrix with eigenvalues $ \lambda_1 \geq \lambda_2 \geq \lambda_4 \geq \lambda_5 $ and so $ \lambda_3 \leq \lambda_3 + \tr \left( B \right) = \sum_{i=1}^{5} \lambda_{i} $.

\end{enumerate}

\item $ M = A \bigoplus B $, where $A$ is a $ 2 \times 2 $ irreducible nonnegative symmetric matrix and $B$ is a $ 3 \times 3 $ nonnegative symmetric matrix. \label{case:reducible_2}

Let $ \mu_1 \geq \mu_2 $ be the eigenvalues of $A$ and let $ \eta_1 \geq \eta_2 \geq \eta_3 $ be the eigenvalues of $B$. By the Perron-Frobenius Theorem $ \mu_1 $ and $ \eta_1 $ are nonnegative, so by assumption $ \mu_1, \eta_1 \in \left\{ \lambda_1, \lambda_2, \lambda_3 \right\} $. We consider the different cases:
\begin{enumerate}

\item $ \mu_1 = \lambda_1 $.

\begin{enumerate}

\item $ \eta_1 = \lambda_2 $. \label{case:reducible_eta_1_eq_lambda_2}

If $ \mu_2 = \lambda_3 $ then $ \lambda_3 = \mu_2 \leq \tr \left( A \right) \leq \tr \left( A \right) + \tr \left( B \right) = \sum_{i=1}^{5} \lambda_{i} $. Else, $ \eta_2 = \lambda_3 $. By the Perron-Frobenius Theorem $ \eta_1 + \eta_3 \geq 0 $. Therefore, $ \lambda_3 = \eta_2 \leq \tr \left( B \right) \leq \tr \left( A \right) + \tr \left( B \right) = \sum_{i=1}^{5} \lambda_{i} $.

\item $ \eta_1 = \lambda_3 $ (and so $ \mu_2 = \lambda_2 $).

We have $ \lambda_3 \leq \lambda_2 = \mu_2 \leq \tr \left( A \right) \leq \tr \left( A \right) + \tr \left( B \right) = \sum_{i=1}^{5} \lambda_{i} $.

\end{enumerate}

\item $ \mu_1 = \lambda_2 $ (and so $ \eta_1 = \lambda_1 $).

Same as case~\ref{case:reducible_eta_1_eq_lambda_2}.

\item $ \mu_1 = \lambda_3 $ (and so $ \eta_1 = \lambda_1 $ and $ \eta_2 = \lambda_2 $).

By the Perron-Frobenius Theorem $ \eta_1 + \eta_3 \geq 0 $. Therefore, we get $ \lambda_3 \leq \lambda_2 = \eta_2 \leq \tr \left( B \right) \leq \tr \left( A \right) + \tr \left( B \right) = \sum_{i=1}^{5} \lambda_{i} $.

\end{enumerate}

\item $ M = A \bigoplus B $, where $A$ is a $ 3 \times 3 $ irreducible nonnegative symmetric matrix and $B$ is a $ 2 \times 2 $ nonnegative symmetric matrix.

If $B$ is reducible then case~\ref{case:reducible_1} applies. Otherwise, case~\ref{case:reducible_2} applies.

\item $ M = A \bigoplus B $, where $A$ is a $ 4 \times 4 $ irreducible nonnegative symmetric matrix and $B$ is a $ 1 \times 1 $ matrix.

Case~\ref{case:reducible_1} applies.

\end{enumerate}
Therefore, Theorem~\ref{th:main_nnc} holds for any reducible $M$, and so {\em we can assume that $M$ is irreducible}.

\paragraph{Matrices orthogonally similar to a positive symmetric matrix.} Let $ \mathbb{M} $ be the set of nonnegative symmetric matrices meeting the conditions of Theorem~\ref{th:main_nnc} and having $ \lambda_1 \leq 1 $ and trace $ \frac{1}{2} $. Let $f$ be the function that maps each element of $ \mathbb{M} $ to its third largest eigenvalue. $ \mathbb{M} $ is a compact set and $f$ is a continuous function. Therefore, $f$ attains a maximum on $ \mathbb{M} $. Let $ M_0 \in \mathbb{M} $ be a matrix at which the maximum of $f$ is attained and let its spectrum be $ \sigma = \left( \lambda_1, \lambda_2, \lambda_3, \lambda_4, \lambda_5 \right) $. Suppose that $ M_0 $ is orthogonally similar to a positive symmetric matrix $ P_0 $. $ M_0 $ and $ P_0 $ have the same spectrum, so $ P_0 \in \mathbb{M} $ and $ f \left( M_0 \right) = f \left( P_0 \right) $. Hence, the maximum of $f$ is attained at $ P_0 $ as well. Let $Q$ be an orthogonal $ 5 \times 5 $ matrix such that $ Q^T P_0 Q = \diag \left( \sigma \right) $. By the Perron-Frobenius Theorem we know that $ \lambda_1 > \lambda_2 $ and $ \lambda_1 + \lambda_5 > 0 $. Let $ \varepsilon_0 = \frac{1}{2} \min \left\{ \lambda_1 - \lambda_2, \lambda_1 + \lambda_5 \right\} $, and for each $ 0 < \varepsilon \leq \varepsilon_0 $ let $ \sigma_\varepsilon = \left( \lambda_1, \lambda_2 + \varepsilon, \lambda_3 + \varepsilon, \lambda_4 - \varepsilon, \lambda_5 - \varepsilon \right) $. Furthermore, let $ D_\varepsilon = \diag \left( \sigma_\varepsilon \right) $. Then, for sufficiently small $ \varepsilon $, the matrix $ P_\varepsilon = Q D_\varepsilon Q^T $ is a positive symmetric matrix and $ P_\varepsilon \in \mathbb{M} $. However, $ f \left( P_\varepsilon \right) = \lambda_3 + \varepsilon > \lambda_3 = f \left( P_0 \right) $, which contradicts the assumption that $f$ attains its maximum value at $ P_0 $. We conclude that $ M_0 $ is not orthogonally similar to a positive symmetric matrix. Therefore, if Theorem~\ref{th:main_nnc} is valid for matrices not orthogonally similar to a positive symmetric matrix, it is also valid for matrices that are orthogonally similar to a positive symmetric matrix. {\em Hence, we may also assume that $M$ is not orthogonally similar to a positive symmetric matrix}.

\paragraph{Zero patterns.} The proof of Theorem~\ref{th:extreme_LM}, as given in~\cite{RefWorks:36}, considers the possible zero patterns extreme matrices with spectrum in $ \mathbb{U}_1 $ can have, and rules out any pattern that either meets the condition $ \lambda_3 \leq \sum_{i=1}^{5} \lambda_{i} $ (and so cannot have a spectrum in $ \mathbb{U}_1 $, in light of Remark~4 in~\cite{RefWorks:36}) or is orthogonally similar to a positive symmetric matrix (and so cannot be extreme).

{\em We shall follow the proof of Theorem~\ref{th:extreme_LM} but instead of considering extreme matrices whose spectrum is in $ \mathbb{U}_1 $ we consider irreducible, positive trace, nonnegative symmetric matrices not orthogonally similar to a positive symmetric matrix. Our goal is to find what zero patterns are left after ruling out patterns that either meet the condition $ \lambda_3 \leq \sum_{i=1}^{5} \lambda_{i} $ (and so Theorem~\ref{th:main_nnc} holds for such patterns), are reducible or are orthogonally similar to a positive symmetric matrix (and so contradict the assumptions)}.

Observation~2 of~\cite{RefWorks:36} lists five square matrices $ S_i $, $ i = 1, 2, 3, 4, 5 $, such that if any of them is a principal sub-matrix of $M$ then $ \lambda_3 \leq \sum_{i=1}^{5} \lambda_{i} $, and so Theorem~\ref{th:main_nnc} holds. {\em Hence, we may assume that none of these $ S_i $'s is a principal sub-matrix of $M$}.

Section~4 of~\cite{RefWorks:36} begins by observing necessary conditions for a potential extreme matrix $A$ with eigenvalues in $ \mathbb{U}_1 $. We check which of these conditions can be assumed in our case as well (calling the matrix $A$ and keeping the numbering as in~\cite{RefWorks:36}):
\begin{enumerate}[(i)]

\item $A$ has at least one positive main diagonal element.

By assumption $ \tr \left( A \right) > 0 $.

\item $A$ has two positive off-diagonal entries in every row. \label{cond:two_positive_in_every_row}

We repeat the discussion of this condition in~\cite{RefWorks:36} with some modifications.

By assumption, $A$ is irreducible, so it must have at least one nonzero off-diagonal element in every row. Suppose $A$ contains a row with only one nonzero element in an off-diagonal position. Then, without loss of generality, we can assume that the first row of $A$ looks like $ \left( a_{11}, 0, 0, 0, a_{15} \right) $, where $ a_{11} \geq 0 $ and $ a_{15} > 0 $. Let $ B = A[1,2,3,4] $. Then
\[ B = \left(
\begin{array}{ccccc}
a_{11} & 0 & 0 & 0 \\ \noalign{\medskip}
0 & * & * & * \\ \noalign{\medskip}
0 & * & * & * \\ \noalign{\medskip}
0 & * & * & *
\end{array}
\right) , \]
where $*$ indicates a zero or a positive element. Let $ \left( \mu_1, \mu_2, \mu_3, \mu_4 \right) $ be the spectrum of $B$, where $ \mu_1 \geq \mu_2 \geq \mu_3 \geq \mu_4 $. As $ a_{11} $ is an eigenvalue of $B$, there exists $ j \in \left\{ 1, 2, 3, 4 \right\} $ such that $ \mu_j = a_{11} $. By a previous assumption, $ \lambda_4 < 0 $. Therefore, by the interlacing property we have $ \mu_4 \leq \lambda_4 < 0 \leq a_{11} $, and so $ \mu_4 \neq a_{11} $.

If $ \mu_1 = a_{11} $ or $ \mu_2 = a_{11} $, then by the interlacing property we get $ \lambda_3 \leq \mu_2 \leq a_{11} \leq \tr \left( A \right) $, so Theorem~\ref{th:main_nnc} holds in these cases.

We are left with the case $ \mu_3 = a_{11} $. Let $ C = A[2,3,4] = B[2,3,4] $. The eigenvalues of $C$ are $ \mu_1 $, $ \mu_2 $, $ \mu_4 $ and we have $ \mu_1 \geq \mu_2 \geq \mu_3 = a_{11} \geq 0 > \mu_4 $. By the Perron-Frobenius
Theorem, applied to $C$, $ \mu_1 + \mu_4 \geq 0 $. Hence, $ \lambda_3 \leq \mu_2 \leq \tr \left( C \right) \leq a_{11} + \tr \left( C \right) = \tr \left( B \right) \leq \tr \left( A \right) $, so Theorem~\ref{th:main_nnc} holds in this case as well.

Therefore, we may assume that $A$ has two positive off-diagonal entries in every row.

\item $A$ has at least one zero on the main diagonal. \label{cond:zero_on_main_diagonal}

This condition can no longer be assumed!

\item $A$ has a zero in every row. \label{cond:zero_in_every_row}

By assumption, $A$ is not orthogonally similar to a positive symmetric matrix. The discussion of this condition in~\cite{RefWorks:36}, in particular the existence of $Y$, is valid in our case by use of Theorem~\ref{th:extreme_L}.

\item $Y$ (of Theorem~\ref{th:extreme_L}) has a nonzero in every row.

By assumption, $A$ is irreducible, so the discussion of this condition in~\cite{RefWorks:36} is valid in our case as well.

\end{enumerate}

Therefore, we need to check where condition~\eqref{cond:zero_on_main_diagonal} is used in Section~4 of~\cite{RefWorks:36} and find the implications of not meeting this condition. There are three occurrences of this condition. The first is in Case I, just before item (a), where the discussion leads to the pattern
\[ \left(
\begin{array}{ccccc}
+ & + & + & + & 0 \\ \noalign{\medskip}
+ & * & * & + & + \\ \noalign{\medskip}
+ & * & * & + & + \\ \noalign{\medskip}
+ & + & + & + & 0 \\ \noalign{\medskip}
0 & + & + & 0 & +
\end{array}
\right) , \]
where $+$ indicates a positive element and $*$ indicates a zero or a positive element. Because we can no longer assume condition~\eqref{cond:zero_on_main_diagonal}, the additional case that needs to be considered is when $ a_{22} > 0 $ and $ a_{33} > 0 $, which gives the pattern
\[ \left(
\begin{array}{ccccc}
+ & + & + & + & 0 \\ \noalign{\medskip}
+ & + & * & + & + \\ \noalign{\medskip}
+ & * & + & + & + \\ \noalign{\medskip}
+ & + & + & + & 0 \\ \noalign{\medskip}
0 & + & + & 0 & +
\end{array}
\right) . \]
If $ a_{22} > a_{33} $ then by a $ 2,3 \,\, OS $ (a kind of orthogonal similarity defined in~\cite{RefWorks:36}) we can make $ a_{23} > 0 $, and if $ a_{22} < a_{33} $ then by a $ 3,2 \,\, OS $ we can make $ a_{23} > 0 $. In both cases $A$ has a positive row, contradicting condition~\eqref{cond:zero_in_every_row}. Therefore, it suffices to check the case where $ a_{22} = a_{33} $ and $ a_{23} = 0 $. We have
\[ A = \left(
\begin{array}{ccccc}
a_{11} & a_{12} & a_{13} & a_{14} & 0 \\ \noalign{\medskip}
a_{12} & a_{22} & 0 & a_{24} & a_{25} \\ \noalign{\medskip}
a_{13} & 0 & a_{22} & a_{34} & a_{35} \\ \noalign{\medskip}
a_{14} & a_{24} & a_{34} & a_{44} & 0 \\ \noalign{\medskip}
0 & a_{25} & a_{35} & 0 & a_{55}
\end{array}
\right) ,
Y = \left(
\begin{array}{ccccc}
0 & 0 & 0 & 0 & y_{15} \\ \noalign{\medskip}
0 & 0 & y_{23} & 0 & 0 \\ \noalign{\medskip}
0 & y_{23} & 0 & 0 & 0 \\ \noalign{\medskip}
0 & 0 & 0 & 0 & y_{45} \\ \noalign{\medskip}
y_{15} & 0 & 0 & y_{45} & 0
\end{array}
\right) . \]
Letting $ A Y = \left( z_{ij} \right) $, $ Y A = \left( w_{ij} \right) $, we get by Theorem~\ref{th:extreme_L} the following equations:
\begin{align}
a_{13} y_{23} = z_{12} &= w_{12} = a_{25} y_{15} , \label{eq:z_12_eq_w_12_LM} \\
a_{12} y_{23} = z_{13} &= w_{13} = a_{35} y_{15} , \label{eq:z_13_eq_w_13_LM} \\
a_{11} y_{15} + a_{14} y_{45} = z_{15} &= w_{15} = a_{55} y_{15} , \label{eq:z_15_eq_w_15_LM} \\
a_{25} y_{45} = z_{24} &= w_{24} = a_{34} y_{23} , \label{eq:z_24_eq_w_24_LM} \\
a_{12} y_{15} + a_{24} y_{45} = z_{25} &= w_{25} = a_{35} y_{23} , \label{eq:z_25_eq_w_25_LM} \\
a_{35} y_{45} = z_{34} &= w_{34} = a_{24} y_{23} , \label{eq:z_34_eq_w_34_LM} \\
a_{14} y_{15} + a_{44} y_{45} = z_{45} &= w_{45} = a_{55} y_{45} . \label{eq:z_45_eq_w_45_LM}
\end{align}
Note that by \eqref{eq:z_12_eq_w_12_LM}, \eqref{eq:z_13_eq_w_13_LM} and \eqref{eq:z_24_eq_w_24_LM} either all the $ y_{ij} $'s, which are not known in advance to be zero, are zero or all of them are nonzero. As $Y$ is nonzero we conclude that they are all nonzero. Therefore, these $ y_{ij} $'s are positive and
\[ \alpha = \frac{ y_{15} }{ y_{23} }, \quad \beta = \frac{ y_{45} }{ y_{23} } \]
are well-defined and positive.

By \eqref{eq:z_12_eq_w_12_LM}, \eqref{eq:z_13_eq_w_13_LM}, \eqref{eq:z_24_eq_w_24_LM}, \eqref{eq:z_34_eq_w_34_LM} we get
\[ a_{13} = \alpha a_{25}, \quad a_{12} = \alpha a_{35}, \quad a_{34} = \beta a_{25}, \quad a_{24} = \beta a_{35} \]
respectively. By \eqref{eq:z_25_eq_w_25_LM} we get
\[ a_{35} = \alpha a_{12} + \beta a_{24} = \alpha^2 a_{35} + \beta^2 a_{35} \]
and therefore, $ \alpha^2 + \beta^2 = 1 $. By \eqref{eq:z_15_eq_w_15_LM} we get
\[ a_{55} = a_{11} + \frac{\beta}{\alpha} a_{14} , \]
and by \eqref{eq:z_45_eq_w_45_LM} we get
\[ a_{44} = a_{55} - \frac{\alpha}{\beta} a_{14} = a_{11} + \frac{ \beta^2 - \alpha^2 }{ \alpha \beta } a_{14} . \]
Therefore,
\[ A = \left(
\begin{array}{ccccc}
a_{11} & \alpha a_{35} & \alpha a_{25} & a_{14} & 0 \\ \noalign{\medskip}
\alpha a_{35} & a_{22} & 0 & \beta a_{35} & a_{25} \\ \noalign{\medskip}
\alpha a_{25} & 0 & a_{22} & \beta a_{25} & a_{35} \\ \noalign{\medskip}
a_{14} & \beta a_{35} & \beta a_{25} & a_{11} + \frac{ \beta^2 - \alpha^2 }{ \alpha \beta } a_{14} & 0 \\ \noalign{\medskip}
0 & a_{25} & a_{35} & 0 & a_{11} + \frac{\beta}{\alpha} a_{14}
\end{array}
\right) . \]
As by assumption, $A$ is not orthogonally similar to a positive symmetric matrix we further investigate this pattern. We perform a $ 2,3 \,\, OS $ on $A$. Let
\begin{align*}
X &= \left(
\begin{array}{cc}
c & -s \\ \noalign{\medskip}
s & c
\end{array}
\right)
\left(
\begin{array}{ccc}
\alpha a_{35} & \beta a_{35} & a_{25} \\ \noalign{\medskip}
\alpha a_{25} & \beta a_{25} & a_{35}
\end{array}
\right) \\
 &=
\left(
\begin{array}{ccc}
\alpha \left( c a_{35} - s a_{25} \right) & \beta \left( c a_{35} - s a_{25} \right) & c a_{25} - s a_{35} \\ \noalign{\medskip}
\alpha \left( c a_{25} + s a_{35} \right) & \beta \left( c a_{25} + s a_{35} \right) & c a_{35} + s a_{25}
\end{array}
\right) = \left( x_{ij} \right) ,
\end{align*}
where
\[ c = \cos \theta, \quad s = \sin \theta \]
for some yet to be chosen $ \theta $. If $ a_{35} = a_{25} $ then choosing $ \theta = \frac{\pi}{4} $ makes the first row of $X$ zero and the second row positive, and so the generated matrix that is orthogonally similar to $A$ has four zeros in its second row. Therefore, $A$ is reducible, contrary to our assumption.

If $ a_{35} < a_{25} $ then there exists $ 0 < \theta < \frac{\pi}{4} $ such that $ x_{11} = x_{12} = 0 $ and the remaining $ x_{ij} $'s are positive. Therefore, the generated matrix that is orthogonally similar to $A$ has the form
\[ A' = \left(
\begin{array}{ccccc}
a_{11} & 0 & + & a_{14} & 0 \\ \noalign{\medskip}
0 & a_{22} & 0 & 0 & + \\ \noalign{\medskip}
+ & 0 & a_{22} & + & + \\ \noalign{\medskip}
a_{14} & 0 & + & a_{11} + \frac{ \beta^2 - \alpha^2 }{ \alpha \beta } a_{14} & 0 \\ \noalign{\medskip}
0 & + & + & 0 & a_{11} + \frac{\beta}{\alpha} a_{14}
\end{array}
\right) . \]
As $A'$ is irreducible and has only a single positive off-diagonal element in the second row, this contradicts condition~\eqref{cond:two_positive_in_every_row}.

If $ a_{35} > a_{25} $ then there exists $ 0 < \theta < \frac{\pi}{4} $ such that $ x_{13} = 0 $ and the remaining $ x_{ij} $'s are positive. Therefore, the generated matrix that is orthogonally similar to $A$ has the form
\[ A'' = \left(
\begin{array}{ccccc}
a_{11} & + & + & a_{14} & 0 \\ \noalign{\medskip}
+ & a_{22} & 0 & + & 0 \\ \noalign{\medskip}
+ & 0 & a_{22} & + & + \\ \noalign{\medskip}
a_{14} & + & + & a_{11} + \frac{ \beta^2 - \alpha^2 }{ \alpha \beta } a_{14} & 0 \\ \noalign{\medskip}
0 & 0 & + & 0 & a_{11} + \frac{\beta}{\alpha} a_{14}
\end{array}
\right) . \]
As $A''$ is irreducible and has only a single positive off-diagonal element in the fifth row, this contradicts condition~\eqref{cond:two_positive_in_every_row}.

To conclude, the first occurrence of condition~\eqref{cond:zero_on_main_diagonal} contradicts our assumptions, so it does not affect the patterns found in the proof of Theorem~\ref{th:extreme_LM}.

The second occurrence of condition~\eqref{cond:zero_on_main_diagonal} is in Case I(b) in Section~4 of~\cite{RefWorks:36}. There, the pattern
\[ \left(
\begin{array}{ccccc}
+ & + & + & + & 0 \\ \noalign{\medskip}
+ & 0 & 0 & + & + \\ \noalign{\medskip}
+ & 0 & + & + & + \\ \noalign{\medskip}
+ & + & + & + & 0 \\ \noalign{\medskip}
0 & + & + & 0 & +
\end{array}
\right) \]
is considered. By doing a $ 3,2 \,\, OS $ we make $ a_{22} > 0 $ and $ a_{23} > 0 $, and so the matrix has a positive row, contradicting condition~\eqref{cond:zero_in_every_row}. As before, this does not affect the patterns found in the proof of Theorem~\ref{th:extreme_LM}.

The third occurrence of condition~\eqref{cond:zero_on_main_diagonal} is in Case IV in Section~4 of~\cite{RefWorks:36}. As we can no longer assume the main diagonal of $A$ has at least one zero element we get a slightly different pattern than the one called PIV in~\cite{RefWorks:36}. The new pattern is
\[ \left(
\begin{array}{ccccc}
+ & + & + & 0 & 0 \\ \noalign{\medskip}
+ & * & 0 & 0 & + \\ \noalign{\medskip}
+ & 0 & * & + & 0 \\ \noalign{\medskip}
0 & 0 & + & * & + \\ \noalign{\medskip}
0 & + & 0 & + & *
\end{array}
\right) , \]
and by orthogonal similarity with the permutation matrix $ I_3 \bigoplus \left(
\begin{array}{cc}
0 & 1 \\ \noalign{\medskip}
1 & 0
\end{array}
\right) $ we get the pattern
\[ C = \left(
\begin{array}{ccccc}
+ & + & + & 0 & 0 \\ \noalign{\medskip}
+ & * & 0 & + & 0 \\ \noalign{\medskip}
+ & 0 & * & 0 & + \\ \noalign{\medskip}
0 & + & 0 & * & + \\ \noalign{\medskip}
0 & 0 & + & + & *
\end{array}
\right) . \]
As pattern PIV is pattern $ C_0 $ in Theorem~\ref{th:extreme_LM}, we get that under our assumptions, it suffices to prove Theorem~\ref{th:main_nnc} for matrices of patterns $H$ and $C$, whose spectral radius is at most $1$ and whose trace is $ \frac{1}{2} $.

\section{Pattern $H$} \label{sec:pattern_H}

Let
\[ A = \left(
\begin{array}{ccccc}
\frac{1}{2} - t - s & a_{12} & a_{13} & 0 & 0 \\ \noalign{\medskip}
a_{12} & 0 & 0 & a_{24} & a_{25} \\ \noalign{\medskip}
a_{13} & 0 & t & 0 & a_{35} \\ \noalign{\medskip}
0 & a_{24} & 0 & s & a_{45} \\ \noalign{\medskip}
0 & a_{25} & a_{35} & a_{45} & 0
\end{array}
\right) , \]
where all the $ a_{ij} $'s, $s$ and $t$ are nonnegative and $ t + s \leq \frac{1}{2} $. Note that $A$ is in the closure of the set of matrices whose pattern is $H$, with a trace of $ \frac{1}{2} $. Let $ \left( \lambda_1, \lambda_2, \dots, \lambda_5 \right) $ be the spectrum of $A$, where $ \lambda_1 \geq \lambda_2 \geq \dots \geq \lambda_5 $.

\begin{lemma} \label{lem:symmetry_H}
\begin{enumerate}

\item \label{case:symmetry_H_1}
Let
\[ P = \left(
\begin{array}{ccccc}
0 & 0 & 1 & 0 & 0 \\ \noalign{\medskip}
0 & 0 & 0 & 0 & 1 \\ \noalign{\medskip}
1 & 0 & 0 & 0 & 0 \\ \noalign{\medskip}
0 & 0 & 0 & 1 & 0 \\ \noalign{\medskip}
0 & 1 & 0 & 0 & 0
\end{array}
\right) . \]
Then
\[ P A P^{-1} = \left(
\begin{array}{ccccc}
t & a_{35} & a_{13} & 0 & 0 \\ \noalign{\medskip}
a_{35} & 0 & 0 & a_{45} & a_{25} \\ \noalign{\medskip}
a_{13} & 0 & \frac{1}{2} - t - s & 0 & a_{12} \\ \noalign{\medskip}
0 & a_{45} & 0 & s & a_{24} \\ \noalign{\medskip}
0 & a_{25} & a_{12} & a_{24} & 0
\end{array}
\right) . \]

\item  \label{case:symmetry_H_2}
Suppose we are given functions $ \phi_1, \phi_2, \dots, \phi_p $ and $ \psi_1, \psi_2, \dots, \psi_q $ of the entries of $A$ such that $ \phi_i > 0 $ (or $ \phi_i \geq 0 $), $ i = 1, 2, \dots, p $ imply $ \psi_j > 0 $ (or $ \psi_j \geq 0 $), $ j = 1, 2, \dots, q $. Then the same holds true if we replace the occurrence of the entries of $A$ in the $ \phi_i $'s and $ \psi_j $'s as follows: $ a_{11} = \frac{1}{2} - t - s $ by $ a_{33} = t $, $ a_{12} $ by $ a_{35} $, $ a_{24} $ by $ a_{45} $, $ a_{33} = t $ by $ a_{11} = \frac{1}{2} - t - s $, $ a_{35} $ by $ a_{12} $, $ a_{45} $ by $ a_{24} $, respectively, keeping $ a_{13} $, $ a_{44} = s $ and $ a_{25} $ unchanged.

\end{enumerate}
\end{lemma}

\begin{proof}
\eqref{case:symmetry_H_1} is obvious and \eqref{case:symmetry_H_2} follows from \eqref{case:symmetry_H_1}.
\end{proof}
The use of Lemma~\ref{lem:symmetry_H} will become clear in the corollaries of this section.

For convenience of notation we name the following expressions:
\begin{align*}
S_1 &= P_{A[2,4,5]} \left( 1 \right) = 1 - s - {a_{24}}^2 - \left( 1 - s \right) {a_{25}}^2 - {a_{45}}^2 - 2 a_{24} a_{25} a_{45} , \\
S_2 &= -8 \cdot P_{A[2,4,5]} \left( \frac{1}{2} \right) \\
 &= -1 + 2 s + 4 {a_{24}}^2 + 4 \left( 1 - 2 s \right) {a_{25}}^2 + 4 {a_{45}}^2 + 16 a_{24} a_{25} a_{45} .
\end{align*}
Note that as $ s \leq \frac{1}{2} $ and all the $ a_{ij} $'s are nonnegative $ S_1 $ is monotonically decreasing in $ a_{25} $ and $ S_2 $ is monotonically increasing in $ a_{25} $. It can be checked that
\begin{align}
P_{A[1,2,4,5]} \left( 1 \right) &= \frac{1}{2} \left( 1 + 2 t + 2 s \right) S_1 + {a_{12}}^2 \left( {a_{45}}^2 + s - 1 \right) , \label{eq:P_A1245_lambda_1} \\
P_{A[1,2,4,5]} \left( \frac{1}{2} \right) &= -\frac{1}{8} \left( t + s \right) S_2 + \frac{1}{4} {a_{12}}^2 \left( 4 {a_{45}}^2 + 2 s - 1 \right) . \label{eq:P_A1245_lambda_1_2}
\end{align}

\begin{lemma} \label{lem:necessary_conditions_H}
If $ \rho \left( A \right) \leq 1 $ then
\begin{equation}
S_1 \geq 0 . \label{eq:S_1_geq_0}
\end{equation}
\end{lemma}

\begin{proof}
The eigenvalues of $ A[4,5] $ are $ \frac{1}{2} s \pm \frac{1}{2} \sqrt{ s^2 + 4 {a_{45}}^2 } $. By Observation~\ref{obs:spectral_radius} we have $ \rho \left( A[4,5] \right) \leq 1 $, so $ \frac{1}{2} s + \frac{1}{2} \sqrt{ s^2 + 4 {a_{45}}^2 } \leq 1 $. Therefore,
\begin{equation}
1 - s - {a_{45}}^2 \geq 0 . \label{eq:a_45_and_s_nonnegative}
\end{equation}
Assume $ S_1 < 0 $. By \eqref{eq:P_A1245_lambda_1} and \eqref{eq:a_45_and_s_nonnegative} we get $ P_{A[1,2,4,5]} \left( 1 \right) < 0 $. By Observation~\ref{obs:negative_cp}, $ \rho \left( A[1,2,4,5] \right) > 1 $, which contradicts Observation~\ref{obs:spectral_radius}. Therefore, $ S_1 \geq 0 $ and so \eqref{eq:S_1_geq_0} holds.
\end{proof}

\begin{lemma} \label{lem:sufficient_conditions_H}
If any of the following conditions hold then $ \lambda_3 \leq \frac{1}{2} $.
\begin{enumerate}
\item $ a_{24} \leq \frac{1}{2} \sqrt{ 1 - 2 s } $,
\item $ a_{45} \leq \frac{1}{2} \sqrt{ 1 - 2 s } $,
\item $ a_{13} \leq \frac{1}{2} \sqrt{ 2 \left( 1 - 2 t \right) \left( t + s \right) } $.
\end{enumerate}
\end{lemma}

\begin{proof}
We check each condition.
\begin{enumerate}

\item The eigenvalues of $ A[2,3,4] $ are $ t, \frac{1}{2} s \pm \frac{1}{2} \sqrt{ s^2 + 4 {a_{24}}^2 } $. We know that $ t \leq \frac{1}{2} $. If $ a_{24} \leq \frac{1}{2} \sqrt{ 1 - 2 s } $ then $ \frac{1}{2} s + \frac{1}{2} \sqrt{ s^2 + 4 {a_{24}}^2 } \leq \frac{1}{2} $, and so $ \rho \left( A[2,3,4] \right) \leq \frac{1}{2} $. By Observation~\ref{obs:3x3_sufficient_condition} we have $ \lambda_3 \leq \frac{1}{2} $.

\item Follows from the previous case by application of Lemma~\ref{lem:symmetry_H}.

\item The eigenvalues of $ A[1,3,4] $ are $ s, \frac{1}{4} - \frac{1}{2} s 
\pm \frac{1}{4} \sqrt{ \left( 1 - 2 s - 4 t \right)^2 + 16 {a_{13}}^2 } $ and we know that $ s \leq \frac{1}{2} $. If $ a_{13} \leq \frac{1}{2} \sqrt{ 2 \left( 1 - 2 t \right) \left( t + s \right) } $ then
\[ \frac{1}{4} - \frac{1}{2} s + \frac{1}{4} \sqrt{ \left( 1 - 2 s - 4 t \right)^2 + 16 {a_{13}}^2 } \leq \frac{1}{2} , \]
and so $ \rho \left( A[1,3,4] \right) \leq \frac{1}{2} $. By Observation~\ref{obs:3x3_sufficient_condition} we have $ \lambda_3 \leq \frac{1}{2} $.

\end{enumerate}
\end{proof}

\begin{lemma} \label{lem:more_necessary_conditions_H}
If $ \rho \left( A \right) \leq 1 $, $ a_{24} > 0 $ and $ a_{45} > 0 $ then
\begin{equation}
a_{12} \leq \frac{1}{2} \sqrt{ \frac{ 2 \left( 1 + 2 t + 2 s \right) S_1 }{ 1 - s - {a_{45}}^2 } } \label{eq:a_12_upper_bound}
\end{equation}
and
\begin{equation}
a_{35} \leq \sqrt{ \frac{ \left( 1 - t \right) S_1 }{ 1 - s - {a_{24}}^2 } } . \label{eq:a_35_upper_bound}
\end{equation}
Moreover, if
\begin{equation}
a_{24} > \frac{1}{2} \sqrt{ 1 - 2 s } \label{eq:a_24_lower_bound}
\end{equation}
and
\begin{equation}
a_{45} > \frac{1}{2} \sqrt{ 1 - 2 s } \label{eq:a_45_lower_bound}
\end{equation}
hold, then
\begin{equation}
a_{24} < \frac{1}{2} \sqrt{ 3 - 2 s } , \label{eq:a_24_upper_bound}
\end{equation}
\begin{equation}
a_{45} < \frac{1}{2} \sqrt{ 3 - 2 s } , \label{eq:a_45_upper_bound}
\end{equation}
\begin{equation}
a_{25} < \frac{1}{ 2 \left( 1 - s \right) } , \label{eq:a_25_upper_bound}
\end{equation}
and
\begin{equation}
S_2 > 0 . \label{eq:S_2_positive}
\end{equation}
\end{lemma}

\begin{proof}
As $ s \leq \frac{1}{2} $ and all the $ a_{ij} $'s are nonnegative, $ {a_{24}}^2 + {a_{45}}^2 \leq 1 - s $ follows immediately from \eqref{eq:S_1_geq_0}. Together with \eqref{eq:a_24_lower_bound} we get \eqref{eq:a_45_upper_bound}, and together with \eqref{eq:a_45_lower_bound} we get \eqref{eq:a_24_upper_bound}. Also, as $ a_{45} > 0 $, we have
\begin{equation}
1 - s - {a_{24}}^2 \geq {a_{45}}^2 > 0 . \label{eq:a_24_and_s_positive}
\end{equation}
Similarly, as $ a_{24} > 0 $, we have
\begin{equation}
1 - s - {a_{45}}^2 \geq {a_{24}}^2 > 0 . \label{eq:a_45_and_s_positive}
\end{equation}
By Observation~\ref{obs:spectral_radius}, $ \rho \left( A[1,2,4,5] \right) \leq 1 $ and by Observation~\ref{obs:negative_cp}, $ P_{A[1,2,4,5]} \left( 1 \right) \geq 0 $. By \eqref{eq:P_A1245_lambda_1} we get \eqref{eq:a_12_upper_bound} and by Lemma~\ref{lem:symmetry_H} we get \eqref{eq:a_35_upper_bound}. Note that by \eqref{eq:S_1_geq_0}, \eqref{eq:a_24_and_s_positive} and \eqref{eq:a_45_and_s_positive} the values inside the square roots of \eqref{eq:a_12_upper_bound} and \eqref{eq:a_35_upper_bound} are nonnegative.

By Observation~\ref{obs:spectral_radius}, $ \rho \left( A[1,2,4,5] \right) \leq 1 $. If $ a_{25} \geq \frac{1}{ 2 \left( 1 - s \right) } $ then by \eqref{eq:a_24_lower_bound} and \eqref{eq:a_45_lower_bound}
\begin{align*}
 S_1 &= 1 - s - {a_{24}}^2 - \left( 1 - s \right) {a_{25}}^2 - {a_{45}}^2 - 2 a_{24} a_{25} a_{45} \\
 &< 1 - s - \left( 1 - s \right) {a_{25}}^2 - 2 \left( \frac{1}{2} \sqrt{ 1 - 2 s } \right)^2 - 2 a_{25} \left( \frac{1}{2} \sqrt{ 1 - 2 s } \right)^2 \\
 &= \frac{1}{2} \left( 1 + a_{25} \right) \left( 1 - 2 \left( 1 - s \right) a_{25} \right) \\
 &< \frac{1}{2} \left( 1 + a_{25} \right) \left( 1 - 2 \left( 1 - s \right) \frac{1}{ 2 \left( 1 - s \right) } \right) = 0 .
\end{align*}
This contradicts \eqref{eq:S_1_geq_0} and so \eqref{eq:a_25_upper_bound} holds.

By \eqref{eq:a_24_lower_bound}, \eqref{eq:a_45_lower_bound} and as $ s \leq \frac{1}{2} $ we have
\begin{align*}
 S_2 &= -1 + 2 s + 4 {a_{24}}^2 + 4 \left( 1 - 2 s \right) {a_{25}}^2 + 4 {a_{45}}^2 + 16 a_{24} a_{25} a_{45} \\
 &> -1 + 2 s + 8 \left( \frac{1}{2} \sqrt{ 1 - 2 s } \right)^2 + 4 \left( 1 - 2 s \right) {a_{25}}^2 + 16 \left( \frac{1}{2} \sqrt{ 1 - 2 s } \right)^2 a_{25} \\
 &= { \left( 1 + 2 a_{25} \right)}^2 \left( 1 - 2 s \right) \geq 0
\end{align*}
and so \eqref{eq:S_2_positive} holds.
\end{proof}

\begin{lemma} \label{lem:a_12_sufficient_condition}
If $ \rho \left( A \right) \leq 1 $, \eqref{eq:a_24_lower_bound}, \eqref{eq:a_45_lower_bound} and 
\[ a_{12} \leq \frac{1}{2} \sqrt{ \frac{2 \left( t + s \right) S_2 }{ 4 {a_{45}}^2 + 2 s - 1 } } \]
hold, then $ \lambda_3 \leq \frac{1}{2} $.
\end{lemma}

\begin{proof}
Let $ \left( \mu_1, \mu_2, \mu_3, \mu_4 \right) $ be the spectrum of $ A[1,2,4,5] $, where $ \mu_1 \geq \mu_2 \geq \mu_3 \geq \mu_4 $. The eigenvalues of $ A[1,4,5] $ are $ \frac{1}{2} - t - s, \frac{1}{2} s \pm \frac{1}{2} \sqrt{ s^2 + 4 {a_{45}}^2 } $. By \eqref{eq:a_45_lower_bound} we have $ \frac{1}{2} s + \frac{1}{2} \sqrt{ s^2 + 4 {a_{45}}^2 } > \frac{1}{2} \geq \frac{1}{2} - t - s $. Hence, by the eigenvalue interlacing property we have that
\[ \mu_4 \leq \frac{1}{2} s - \frac{1}{2} \sqrt{ s^2 + 4 {a_{45}}^2 } \leq \mu_3 \leq \frac{1}{2} - t - s \leq \mu_2 \leq \frac{1}{2} s + \frac{1}{2} \sqrt{ s^2 + 4 {a_{45}}^2 } \leq \mu_1 . \]
In particular, $ \mu_1 > \frac{1}{2} $. Note that $ \mu_3 < \frac{1}{2} $, since otherwise $ s = t = 0 $ and then
\[ \frac{1}{2} = \frac{1}{2} - t = \tr \left( A[1,2,4,5] \right) = \mu_1 + \mu_2 + \mu_3 + \mu_4 \geq \mu_2 + \mu_3 \geq 2 \mu_3 = 1 . \]

By Lemma~\ref{lem:more_necessary_conditions_H} we have \eqref{eq:S_2_positive}. Assume that $ a_{12} \leq \frac{1}{2} \sqrt{ \frac{2 \left( t + s \right) S_2 }{ 4 {a_{45}}^2 + 2 s - 1 } } $ and note that the value inside the square root is nonnegative by \eqref{eq:S_2_positive} and \eqref{eq:a_45_lower_bound}. By \eqref{eq:P_A1245_lambda_1_2}, $ P_{A[1,2,4,5]} \left( \frac{1}{2} \right) \leq 0 $. If $ P_{A[1,2,4,5]} \left( \frac{1}{2} \right) = 0 $ then, as $ \mu_3 < \frac{1}{2} < \mu_1 $, we have $ \mu_2 = \frac{1}{2} $. If $ P_{A[1,2,4,5]} \left( \frac{1}{2} \right) < 0 $ then by Observation~\ref{obs:negative_cp} we have $ \mu_2 < \frac{1}{2} $. In both cases, by Observation~\ref{obs:4x4_sufficient_condition} we have $ \lambda_3 \leq \frac{1}{2} $.
\end{proof}

\begin{corollary} \label{cor:a_35_sufficient_condition}
If $ \rho \left( A \right) \leq 1 $, \eqref{eq:a_24_lower_bound}, \eqref{eq:a_45_lower_bound} and 
\[ a_{35} \leq \frac{1}{2} \sqrt{ \frac{ \left( 1 - 2 t \right) S_2 }{ 4 {a_{24}}^2 + 2 s - 1 } } \]
hold, then $ \lambda_3 \leq \frac{1}{2} $.
\end{corollary}

\begin{proof}
Direct application of Lemma~\ref{lem:symmetry_H}.
\end{proof}

\begin{lemma} \label{lem:a_45_better_lower_bound}
If $ \rho \left( A \right) \leq 1 $, \eqref{eq:a_24_lower_bound}, \eqref{eq:a_45_lower_bound} and
\begin{equation}
a_{12} > \frac{1}{2} \sqrt{ \frac{2 \left( t + s \right) S_2 }{ 4 {a_{45}}^2 + 2 s - 1 } } \label{eq:a_12_lower_bound}
\end{equation}
hold, then
\begin{equation}
a_{24} < \frac{1}{2} \sqrt{ \frac{ \left( 3 - 2 s \right) \left( 1 + 3 t + 3 s - 2 \sqrt{ \left( 1 + 2 t + 2 s \right) \left( t + s \right) } \right) }{ 1 + t + s } } \label{eq:a_24_better_upper_bound}
\end{equation}
and
\begin{equation}
a_{45} > \sqrt{ \frac{1}{2} \left( 1 - s \right) - \frac{1}{4} \sqrt{ 1 - \frac{ \left( 1 - 2 s \right) \left( 3 - 2 s \right) \left( t + s \right) }{ 1 + t + s } } } . \label{eq:a_45_better_lower_bound}
\end{equation}
\end{lemma}

\begin{proof}
Assume \eqref{eq:a_12_lower_bound}. By Lemma~\ref{lem:more_necessary_conditions_H} we have \eqref{eq:a_12_upper_bound}. Also, $ S_1 $ is monotonically decreasing in $ a_{25} $ and $ S_2 $ is monotonically increasing in $ a_{25} $. Therefore,
\begin{align*}
& \frac{1}{2} \sqrt{ \frac{2 \left( t + s \right) \left( -1 + 2 s + 4 {a_{24}}^2 + 4 {a_{45}}^2 \right) }{ 4 {a_{45}}^2 + 2 s - 1 } } \leq \frac{1}{2} \sqrt{ \frac{2 \left( t + s \right) S_2 }{ 4 {a_{45}}^2 + 2 s - 1 } } < a_{12} \\
 & \quad \leq \frac{1}{2} \sqrt{ \frac{ 2 \left( 1 + 2 t + 2 s \right) S_1 }{ 1 - s - {a_{45}}^2 } } \leq \frac{1}{2} \sqrt{ \frac{ 2 \left( 1 + 2 t + 2 s \right) \left( 1 - s - {a_{24}}^2 - {a_{45}}^2 \right) }{ 1 - s - {a_{45}}^2 } } .
\end{align*}
Equivalently,
\begin{equation}
f_1 \left( {a_{45}}^2, {a_{24}}^2 \right) < 0 , \label{eq:f_1_at_a_45_sq}
\end{equation}
where
\begin{align*}
f_1 &: \mathbb{R} \times D \to \mathbb{R} , \quad D = \left[ \frac{1}{4} - \frac{1}{2} s, \frac{3}{4} - \frac{1}{2} s \right] , \\
f_1 \left( x, y \right) &= 4 C_1 x^2 + C_2 \left( y \right) x + C_3 \left( y \right) , \\
C_1 &= 1 + t + s , \\
C_2 \left( y \right) &= \left( 1 + t + s \right) \left( 4 y + 6 s - 5 \right) , \\
C_3 \left( y \right) &= \left( -1 + 2 t + 4 s \right) y + \left( 1 - s \right) \left( 1 - 2 s \right) \left( 1 + t + s \right) .
\end{align*}
By Lemma~\ref{lem:more_necessary_conditions_H} we have \eqref{eq:a_24_upper_bound} and \eqref{eq:a_45_upper_bound}. Note that $ {a_{45}}^2, {a_{24}}^2 \in D $ by \eqref{eq:a_24_lower_bound}, \eqref{eq:a_45_lower_bound}, \eqref{eq:a_24_upper_bound} and \eqref{eq:a_45_upper_bound}, and that $ 0 \in D $ only when $ s = \frac{1}{2} $.
$ C_1 > 0 $ since $ t \geq 0 $ and $ s \geq 0 $. As $ s \leq \frac{1}{2} $ we have
\[ 4 y + 6 s - 5 \leq 4 \left( \frac{3}{4} - \frac{1}{2} s \right) + 6 s - 5 = 4 s - 2 \leq 0 . \]
Therefore, $ C_2 \left( y \right) \leq 0 $ and in particular $ C_2 \left( {a_{24}}^2 \right) < 0 $.
If $ -1 + 2 t + 4 s \geq 0 $ then as $ t \geq 0 $ and $ 0 \leq s \leq \frac{1}{2} $,
\begin{align*}
C_3 \left( y \right) & \geq \left( -1 + 2 t + 4 s \right) \left( \frac{1}{4} - \frac{1}{2} s \right) + \left( 1 - s \right) \left( 1 - 2 s \right) \left( 1 + t + s \right) \\
 &= \frac{1}{4} \left( 1 - 2 s \right) \left( 3 - 2 s \right) \left( 1 + 2 t + 2 s \right) \geq 0 .
\end{align*}
Note that in particular $ C_3 \left( {a_{24}}^2 \right) > 0 $ in this case. Indeed, if $ s = \frac{1}{2} $ then $ t = 0 $ and so by \eqref{eq:a_24_lower_bound} the left term of $ C_3 $ is positive, and when $ s < \frac{1}{2} $ the right term of $ C_3 $ is positive. If $ -1 + 2 t + 4 s < 0 $ then as $ t \geq 0 $ and $ 0 \leq s \leq \frac{1}{2} $,
\begin{align*}
C_3 \left( y \right) & \geq \left( -1 + 2 t + 4 s \right) \left( \frac{3}{4} - \frac{1}{2} s \right) + \left( 1 - s \right) \left( 1 - 2 s \right) \left( 1 + t + s \right) \\
 &= \frac{3}{2} s \left( 1 - 2 s \right) + \frac{1}{2} t \left( 5 - 8 s \right) + 2 s^3 + 2 t s^2 + \frac{1}{4} \geq \frac{1}{4} .
\end{align*}
Therefore, $ C_3 \left( y \right) \geq 0 $ and in particular $ C_3 \left( {a_{24}}^2 \right) > 0 $.

Given $ y \in D $, the roots of $ f_1 \left( x, y \right) $ are
\begin{align*}
r_1^{\pm} \left( y \right) &= -\frac{ C_2 \left( y \right) }{ 8 C_1 } \pm \frac{ \sqrt{ \left( C_2 \left( y \right) \right)^2 - 16 C_1 C_3 \left( y \right) } }{ 8 C_1 } \\
 &= -\frac{1}{8} \left( 4 y + 6 s - 5 \right) \pm \frac{1}{8} \sqrt{ \frac{ g_1 \left( y \right) }{ 1 + t + s } } ,
\end{align*}
where
\begin{equation*}
g_1 \left( y \right) = 16 C_1 y^2 - 8 \left( 3 - 2 s \right) \left( 1 + 3 t + 3 s \right) y + \left( 3 - 2 s \right)^2 C_1 .
\end{equation*}
The roots of $ g_1 \left( y \right) $ are
\[ q_1^{\pm} = q_1^{\pm} \left( s, t \right) = \frac{ \left( 3 - 2 s \right) \left( 1 + 3 t + 3 s \pm 2 \sqrt{ \left( 1 + 2 t + 2 s \right) \left( t + s \right) } \right) }{ 4 \left( 1 + t + s \right) } \]
and therefore real. Also, $ q_1^+ \geq q_1^- $. Moreover, $ q_1^+ > 0 $ and $ q_1^+ q_1^- = \frac{ \left( 3 - 2 s \right)^2 }{16} > 0 $, so $ q_1^- > 0 $. Furthermore, $ \left( q_1^+ \right)^2 \geq q_1^+ q_1^- $ and so $ q_1^+ \geq \frac{ \left( 3 - 2 s \right) }{4} \geq y $ for any $ y \in D $. As $ C_1 > 0 $, $ g_1 \left( y \right) \leq 0 $ if and only if $ q_1^- \leq y \leq q_1^+ $. By \eqref{eq:f_1_at_a_45_sq} we know that $ f_1 \left( x, {a_{24}}^2 \right) $ has two distinct real roots and so $ g_1 \left( {a_{24}}^2 \right) > 0 $. As $ {a_{24}}^2 \leq q_1^+ $ we conclude that $ a_{24} < \sqrt{ q_1^- } $, which is exactly \eqref{eq:a_24_better_upper_bound}.

Note that
\begin{align*}
& \frac{ \partial }{ \partial y } f_1 \left( x, y \right) = 4 \left( 1 + t + s \right) x - 1 + 2 t + 4 s , \\
& \left. \frac{ \partial }{ \partial y } f_1 \left( x, y \right) \right|_{ x = \frac{1}{4} - \frac{1}{2} s } = \left( 3 - 2 s \right) \left( t + s \right) \geq 0 ,
\end{align*}
so $ f_1 \left( x, y \right) $ is monotonically increasing in $y$ when $ x \geq \frac{1}{4} - \frac{1}{2} s $.

From now on assume that $ y \leq {a_{24}}^2 $. Therefore, $ y < q_1^- $ and so $ g_1 \left( y \right) > 0 $. Hence, $ r_1^{\pm} \left( y \right) $ are distinct real numbers. As $ C_1 > 0 $, $ C_2 \left( y \right) \leq 0 $ and $ C_3 \left( y \right) \geq 0 $, we have
\[ r_1^+ \left( y \right) > -\frac{1}{8} \left( 4 y + 6 s - 5 \right) > r_1^- \left( y \right) \geq -\frac{ C_2 \left( y \right) }{ 8 C_1 } - \frac{ \sqrt{ \left( C_2 \left( y \right) \right)^2 } }{ 8 C_1 } = 0 \]
and $ f_1 \left( x, y \right) < 0 $ if and only if $ r_1^- \left( y \right) < x < r_1^+ \left( y \right) $. In particular, as $ r_1^{\pm} \left( {a_{24}}^2 \right) $ are real, $ r_1^- \left( {a_{24}}^2 \right) \geq 0 $ and $ f_1 \left( {a_{45}}^2, {a_{24}}^2 \right) < 0 $, we must have $ a_{45} > \sqrt{ r_1^- \left( {a_{24}}^2 \right) } $.

By $ y \leq {a_{24}}^2 \leq \frac{3}{4} - \frac{1}{2} s $ we get $ \frac{1}{4} - \frac{1}{2} s \leq -\frac{1}{8} \left( 4 y + 6 s - 5 \right) < r_1^+ \left( y \right) $. As
\[ f_1 \left( \frac{1}{4} - \frac{1}{2} s, y \right) = \left( 3 - 2 s \right) \left( t + s \right) y \geq 0 \]
we must have $ r_1^- \left( y \right) \geq \frac{1}{4} - \frac{1}{2} s $.

Let $ \frac{1}{4} - \frac{1}{2} s \leq y_1 \leq y_2 \leq {a_{24}}^2 $. As $ r_1^- \left( y_2 \right) \geq \frac{1}{4} - \frac{1}{2} s $ we have
\[ f_1 \left( r_1^- \left( y_2 \right), y_1 \right) \leq f_1 \left( r_1^- \left( y_2 \right), y_2 \right) = 0 \]
and so, $ r_1^- \left( y_2 \right) \geq r_1^- \left( y_1 \right) $. Therefore, $ r_1^- \left( y \right) $ is monotonically increasing in $y$ when $ \frac{1}{4} - \frac{1}{2} s \leq y \leq {a_{24}}^2 $. We conclude that
\begin{align*}
a_{45} &> \sqrt{ r_1^- \left( {a_{24}}^2 \right) } \geq \sqrt{ r_1^- \left( \frac{1}{4} - \frac{1}{2} s \right) } \\
 &= \sqrt{ \frac{1}{2} \left( 1 - s \right) - \frac{1}{4} \sqrt{ 1 - \frac{ \left( 1 - 2 s \right) \left( 3 - 2 s \right) \left( t + s \right) }{ 1 + t + s } } } ,
\end{align*}
which is exactly \eqref{eq:a_45_better_lower_bound}. Note that we already found that $ r_1^- \left( \frac{1}{4} - \frac{1}{2} s \right) \geq 0 $ so the value inside the inner square root is nonnegative.
\end{proof}

\begin{corollary} \label{cor:a_24_better_lower_bound}
If $ \rho \left( A \right) \leq 1 $, \eqref{eq:a_24_lower_bound}, \eqref{eq:a_45_lower_bound} and
\begin{equation}
a_{35} > \frac{1}{2} \sqrt{ \frac{ \left( 1 - 2 t \right) S_2 }{ 4 {a_{24}}^2 + 2 s - 1 } } \label{eq:a_35_lower_bound}
\end{equation}
hold, then
\begin{equation}
a_{45} < \frac{1}{2} \sqrt{ \frac{ \left( 3 - 2 s \right) \left( 5 - 6 t - 4 \sqrt{ \left( 1 - t \right) \left( 1 - 2 t \right) } \right) }{ 3 - 2 t } } \label{eq:a_45_better_upper_bound}
\end{equation}
and
\begin{equation}
a_{24} > \sqrt{ \frac{1}{2} \left( 1 - s \right) - \frac{1}{4} \sqrt{ 1 - \frac{ \left( 1 - 2 s \right) \left( 3 - 2 s \right) \left( 1 - 2 t \right) }{ 3 - 2 t } } } . \label{eq:a_24_better_lower_bound}
\end{equation}
\end{corollary}

\begin{proof}
Direct application of Lemma~\ref{lem:symmetry_H}.
\end{proof}

\begin{lemma} \label{lem:spectral_radius_of_A[1,2,3,5]}
If $ t \leq \frac{1}{2} - t - s $, $ t + s > 0 $, \eqref{eq:S_1_geq_0}, \eqref{eq:a_12_upper_bound}, \eqref{eq:a_35_upper_bound}, \eqref{eq:a_24_lower_bound}, \eqref{eq:a_45_lower_bound}, \eqref{eq:a_25_upper_bound}, \eqref{eq:a_12_lower_bound}, \eqref{eq:a_45_better_lower_bound}, \eqref{eq:a_35_lower_bound}, \eqref{eq:a_45_better_upper_bound} and
\begin{equation}
a_{13} > \frac{1}{2} \sqrt{ 2 \left( 1 - 2 t \right) \left( t + s \right) } \label{eq:a_13_lower_bound}
\end{equation}
hold, then $ \rho \left( A[1,2,3,5] \right) > 1 $.
\end{lemma}

\begin{proof}
It can be checked that
\begin{equation*}
P_{A[1,2,3,5]} \left( 1 \right) = f_2 \left( a_{12}, a_{13}, a_{25}, a_{35} \right) ,
\end{equation*}
where
\begin{align*}
f_2 \left( x, y, z, w \right) =& \left( -1 + z^2 \right) y^2 - 2 x z w y - \frac{1}{2} \left( 1 - t \right) \left( 1 + 2 t + 2 s \right) z^2 \\
 &- \frac{1}{2} \left( t - 1 + w^2 \right) \left( 1 + 2 t + 2 s -2 x^2 \right) .
\end{align*}
By \eqref{eq:a_25_upper_bound} the coefficient of $ {a_{13}}^2 $ in $ f_2 \left( a_{12}, a_{13}, a_{25}, a_{35} \right) $ is negative and as the $ a_{ij} $'s are nonnegative the coefficient of $ a_{13} $ in $ f_2 \left( a_{12}, a_{13}, a_{25}, a_{35} \right) $ is nonpositive. Let \[ \check{a}_{13} = \frac{1}{2} \sqrt{ 2 \left( 1 - 2 t \right) \left( t + s \right) } . \] By \eqref{eq:a_13_lower_bound} we have
\begin{align*}
P_{A[1,2,3,5]} \left( 1 \right) <& f_2 \left( a_{12}, \check{a}_{13}, a_{25}, a_{35} \right) \\
 =& - \frac{1}{2} \left( 1 + s \right) {a_{25}}^2 - a_{12} a_{35} \sqrt{ 2 \left( 1 - 2 t \right) \left( t + s \right) } a_{25} \\
 &+ \left( t - 1 + {a_{35}}^2 \right) {a_{12}}^2 + \frac{1}{2} \left( 1 + s \right) - \frac{1}{2} \left( 1 + 2 t + 2 s \right) {a_{35}}^2 .
\end{align*}
Clearly, the coefficients of the powers of $ a_{25} $ in $ f_2 \left( a_{12}, \check{a}_{13}, a_{25}, a_{35} \right) $ are nonpositive. Therefore, we have
\begin{align*}
f_2 \left( a_{12}, \check{a}_{13}, a_{25}, a_{35} \right) \leq& f_2 \left( a_{12}, \check{a}_{13}, 0, a_{35} \right) \\
 =& \left( t - 1 + {a_{35}}^2 \right) {a_{12}}^2 - \frac{1}{2} \left( 1 + 2 t + 2 s \right) {a_{35}}^2 + \frac{1}{2} \left( 1 + s \right) \\
 =& \frac{1}{2} \left( 2 {a_{12}}^2 - 1 - 2 t - 2 s \right) {a_{35}}^2 + \left( t - 1 \right) {a_{12}}^2 + \frac{1}{2} \left( 1 + s \right) .
\end{align*}
By \eqref{eq:a_35_upper_bound}, \eqref{eq:S_1_geq_0} and by the definition of $ S_1 $ we have $ {a_{35}}^2 \leq 1 - t $, so the coefficient of $ {a_{12}}^2 $ in $ f_2 \left( a_{12}, \check{a}_{13}, 0, a_{35} \right) $ is nonpositive. By \eqref{eq:a_12_upper_bound}, \eqref{eq:S_1_geq_0} and by the definition of $ S_1 $ we have $ {a_{12}}^2 \leq \frac{1}{2} \left( 1 + 2 t + 2 s \right) $, so the coefficient of $ {a_{35}}^2 $ in $ f_2 \left( a_{12}, \check{a}_{13}, 0, a_{35} \right) $ is nonpositive. Let
\begin{align*}
\check{a}_{12} &= \frac{1}{2} \sqrt{ \frac{2 \left( t + s \right) \left( -1 + 2 s + 4 {a_{24}}^2 + 4 {a_{45}}^2 \right) }{ 4 {a_{45}}^2 + 2 s - 1 } } , \\
\check{a}_{35} &= \frac{1}{2} \sqrt{ \frac{ \left( 1 - 2 t \right) \left( -1 + 2 s + 4 {a_{24}}^2 + 4 {a_{45}}^2 \right) }{ 4 {a_{24}}^2 + 2 s - 1 } }
\end{align*}
and note that these expressions are well-defined due to \eqref{eq:a_24_lower_bound} and \eqref{eq:a_45_lower_bound}. By \eqref{eq:a_12_lower_bound} and as $ S_2 $ is monotonically decreasing in $ a_{25} $ we have $ a_{12} > \check{a}_{12} $. Similarly, by \eqref{eq:a_35_lower_bound} we have $ a_{35} > \check{a}_{35} $. Hence,
\begin{align*}
f_2 \left( a_{12}, \check{a}_{13}, 0, a_{35} \right) \leq& f_2 \left( \check{a}_{12}, \check{a}_{13}, 0, \check{a}_{35} \right) \\
 =& \left( t - 1 + {\check{a}_{35}}^2 \right) {\check{a}_{12}}^2 + \frac{1}{2} \left( 1 + s \right) - \frac{1}{2} \left( 1 + 2 t + 2 s \right) {\check{a}_{35}}^2 \\
 =& \frac{ f_3 \left( {a_{45}}^2, {a_{24}}^2, s, t \right) }{ 8 \left( 4 {a_{24}}^2 + 2 s - 1 \right) \left( 4 {a_{45}}^2 + 2 s - 1 \right) } ,
\end{align*}
where
\begin{align*}
f_3 \left( x, y, s, t \right) =& -4 \left( \left( 1 - 2 t \right) \left( 1 + t + s \right) x + \left( 3 - 2 t \right) \left( t + s \right) y \right) \left( 4 x + 4 y + 2 s - 1 \right) \\
 &- \left( 1 - 2 s \right) \left( 6 t^2 - 3 \left( 1 - 2 s \right) t - 5 s - 1 \right) \left( 4 x + 4 y + 2 s - 1 \right) \\
 &+ 4 \left( 1 + s \right) \left( 4 x + 2 s - 1 \right) \left( 4 y + 2 s - 1 \right) .
\end{align*}
Our goal is to show that $ f_3 \left( {a_{45}}^2, {a_{24}}^2, s, t \right) \leq 0 $. If we do that we will have $ f_2 \left( \check{a}_{12}, \check{a}_{13}, 0, \check{a}_{35} \right) \leq 0 $. In order to do that we bound the maximum value of $ f_3 \left( x, y, s, t \right) $ over the relevant range of its parameters. Based on \eqref{eq:a_45_better_lower_bound} and \eqref{eq:a_45_better_upper_bound} we define
\begin{align*}
x_{min} \left( s, t \right) &= \frac{1}{2} \left( 1 - s \right) - \frac{1}{4} \sqrt{ 1 - \frac{ \left( 1 - 2 s \right) \left( 3 - 2 s \right) \left( t + s \right) }{ 1 + t + s } } , \\
x_{max} \left( s, t \right) &= \frac{ \left( 3 - 2 s \right) \left( 5 - 6 t - 4 \sqrt{ \left( 1 - t \right) \left( 1 - 2 t \right) } \right) }{ 4 \left( 3 - 2 t \right) } ,
\end{align*}
and
\[ D_{s, t} = \{ \left( x, y \right) \mid x_{min} \left( s, t \right) \leq x \leq x_{max} \left( s, t \right) \} . \]
It suffices to show that for any $ \left( x, y \right) \in D_{ s, t } $ we have $ f_3 \left( x, y, s, t \right) \leq 0 $.

Note that $ f_3 $ is a smooth function over $ \mathbb{R}^4 $. Further note that by assumption $ t \leq \frac{1}{2} - t - s $, or equivalently, that $ t \leq \frac{1}{4} - \frac{1}{2} s $.

We begin by observing that
\[ f_3 \left( x, y, \frac{1}{2}, 0 \right) = -24 \left( x - y \right)^2 \leq 0 . \]
Therefore, from now on we assume that $ s < \frac{1}{2} $. Also, $ f_3 $ may be regarded as a quadratic in $y$, and the coefficient of $ y^2 $ is $ -16 \left( 3 - 2 t \right) \left( t + s \right) < 0 $. As such it has no real roots for given $x$, $s$ and $t$ when its discriminant is negative. In this case $ f_3 $ is negative for all $y$ values. When the discriminant is zero $ f_3 $ has a single real root, and this means that $ f_3 $ is nonpositive for all $y$ values.
The discriminant of $ f_3 $, regarded as a quadratic in $x$, is
\[ g_2 \left( x \right) = 256 \left( 3 + 2 s \right) C_4 x^2 - 128 \left( 1 - 2 s \right) C_5 x + 16 \left( 1 - 2 s \right)^2 C_6 , \]
where
\begin{align*}
C_4 &= 16 t^2 - 8 \left( 1 - 2 s \right) t + 3 \left( 1 - 2 s \right) , \\
C_5 &= 16 t^3 + 56 t^2 s + 40 t s^2 + 28 t^2 - 12 s^2 + 12 t s - 12 t + 9 , \\
C_6 &= \left( 4 t^2 + 4 t s + 2 s + 3 \right)^2 .
\end{align*}
Note that $ C_4 $, as a quadratic in $t$, has a minimum at $ t = \frac{ 1 - 2 s }{4} $. The coefficient of $ x^2 $ in $ g_2 \left( x \right) $ is positive because
\begin{align*}
C_4 &= 16 t^2 - 8 \left( 1 - 2 s \right) t + 3 \left( 1 - 2 s \right) \\
 &\geq 16 \left( \frac{ 1 - 2 s }{4} \right)^2 - 8 \left( 1 - 2 s \right) \left( \frac{ 1 - 2 s }{4} \right) + 3 \left( 1 - 2 s \right) \\
 &= 2 \left( 1 + s \right) \left( 1 - 2 s \right) > 0 .
\end{align*}
Therefore, $ g_2 \left( x \right) < 0 $ when $ q_2^- < x < q_2^+ $, where
\begin{align*}
q_2^\pm &= q_2^\pm \left( s, t \right) = \frac{ \left( 1 - 2 s \right) C_5 }{ 4 \left( 3 + 2 s \right) C_4 } \pm \frac{ \left( 1 - 2 s \right) \sqrt{ {C_5}^2 - \left( 3 + 2 s \right) C_4 C_6 } }{ 4 \left( 3 + 2 s \right) C_4 } \\
 &= \frac{ \left( 1 - 2 s \right) C_5 }{ 4 \left( 3 + 2 s \right) C_4 } \pm \frac{ 2 \left( 1 - 2 t \right) \left( 1 - 2 s \right) \sqrt{ \left( 3 - 2 t \right) \left( 1 + s \right) \left( t + s \right)^3 } }{ \left( 3 + 2 s \right) C_4 } .
\end{align*}
Hence, when $ q_2^- < x < q_2^+ $ we have $ f_3 \left( x, y, s, t \right) < 0 $ and when $ q_2^- \leq x \leq q_2^+ $ we have $ f_3 \left( x, y, s, t \right) \leq 0 $. In Appendix~\ref{app:xmin_geq_q2} we prove\footnote{The proofs are given in the Appendices as they are very technical and lengthy.} that $ q_2^- \leq x_{min} \left( s, t \right) $ and in Appendix~\ref{app:q2_qeq_xmax} we prove that $ x_{max} \left( s, t \right) \leq q_2^+ $. This implies that $ f_3 \left( x, y, s, t \right) \leq 0 $ for any $ \left( x, y \right) \in D_{ s, t } $.

To conclude, we showed that
\begin{align*}
P_{A[1,2,3,5]} \left( 1 \right) &= f_2 \left( a_{12}, a_{13}, a_{25}, a_{35} \right) < f_2 \left( a_{12}, \check{a}_{13}, a_{25}, a_{35} \right) \\
 &\leq f_2 \left( a_{12}, \check{a}_{13}, 0, a_{35} \right) \leq f_2 \left( \check{a}_{12}, \check{a}_{13}, 0, \check{a}_{35} \right) \leq 0 .
\end{align*}
By Observation~\ref{obs:negative_cp} we get that $ \rho \left( A[1,2,3,5] \right) > 1 $.
\end{proof}

\begin{corollary} \label{cor:spectral_radius_of_A[1,2,3,5]}
If $ t \geq \frac{1}{2} - t - s $, $ t < \frac{1}{2} $, \eqref{eq:S_1_geq_0}, \eqref{eq:a_12_upper_bound}, \eqref{eq:a_35_upper_bound}, \eqref{eq:a_24_lower_bound}, \eqref{eq:a_45_lower_bound}, \eqref{eq:a_25_upper_bound}, \eqref{eq:a_12_lower_bound}, \eqref{eq:a_24_better_upper_bound}, \eqref{eq:a_35_lower_bound}, \eqref{eq:a_24_better_lower_bound} and \eqref{eq:a_13_lower_bound} hold, then $ \rho \left( A[1,2,3,5] \right) > 1 $.
\end{corollary}

\begin{proof}
Direct application of Lemma~\ref{lem:symmetry_H}.
\end{proof}

\begin{theorem} \label{th:pattern_H}
If $ \rho \left( A \right) \leq 1 $ then $ \lambda_3 \leq \frac{1}{2} $.
\end{theorem}

\begin{proof}
If any of the conditions in Lemma~\ref{lem:sufficient_conditions_H} hold we are done. Otherwise, we have \eqref{eq:a_24_lower_bound}, \eqref{eq:a_45_lower_bound} and \eqref{eq:a_13_lower_bound}. By Lemma~\ref{lem:necessary_conditions_H} we have \eqref{eq:S_1_geq_0} and by Lemma~\ref{lem:more_necessary_conditions_H} we have \eqref{eq:a_35_upper_bound}, \eqref{eq:a_12_upper_bound}, \eqref{eq:a_25_upper_bound}. If the conditions of Lemma~\ref{lem:a_12_sufficient_condition} hold or the conditions of Corollary~\ref{cor:a_35_sufficient_condition} hold we are done. Otherwise, we have \eqref{eq:a_12_lower_bound} and \eqref{eq:a_35_lower_bound}.
Therefore, by Lemma~\ref{lem:a_45_better_lower_bound} we have \eqref{eq:a_24_better_upper_bound} and \eqref{eq:a_45_better_lower_bound}, and by Corollary~\ref{cor:a_24_better_lower_bound} we have \eqref{eq:a_45_better_upper_bound} and \eqref{eq:a_24_better_lower_bound}. If $ s = t = 0 $ we get by \eqref{eq:a_45_better_lower_bound} and \eqref{eq:a_45_better_upper_bound} that $ \frac{1}{2} < a_{45} < \frac{1}{2} $, which is a contradiction. If $ t = \frac{1}{2} $ then $ s = 0 $, and we get by \eqref{eq:a_24_better_upper_bound} and \eqref{eq:a_24_better_lower_bound} that $ \frac{1}{2} < a_{24} < \frac{1}{2} $, which is a contradiction. Otherwise, by either Lemma~\ref{lem:spectral_radius_of_A[1,2,3,5]} or Corollary~\ref{cor:spectral_radius_of_A[1,2,3,5]} we have $ \rho \left( A[1,2,3,5] \right) > 1 $, which contradicts Observation~\ref{obs:spectral_radius}.
\end{proof}

We actually proved a stronger result:
\begin{theorem}
If every principal $ 4 \times 4 $ submatrix $M$ of $A$ has $ \rho \left( M \right) \leq 1 $ then $ \lambda_3 \leq \frac{1}{2} $.
\end{theorem}

\section{Pattern $C$} \label{sec:pattern_C}

Let
\[ B = \left(
\begin{array}{ccccc}
b_{11} & b_{12} & b_{13} & 0 & 0 \\ \noalign{\medskip}
b_{12} & b_{22} & 0 & b_{24} & 0 \\ \noalign{\medskip}
b_{13} & 0 & b_{33} & 0 & b_{35} \\ \noalign{\medskip}
0 & b_{24} & 0 & b_{44} & b_{45} \\ \noalign{\medskip}
0 & 0 & b_{35} & b_{45} & b_{55}
\end{array}
\right) , \]
where all the $ b_{ij} $'s are nonnegative and $ \tr \left( B \right) = \frac{1}{2} $. Note that $B$ is in the closure of the set of matrices whose pattern is $C$, with a trace of $ \frac{1}{2} $. Let $ \left( \lambda_1, \lambda_2, \dots, \lambda_5 \right) $ be the spectrum of $B$, where $ \lambda_1 \geq \lambda_2 \geq \dots \geq \lambda_5 $.

Let
\[ Q = \left(
\begin{array}{ccccc}
0 & 1 & 0 & 0 & 0 \\ \noalign{\medskip}
0 & 0 & 0 & 1 & 0 \\ \noalign{\medskip}
1 & 0 & 0 & 0 & 0 \\ \noalign{\medskip}
0 & 0 & 0 & 0 & 1 \\ \noalign{\medskip}
0 & 0 & 1 & 0 & 0
\end{array}
\right) . \]
Then $ Q B Q^{-1} $, $ Q^2 B Q^{-2} $, $ Q^3 B Q^{-3} $, $ Q^4 B Q^{-4} $ are respectively
\begin{align*}
\left(
\begin{array}{ccccc}
b_{22} & b_{24} & b_{12} & 0 & 0 \\ \noalign{\medskip}
b_{24} & b_{44} & 0 & b_{45} & 0 \\ \noalign{\medskip}
b_{12} & 0 & b_{11} & 0 & b_{13} \\ \noalign{\medskip}
0 & b_{45} & 0 & b_{55} & b_{35} \\ \noalign{\medskip}
0 & 0 & b_{13} & b_{35} & b_{33}
\end{array}
\right) ,
\left(
\begin{array}{ccccc}
b_{44} & b_{45} & b_{24} & 0 & 0 \\ \noalign{\medskip}
b_{45} & b_{55} & 0 & b_{35} & 0 \\ \noalign{\medskip}
b_{24} & 0 & b_{22} & 0 & b_{12} \\ \noalign{\medskip}
0 & b_{35} & 0 & b_{33} & b_{13} \\ \noalign{\medskip}
0 & 0 & b_{12} & b_{13} & b_{11}
\end{array}
\right) , \\
\left(
\begin{array}{ccccc}
b_{55} & b_{35} & b_{45} & 0 & 0 \\ \noalign{\medskip}
b_{35} & b_{33} & 0 & b_{13} & 0 \\ \noalign{\medskip}
b_{45} & 0 & b_{44} & 0 & b_{24} \\ \noalign{\medskip}
0 & b_{13} & 0 & b_{11} & b_{12} \\ \noalign{\medskip}
0 & 0 & b_{24} & b_{12} & b_{22}
\end{array}
\right) ,
\left(
\begin{array}{ccccc}
b_{33} & b_{13} & b_{35} & 0 & 0 \\ \noalign{\medskip}
b_{13} & b_{11} & 0 & b_{12} & 0 \\ \noalign{\medskip}
b_{35} & 0 & b_{55} & 0 & b_{45} \\ \noalign{\medskip}
0 & b_{12} & 0 & b_{22} & b_{24} \\ \noalign{\medskip}
0 & 0 & b_{45} & b_{24} & b_{44}
\end{array}
\right) .
\end{align*}
These matrices, when considering $ b_{11}, b_{12}, \dots, b_{45}, b_{55} $ as symbols, have the same zero pattern as $B$. For each given main diagonal element of $B$ and each given main diagonal location there is $ i \in \{ 0, 1, 2 ,3, 4 \} $ such that the given element appears in the given location in $ Q^i B Q^{-i} $. Therefore, without loss of generality, we may assume:

\begin{assumption} \label{assmp:min_diagonal}
$ b_{22} $ is the smallest main diagonal element of $B$ and in particular $ b_{22} \leq \frac{1}{10} $.
\end{assumption}

Note that when $ b_{22} = 0 $ and $ b_{55} = 0 $, $B$ is in fact $A$ of Section~\ref{sec:pattern_H} with $ a_{25} = 0 $. Similarly, when $ b_{22} = 0 $ and $ b_{33} = 0 $ the matrix $ Q^3 B Q^{-3} $ is matrix $A$ with $ a_{25} = 0 $, and when $ b_{11} = 0 $ and $ b_{44} = 0 $ the matrix $ Q^4 B Q^{-4} $ is matrix $A$ with $ a_{25} = 0 $. Therefore, Theorem~\ref{th:pattern_H} applies in these cases. When any of the main diagonal elements of $B$ is $ \frac{1}{2} $ the rest of the main diagonal elements are zero, so Theorem~\ref{th:pattern_H} applies in these cases as well. Together with Assumption~\ref{assmp:min_diagonal} we can therefore assume:

\begin{assumption} \label{assmp:max_diagonal}
$ b_{11} + b_{44} > 0 $, $ b_{11} < \frac{1}{2} $, $ 0 < b_{33} < \frac{1}{2} $, $ b_{44} < \frac{1}{2} $, $ 0 < b_{55} < \frac{1}{2} $.
\end{assumption}

Let
\[ P = \left(
\begin{array}{ccccc}
0 & 0 & 0 & 1 & 0 \\ \noalign{\medskip}
0 & 1 & 0 & 0 & 0 \\ \noalign{\medskip}
0 & 0 & 0 & 0 & 1 \\ \noalign{\medskip}
1 & 0 & 0 & 0 & 0 \\ \noalign{\medskip}
0 & 0 & 1 & 0 & 0
\end{array}
\right) . \]
Then
\[ P B P^{-1} = \left(
\begin{array}{ccccc}
b_{44} & b_{24} & b_{45} & 0 & 0 \\ \noalign{\medskip}
b_{24} & b_{22} & 0 & b_{12} & 0 \\ \noalign{\medskip}
b_{45} & 0 & b_{55} & 0 & b_{35} \\ \noalign{\medskip}
0 & b_{12} & 0 & b_{11} & b_{13} \\ \noalign{\medskip}
0 & 0 & b_{35} & b_{13} & b_{33}
\end{array}
\right) . \]
The zeros exhibited in $B$ and in $ P B P^{-1} $ appear in the same locations. Therefore, by Assumption~\ref{assmp:max_diagonal} we may also assume:
\begin{assumption} \label{assmp:b_11_geq_b_44}
$ b_{11} \geq b_{44} $, $ b_{11} > 0 $.
\end{assumption}

If any of the off-diagonal elements of $B$ is zero then we can find a principal $ 3 \times 3 $ submatrix which is diagonal and so its eigenvalues are three main diagonal elements of $B$ (e.g., when $ b_{12} = 0 $ we take $ B[1,2,5] $). Therefore, this submatrix has a spectral radius at most $ \frac{1}{2} $ and so, by Observation~\ref{obs:3x3_sufficient_condition}, $ \lambda_3 \leq \frac{1}{2} $. We can therefore assume that:
\begin{assumption} \label{assmp:non_diagonal}
$ b_{12} > 0 $, $ b_{13} > 0 $, $ b_{24} > 0 $, $ b_{35} > 0 $, $ b_{45} > 0 $.
\end{assumption}

In light of the discussion in Section~\ref{sec:main_theorem} we may also assume that:
\begin{assumption} \label{assmp:not_similar_to_positive}
$B$ is not orthogonally similar to a positive symmetric matrix.
\end{assumption}

\begin{lemma} \label{lem:not_similar_to_positive_C}
$B$ has the following properties:
\begin{align}
b_{12} & \geq b_{35} , \label{eq:b_12_geq_b_35} \\
b_{24} & \geq b_{35} , \label{eq:b_24_geq_b_35} \\
b_{45} & \geq b_{13} . \nonumber
\end{align}
Moreover, if $ b_{33} \geq b_{44} $ then $ b_{24} \geq b_{13} $.
\end{lemma}

\begin{proof}
By Assumption~\ref{assmp:not_similar_to_positive} and by Theorem~\ref{th:extreme_L} there exists a $ 5 \times 5 $ nonnegative symmetric and nonzero matrix $Y$ such that $ B Y = Y B $ and $ b_{ij} y_{ij} = 0 $ for all $i$, $j$. By Assumption~\ref{assmp:non_diagonal} we deduce that $ y_{12} = y_{13} = y_{24} = y_{35} = y_{45} = 0 $.
By Assumption~\ref{assmp:min_diagonal}, if there is a zero diagonal element of $B$ then $ b_{22} = 0 $ as well. By Assumption~\ref{assmp:max_diagonal} we have $ b_{33} > 0 $ and $ b_{55} > 0 $, and by Assumption~\ref{assmp:b_11_geq_b_44} we have $ b_{11} > 0 $. Therefore, we are only left with the following cases:
\begin{enumerate}

\item $B$ has a positive main diagonal. \label{case:positive_diagonal}
Therefore,
\[ Y = \left(
\begin{array}{ccccc}
0 & 0 & 0 & y_{14} & y_{15} \\ \noalign{\medskip}
0 & 0 & y_{23} & 0 & y_{25} \\ \noalign{\medskip}
0 & y_{23} & 0 & y_{34} & 0 \\ \noalign{\medskip}
y_{14} & 0 & y_{34} & 0 & 0 \\ \noalign{\medskip}
y_{15} & y_{25} & 0 & 0 & 0
\end{array}
\right) .
\]
Letting $ B Y = \left( z_{ij} \right) $, $ Y B = \left( w_{ij} \right) $ we get the following equations:
\begin{align}
b_{13} y_{23} = z_{12} &= w_{12} = b_{24} y_{14} , \label{eq:z_12_eq_w_12_positive} \\
b_{12} y_{23} = z_{13} &= w_{13} = b_{35} y_{15} , \label{eq:z_13_eq_w_13_positive} \\
b_{11} y_{14} + b_{13} y_{34} = z_{14} &= w_{14} = b_{44} y_{14} + b_{45} y_{15} , \label{eq:z_14_eq_w_14_positive} \\
b_{22} y_{23} + b_{24} y_{34} = z_{23} &= w_{23} = b_{33} y_{23} + b_{35} y_{25} , \label{eq:z_23_eq_w_23_positive} \\
b_{12} y_{14} = z_{24} &= w_{24} = b_{45} y_{25} , \label{eq:z_24_eq_w_24_positive} \\
b_{12} y_{15} + b_{22} y_{25} = z_{25} &= w_{25} = b_{35} y_{23} + b_{55} y_{25} , \label{eq:z_25_eq_w_25_positive} \\
b_{13} y_{14} + b_{33} y_{34} = z_{34} &= w_{34} = b_{24} y_{23} + b_{44} y_{34} , \label{eq:z_34_eq_w_34_positive} \\
b_{13} y_{15} = z_{35} &= w_{35} = b_{45} y_{34} , \label{eq:z_35_eq_w_35_positive} \\
b_{24} y_{25} = z_{45} &= w_{45} = b_{35} y_{34} . \label{eq:z_45_eq_w_45_positive}
\end{align}
Due to Assumption~\ref{assmp:non_diagonal} we may divide by the off-diagonal elements of $B$. Note that by \eqref{eq:z_24_eq_w_24_positive}, \eqref{eq:z_12_eq_w_12_positive}, \eqref{eq:z_13_eq_w_13_positive} and \eqref{eq:z_35_eq_w_35_positive} either all the $ y_{ij} $'s, which are not known in advance to be zero, are zero or all of them are nonzero. As $Y$ is nonzero we conclude that they are all nonzero. Therefore, these $ y_{ij} $'s are positive.

By \eqref{eq:z_25_eq_w_25_positive} and \eqref{eq:z_13_eq_w_13_positive} we have
\begin{equation*}
b_{55} - b_{22} = \frac{ b_{12} y_{15} - b_{35} y_{23} }{ y_{25} } = \frac{ \left( {b_{12}}^2 - {b_{35}}^2 \right) y_{23} }{ b_{35} y_{25} } ,
\end{equation*}
and so by Assumption~\ref{assmp:min_diagonal} and Assumption~\ref{assmp:non_diagonal}, \eqref{eq:b_12_geq_b_35} follows.

By \eqref{eq:z_23_eq_w_23_positive} and \eqref{eq:z_45_eq_w_45_positive} we have
\begin{equation*}
b_{33} - b_{22} = \frac{ b_{24} y_{34} - b_{35} y_{25} }{ y_{23} } = \frac{ \left( {b_{24}}^2 - {b_{35}}^2 \right) y_{25} }{ b_{35} y_{23} } ,
\end{equation*}
and so by Assumption~\ref{assmp:min_diagonal} and Assumption~\ref{assmp:non_diagonal}, \eqref{eq:b_24_geq_b_35} follows.

By \eqref{eq:z_14_eq_w_14_positive} and \eqref{eq:z_35_eq_w_35_positive} we have
\begin{equation*}
b_{11} - b_{44} = \frac{ b_{45} y_{15} - b_{13} y_{34} }{ y_{14} } = \frac{ \left( {b_{45}}^2 - {b_{13}}^2 \right) y_{34} }{ b_{13} y_{14} } ,
\end{equation*}
and so by Assumption~\ref{assmp:b_11_geq_b_44} and Assumption~\ref{assmp:non_diagonal}, $ b_{45} \geq b_{13} $.

By \eqref{eq:z_34_eq_w_34_positive} and \eqref{eq:z_12_eq_w_12_positive} we have
\begin{equation*}
b_{33} - b_{44} = \frac{ b_{24} y_{23} - b_{13} y_{14} }{ y_{34} } = \frac{ \left( {b_{24}}^2 - {b_{13}}^2 \right) y_{23} }{ b_{24} y_{34} } .
\end{equation*}
If $ b_{33} \geq b_{44} $ then by Assumption~\ref{assmp:non_diagonal}, $ b_{24} \geq b_{13} $.

\item The only zero main diagonal element of $B$ is $ b_{22} $. \label{case:b_22_zero}
Therefore,
\[ B = \left(
\begin{array}{ccccc}
b_{11} & b_{12} & b_{13} & 0 & 0 \\ \noalign{\medskip}
b_{12} & 0 & 0 & b_{24} & 0 \\ \noalign{\medskip}
b_{13} & 0 & b_{33} & 0 & b_{35} \\ \noalign{\medskip}
0 & b_{24} & 0 & b_{44} & b_{45} \\ \noalign{\medskip}
0 & 0 & b_{35} & b_{45} & b_{55}
\end{array}
\right) ,
Y = \left(
\begin{array}{ccccc}
0 & 0 & 0 & y_{14} & y_{15} \\ \noalign{\medskip}
0 & y_{22} & y_{23} & 0 & y_{25} \\ \noalign{\medskip}
0 & y_{23} & 0 & y_{34} & 0 \\ \noalign{\medskip}
y_{14} & 0 & y_{34} & 0 & 0 \\ \noalign{\medskip}
y_{15} & y_{25} & 0 & 0 & 0
\end{array}
\right) .
\]
Letting $ B Y = \left( z_{ij} \right) $, $ Y B = \left( w_{ij} \right) $ we get \eqref{eq:z_13_eq_w_13_positive}, \eqref{eq:z_14_eq_w_14_positive}, \eqref{eq:z_34_eq_w_34_positive}, \eqref{eq:z_35_eq_w_35_positive}, \eqref{eq:z_45_eq_w_45_positive} and the following equations:
\begin{align}
b_{12} y_{22} + b_{13} y_{23} = z_{12} &= w_{12} = b_{24} y_{14} , \label{eq:z_12_eq_w_12_b_22_zero} \\
b_{24} y_{34} = z_{23} &= w_{23} = b_{33} y_{23} + b_{35} y_{25} , \label{eq:z_23_eq_w_23_b_22_zero} \\
b_{12} y_{14} = z_{24} &= w_{24} = b_{24} y_{22} + b_{45} y_{25} , \label{eq:z_24_eq_w_24_b_22_zero} \\
b_{12} y_{15} = z_{25} &= w_{25} = b_{35} y_{23} + b_{55} y_{25} . \label{eq:z_25_eq_w_25_b_22_zero}
\end{align}
Due to Assumption~\ref{assmp:non_diagonal} we may divide by the off-diagonal elements of $B$. If either of $ y_{14} $, $ y_{15} $, $ y_{23} $, $ y_{34} $ is zero then by \eqref{eq:z_12_eq_w_12_b_22_zero}, \eqref{eq:z_13_eq_w_13_positive}, \eqref{eq:z_35_eq_w_35_positive}, \eqref{eq:z_25_eq_w_25_b_22_zero} and \eqref{eq:z_34_eq_w_34_positive} we get that all the $ y_{ij} $'s, which are not known in advance to be zero, are zero, which contradicts $Y$ being nonzero. Therefore, $ y_{14} $, $ y_{15} $, $ y_{23} $ and $ y_{34} $ are positive.

If $ y_{25} = 0 $ then by \eqref{eq:z_13_eq_w_13_positive}, \eqref{eq:z_25_eq_w_25_b_22_zero} and by Assumption~\ref{assmp:non_diagonal} we get $ b_{12} = b_{35} $, so \eqref{eq:b_12_geq_b_35} holds. Else, $ y_{25} > 0 $. By \eqref{eq:z_25_eq_w_25_b_22_zero} and \eqref{eq:z_13_eq_w_13_positive} we have
\begin{equation*}
b_{55} = \frac{ b_{12} y_{15} - b_{35} y_{23} }{ y_{25} } = \frac{ \left( {b_{12}}^2 - {b_{35}}^2 \right) y_{23} }{ b_{35} y_{25} } ,
\end{equation*}
and so by Assumption~\ref{assmp:non_diagonal}, \eqref{eq:b_12_geq_b_35} follows.

By \eqref{eq:z_23_eq_w_23_b_22_zero} and \eqref{eq:z_45_eq_w_45_positive} we have
\begin{equation*}
b_{33} = \frac{ b_{24} y_{34} - b_{35} y_{25} }{ y_{23} } = \frac{ \left( {b_{24}}^2 - {b_{35}}^2 \right) y_{25} }{ b_{35} y_{23} } ,
\end{equation*}
and so by Assumption~\ref{assmp:non_diagonal}, \eqref{eq:b_24_geq_b_35} follows.

By identical arguments to the ones in case \ref{case:positive_diagonal}, $ b_{45} \geq b_{13} $.

By \eqref{eq:z_34_eq_w_34_positive} and \eqref{eq:z_12_eq_w_12_b_22_zero} we have
\begin{equation}
b_{33} - b_{44} = \frac{ b_{24} y_{23} - b_{13} y_{14} }{ y_{34} } = \frac{ \left( {b_{24}}^2 - {b_{13}}^2 \right) y_{23} }{ b_{24} y_{34} } - \frac{ b_{12} b_{13} y_{22} }{ b_{24} y_{34} } . \label{eq:b_33_m_b_44}
\end{equation}
If $ b_{33} \geq b_{44} $ then the left term of the right-hand side of \eqref{eq:b_33_m_b_44} is nonnegative and so by Assumption~\ref{assmp:non_diagonal}, $ b_{24} \geq b_{13} $.

\item The only zero main diagonal elements of $B$ are $ b_{22} $ and $ b_{44} $. Therefore,
\[ B = \left(
\begin{array}{ccccc}
b_{11} & b_{12} & b_{13} & 0 & 0 \\ \noalign{\medskip}
b_{12} & 0 & 0 & b_{24} & 0 \\ \noalign{\medskip}
b_{13} & 0 & b_{33} & 0 & b_{35} \\ \noalign{\medskip}
0 & b_{24} & 0 & 0 & b_{45} \\ \noalign{\medskip}
0 & 0 & b_{35} & b_{45} & b_{55}
\end{array}
\right) ,
Y = \left(
\begin{array}{ccccc}
0 & 0 & 0 & y_{14} & y_{15} \\ \noalign{\medskip}
0 & y_{22} & y_{23} & 0 & y_{25} \\ \noalign{\medskip}
0 & y_{23} & 0 & y_{34} & 0 \\ \noalign{\medskip}
y_{14} & 0 & y_{34} & y_{44} & 0 \\ \noalign{\medskip}
y_{15} & y_{25} & 0 & 0 & 0
\end{array}
\right) .
\]
Letting $ B Y = \left( z_{ij} \right) $, $ Y B = \left( w_{ij} \right) $ we get \eqref{eq:z_13_eq_w_13_positive}, \eqref{eq:z_35_eq_w_35_positive}, \eqref{eq:z_12_eq_w_12_b_22_zero}, \eqref{eq:z_23_eq_w_23_b_22_zero}, \eqref{eq:z_25_eq_w_25_b_22_zero} and the following equations:
\begin{align}
b_{11} y_{14} + b_{13} y_{34} = z_{14} &= w_{14} = b_{45} y_{15} , \label{eq:z_14_eq_w_14_b_22_b_44_zero} \\
b_{12} y_{14} + b_{24} y_{44} = z_{24} &= w_{24} = b_{24} y_{22} + b_{45} y_{25} , \label{eq:z_24_eq_w_24_b_22_b_44_zero} \\
b_{13} y_{14} + b_{33} y_{34} = z_{34} &= w_{34} = b_{24} y_{23} , \label{eq:z_34_eq_w_34_b_22_b_44_zero} \\
b_{24} y_{25} = z_{45} &= w_{45} = b_{35} y_{34} + b_{45} y_{44} . \label{eq:z_45_eq_w_45_b_22_b_44_zero}
\end{align}
Due to Assumption~\ref{assmp:non_diagonal} we may divide by the off-diagonal elements of $B$. If either of $ y_{14} $, $ y_{15} $, $ y_{23} $, $ y_{34} $ is zero then by \eqref{eq:z_12_eq_w_12_b_22_zero}, \eqref{eq:z_13_eq_w_13_positive}, \eqref{eq:z_35_eq_w_35_positive}, \eqref{eq:z_25_eq_w_25_b_22_zero}, \eqref{eq:z_14_eq_w_14_b_22_b_44_zero}, \eqref{eq:z_24_eq_w_24_b_22_b_44_zero} and by Assumption~\ref{assmp:max_diagonal} we get that all the $ y_{ij} $'s, which are not known in advance to be zero, are zero, which contradicts $Y$ being nonzero. Therefore, $ y_{14} $, $ y_{15} $, $ y_{23} $ and $ y_{34} $ are positive.

By identical arguments to the ones in case \ref{case:b_22_zero}, \eqref{eq:b_12_geq_b_35} follows. By \eqref{eq:z_23_eq_w_23_b_22_zero} and \eqref{eq:z_45_eq_w_45_b_22_b_44_zero} we have
\begin{equation}
b_{33} = \frac{ b_{24} y_{34} - b_{35} y_{25} }{ y_{23} } = \frac{ \left( {b_{24}}^2 - {b_{35}}^2 \right) y_{34} }{ b_{24} y_{23} } - \frac{ b_{35} b_{45} y_{44} }{ b_{24} y_{23} } . \label{eq:b_33}
\end{equation}
As $ b_{33} \geq 0 $, the left term of the right-hand side of \eqref{eq:b_33} is nonnegative and so by Assumption~\ref{assmp:non_diagonal}, \eqref{eq:b_24_geq_b_35} follows.

By \eqref{eq:z_14_eq_w_14_b_22_b_44_zero} and \eqref{eq:z_35_eq_w_35_positive} we have
\begin{equation*}
b_{11} = \frac{ b_{45} y_{15} - b_{13} y_{34} }{ y_{14} } = \frac{ \left( {b_{45}}^2 - {b_{13}}^2 \right) y_{34} }{ b_{13} y_{14} } ,
\end{equation*}
and so by Assumption~\ref{assmp:non_diagonal}, $ b_{45} \geq b_{13} $.

By \eqref{eq:z_34_eq_w_34_b_22_b_44_zero} and \eqref{eq:z_12_eq_w_12_b_22_zero} we have
\begin{equation}
b_{33} = \frac{ b_{24} y_{23} - b_{13} y_{14} }{ y_{34} } = \frac{ \left( {b_{24}}^2 - {b_{13}}^2 \right) y_{23} }{ b_{24} y_{34} } - \frac{ b_{12} b_{13} y_{22} }{ b_{24} y_{34} } . \label{eq:b_33_m_b_44_2}
\end{equation}
As $ b_{33} \geq 0 $, the left term of the right-hand side of \eqref{eq:b_33_m_b_44_2} is nonnegative and so by Assumption~\ref{assmp:non_diagonal}, $ b_{24} \geq b_{13} $.

\end{enumerate}
\end{proof}

\begin{lemma} \label{lem:4x4_conditions}
Let
\[ M = \left(
\begin{array}{ccccc}
x_1 & y_1 & y_2 & 0 \\ \noalign{\medskip}
y_1 & x_2 & 0 & y_3 \\ \noalign{\medskip}
y_2 & 0 & x_3 & 0 \\ \noalign{\medskip}
0 & y_3 & 0 & x_4
\end{array}
\right) , \]
where $ \tr \left( M \right) \leq \frac{1}{2} $, $ 0 \leq x_i < \frac{1}{2} $ for $ i = 1, 2, 3, 4 $ and $ y_j > 0 $ for $ j = 1, 2, 3 $. Let $ \left( \mu_1, \mu_2, \mu_3, \mu_4 \right) $ be the spectrum of $M$, where $ \mu_1 \geq \mu_2 \geq \mu_3 \geq \mu_4 $.
\begin{enumerate}

\item If $ \rho \left( M \right) \leq 1 $ then
\begin{equation}
y_2 < \sqrt{ \left( 1 - x_1 \right) \left( 1 - x_3 \right) } \label{eq:y_2_upper_bound}
\end{equation}
and
\begin{equation}
\ y_3 \leq \sqrt{ \left( 1 - x_4 \right) \left( 1 - x_2 - \frac{ \left( 1 - x_3 \right) {y_1}^2 }{ \left( 1 - x_1 \right) \left( 1 - x_3 \right)- {y_2}^2 } \right) } \label{eq:y_3_upper_bound}
\end{equation}
hold.

\item If
\begin{equation}
y_2 \leq \frac{1}{2} \sqrt{ \left( 1 - 2 x_1 \right) \left( 1 - 2 x_3 \right) } \label{eq:y_2_2nd_upper_bound}
\end{equation}
holds, then $ \mu_2 \leq \frac{1}{2} $.

\item If
\begin{equation}
y_2 > \frac{1}{2} \sqrt{ \left( 1 - 2 x_1 \right) \left( 1 - 2 x_3 \right) } \label{eq:y_2_lower_bound}
\end{equation}
and
\begin{equation}
y_3 \leq \frac{1}{2} \sqrt{ \left( 1 - 2 x_4 \right) \left( 1 - 2 x_2 + \frac{ 4 \left( 1 - 2 x_3 \right) {y_1}^2 }{ 4 {y_2}^2 - \left( 1 - 2 x_1 \right) \left( 1 - 2 x_3 \right) } \right) } \label{eq:y_3_2nd_upper_bound}
\end{equation}
hold, then $ \mu_2 \leq \frac{1}{2} $.

\end{enumerate}
\end{lemma}

\begin{proof}
\begin{enumerate}

\item It can be checked that
\[ P_{M[1,2,3]} \left( 1 \right) = - \left( 1 - x_3 \right) {y_1}^2 - \left( 1 - x_2 \right) \left( {y_2}^2 - \left( 1 - x_1 \right) \left( 1 - x_3 \right) \right) . \]
If $ y_2 \geq \sqrt{ \left( 1 - x_1 \right) \left( 1 - x_3 \right) } $ then, as $ x_2 < \frac{1}{2} $, $ x_3 < \frac{1}{2} $ and $ y_1 > 0 $, we have $ P_{M[1,2,3]} \left( 1 \right) < 0 $. By Observation~\ref{obs:negative_cp} $ \rho \left( M[1,2,3] \right) > 1 $, which contradicts Observation~\ref{obs:spectral_radius}. Therefore, \eqref{eq:y_2_upper_bound} holds.

It can be checked that
\begin{align*}
P_M & \left( 1 \right) = - \left( \left( 1 - x_1 \right) \left( 1 - x_3 \right)- {y_2}^2 \right) {y_3}^2 \\
 &+ \left( 1 - x_4 \right) \left( \left( 1 - x_2 \right) \left( \left( 1 - x_1 \right) \left( 1 - x_3 \right)- {y_2}^2 \right) - \left( 1 - x_3 \right) {y_1}^2 \right) .
\end{align*}
By \eqref{eq:y_2_upper_bound} the coefficient of $ {y_3}^2 $ is negative. If
\[ \ y_3 > \sqrt{ \left( 1 - x_4 \right) \left( 1 - x_2 - \frac{ \left( 1 - x_3 \right) {y_1}^2 }{ \left( 1 - x_1 \right) \left( 1 - x_3 \right)- {y_2}^2 } \right) } \]
we get that $ P_M \left( 1 \right) < 0 $. By Observation~\ref{obs:negative_cp} $ \rho \left( M \right) > 1 $, which contradicts the assumption. Therefore, \eqref{eq:y_3_upper_bound} holds. Note that the denominator in \eqref{eq:y_3_upper_bound} is positive by \eqref{eq:y_2_upper_bound}, and the entire value inside the square root of \eqref{eq:y_3_upper_bound} is nonnegative because $ P_{M[1,2,3]} \left( 1 \right) \geq 0 $ and $ x_4 < \frac{1}{2} $.

\item The eigenvalues of $ M[1,3,4] $ are $ x_4, \frac{1}{2} \left( x_1 + x_3 \right) \pm \frac{1}{2} \sqrt{ \left( x_1 - x_3 \right)^2 + 4 {y_2}^2 } $. By assumption $ x_4 < \frac{1}{2} $. If \eqref{eq:y_2_2nd_upper_bound} holds then
\[ \frac{1}{2} \left( x_1 + x_3 \right) + \frac{1}{2} \sqrt{ \left( x_1 - x_3 \right)^2 + 4 {y_2}^2 } \leq \frac{1}{2} , \]
and so $ \rho \left( M[1,3,4] \right) \leq \frac{1}{2} $. By the eigenvalue interlacing property we have $ \mu_2 \leq \frac{1}{2} $.

\item By \eqref{eq:y_2_lower_bound} we have
\begin{equation}
\frac{1}{2} \left( x_1 + x_3 \right) + \frac{1}{2} \sqrt{ \left( x_1 - x_3 \right)^2 + 4 {y_2}^2 } > \frac{1}{2} . \label{eq:ev_M[1,3,4]}
\end{equation}
As $ x_4 \geq 0 $, $ M[1,3,4] $ has at least two nonnegative eigenvalues. If all the eigenvalues of $ M[1,3,4] $ are nonnegative then, as $ M[1,3,4] $ is nonnegative and
\[ \tr \left( M[1,3,4] \right) = x_1 + x_3 + x_4 \leq \tr \left( M \right) \leq \frac{1}{2} , \]
we also have $ \rho \left( M[1,3,4] \right) \leq \frac{1}{2} $, contradicting \eqref{eq:ev_M[1,3,4]}. Therefore,
\begin{align*}
\frac{1}{2} \left( x_1 + x_3 \right) + \frac{1}{2} \sqrt{ \left( x_1 - x_3 \right)^2 + 4 {y_2}^2 } &> \frac{1}{2} > x_4 \geq 0 \\
 > \frac{1}{2} \left( x_1 + x_3 \right) &- \frac{1}{2} \sqrt{ \left( x_1 - x_3 \right)^2 + 4 {y_2}^2 } . 
\end{align*}
Hence, by the eigenvalue interlacing property we have that
\begin{align*}
\mu_4 &\leq \frac{1}{2} \left( x_1 + x_3 \right) - \frac{1}{2} \sqrt{ \left( x_1 - x_3 \right)^2 + 4 {y_2}^2 } \leq \mu_3 \\
 &\leq x_4 \leq \mu_2 \leq \frac{1}{2} \left( x_1 + x_3 \right) + \frac{1}{2} \sqrt{ \left( x_1 - x_3 \right)^2 + 4 {y_2}^2 } \leq \mu_1 .
\end{align*}
Note that $ \mu_3 < \frac{1}{2} $, since otherwise $ x_4 \geq \frac{1}{2} $, contradicting the assumptions.

It can be checked that
\begin{align*}
P_M & \left( \frac{1}{2} \right) = \frac{1}{4} \left( 4 {y_2}^2 - \left( 1 - 2 x_1 \right) \left( 1 - 2 x_3 \right) \right) {y_3}^2 \\
 &- \frac{1}{16} \left( 1 - 2 x_4 \right) \left( \left( 1 - 2 x_2 \right) \left( 4 {y_2}^2 - \left( 1 - 2 x_1 \right) \left( 1 - 2 x_3 \right) \right) + 4 \left( 1 - 2 x_3 \right) {y_1}^2 \right) .
\end{align*}
By \eqref{eq:y_2_lower_bound} the coefficient of $ {y_3}^2 $ is positive, and so by \eqref{eq:y_3_2nd_upper_bound} (which is well-defined by \eqref{eq:y_2_lower_bound}) we have $ P_M \left( \frac{1}{2} \right) \leq 0 $. If $ P_M \left( \frac{1}{2} \right) = 0 $ then, as $ \mu_3 < \frac{1}{2} < \mu_1 $, we have $ \mu_2 = \frac{1}{2} $. If $ P_M \left( \frac{1}{2} \right) < 0 $ then by Observation~\ref{obs:negative_cp} we have $ \mu_2 < \frac{1}{2} $.

\end{enumerate}
\end{proof}

\begin{lemma} \label{lem:b_12_and_b_45_conditions}
\begin{enumerate}

\item If $ \rho \left( B \right) \leq 1 $ then
\begin{equation}
b_{12} < \sqrt{ \left( 1 - b_{11} \right) \left( 1 - b_{22} \right) } \label{eq:b_12_upper_bound}
\end{equation}
and
\begin{equation}
b_{45} \leq \sqrt{ \left( 1 - b_{55} \right) \left( 1 - b_{44} - \frac{ \left( 1 - b_{11} \right) {b_{24}}^2 }{ \left( 1 - b_{11} \right) \left( 1 - b_{22} \right) - {b_{12}}^2 } \right) } . \label{eq:b_45_upper_bound}
\end{equation}

\item If
\begin{equation}
b_{12} \leq \frac{1}{2} \sqrt{ \left( 1 - 2 b_{11} \right) \left( 1 - 2 b_{22} \right) } \label{eq:b_12_sufficient_upper_bound}
\end{equation}
holds, then $ \lambda_3 \leq \frac{1}{2} $.

\item If
\begin{equation}
b_{12} > \frac{1}{2} \sqrt{ \left( 1 - 2 b_{11} \right) \left( 1 - 2 b_{22} \right) } \label{eq:b_12_lower_bound}
\end{equation}
and
\begin{equation}
b_{45} \leq \frac{1}{2} \sqrt{ \left( 1 - 2 b_{55} \right) \left( 1 - 2 b_{44} + \frac{ 4 \left( 1 - 2 b_{11} \right) {b_{24}}^2 }{ 4 {b_{12}}^2 - \left( 1 - 2 b_{11} \right) \left( 1 - 2 b_{22} \right) } \right) } \label{eq:b_45_sufficient_upper_bound}
\end{equation}
hold, then $ \lambda_3 \leq \frac{1}{2} $.

\end{enumerate}
\end{lemma}

\begin{proof}
Note first that \eqref{eq:b_45_upper_bound} is well-defined since
\[ \left( \left( 1 - b_{11} \right) \left( 1 - b_{22} \right) - {b_{12}}^2 \right) \left(1 - b_{44} \right) - \left( 1 - b_{11} \right) {b_{24}}^2 = P_{B[1,2,4]} \left( 1 \right) \geq 0 , \]
where the inequality follows from $ \rho \left( B \right) \leq 1 $ and Observation~\ref{obs:spectral_radius}.

Let
\[ P = \left(
\begin{array}{ccccc}
0 & 1 & 0 & 0 \\ \noalign{\medskip}
0 & 0 & 1 & 0 \\ \noalign{\medskip}
1 & 0 & 0 & 0 \\ \noalign{\medskip}
0 & 0 & 0 & 1
\end{array}
\right) , \quad
M = P \cdot B[1,2,4,5] \cdot P^{-1} = \left(
\begin{array}{ccccc}
b_{22} & b_{24} & b_{12} & 0 \\ \noalign{\medskip}
b_{24} & b_{44} & 0 & b_{45} \\ \noalign{\medskip}
b_{12} & 0 & b_{11} & 0 \\ \noalign{\medskip}
0 & b_{45} & 0 & b_{55}
\end{array}
\right). \]
Then $ \tr \left( M \right) \leq \tr \left( B \right) = \frac{1}{2} $. By Observation~\ref{obs:spectral_radius}, $ \rho \left( M \right) \leq 1 $. The results follow by Lemma~\ref{lem:4x4_conditions}, as $M$ has the required pattern and due to Assumption~\ref{assmp:min_diagonal}, Assumption~\ref{assmp:max_diagonal} and Assumption~\ref{assmp:non_diagonal}, and by using Observation~\ref{obs:4x4_sufficient_condition}.
\end{proof}

\begin{lemma} \label{lem:b_24_and_b_13_conditions}
\begin{enumerate}

\item If $ \rho \left( B \right) \leq 1 $ then
\begin{equation}
b_{24} < \sqrt{ \left( 1 - b_{22} \right) \left( 1 - b_{44} \right) } \label{eq:b_24_upper_bound}
\end{equation}
and
\begin{equation}
b_{13} \leq \sqrt{ \left( 1 - b_{33} \right) \left( 1 - b_{11} - \frac{ \left( 1 - b_{44} \right) {b_{12}}^2 }{ \left( 1 - b_{22} \right) \left( 1 - b_{44} \right) - {b_{24}}^2 } \right) } . \label{eq:b_13_upper_bound}
\end{equation}

\item If
\begin{equation}
b_{24} \leq \frac{1}{2} \sqrt{ \left( 1 - 2 b_{22} \right) \left( 1 - 2 b_{44} \right) } \label{eq:b_24_sufficient_upper_bound}
\end{equation}
holds, then $ \lambda_3 \leq \frac{1}{2} $.

\item If
\begin{equation}
b_{24} > \frac{1}{2} \sqrt{ \left( 1 - 2 b_{22} \right) \left( 1 - 2 b_{44} \right) } \label{eq:b_24_lower_bound}
\end{equation}
and
\begin{equation}
b_{13} \leq \frac{1}{2} \sqrt{ \left( 1 - 2 b_{33} \right) \left( 1 - 2 b_{11} + \frac{ 4 \left( 1 - 2 b_{44} \right) {b_{12}}^2 }{ 4 {b_{24}}^2 - \left( 1 - 2 b_{22} \right) \left( 1 - 2 b_{44} \right) } \right) } \label{eq:b_13_sufficient_upper_bound}
\end{equation}
hold, then $ \lambda_3 \leq \frac{1}{2} $.

\end{enumerate}
\end{lemma}

\begin{proof}
Note first that \eqref{eq:b_13_upper_bound} is well-defined since
\[ \left( \left( 1 - b_{22} \right) \left( 1 - b_{44} \right) - {b_{24}}^2 \right) \left(1 - b_{11} \right) - \left( 1 - b_{44} \right) {b_{12}}^2 = P_{B[1,2,4]} \left( 1 \right) \geq 0 , \]
where the inequality follows from $ \rho \left( B \right) \leq 1 $ and Observation~\ref{obs:spectral_radius}.

Let 
\[ P = \left(
\begin{array}{ccccc}
0 & 1 & 0 & 0 \\ \noalign{\medskip}
1 & 0 & 0 & 0 \\ \noalign{\medskip}
0 & 0 & 0 & 1 \\ \noalign{\medskip}
0 & 0 & 1 & 0
\end{array}
\right) , \quad
M = P \cdot B[1,2,3,4] \cdot P^{-1} = \left(
\begin{array}{ccccc}
b_{22} & b_{12} & b_{24} & 0 \\ \noalign{\medskip}
b_{12} & b_{11} & 0 & b_{13} \\ \noalign{\medskip}
b_{24} & 0 & b_{44} & 0 \\ \noalign{\medskip}
0 & b_{13} & 0 & b_{33}
\end{array}
\right). \]
Then $ \tr \left( M \right) \leq \tr \left( B \right) = \frac{1}{2} $. By Observation~\ref{obs:spectral_radius}, $ \rho \left( M \right) \leq 1 $. The results follow by Lemma~\ref{lem:4x4_conditions}, as $M$ has the required pattern and due to Assumption~\ref{assmp:min_diagonal}, Assumption~\ref{assmp:max_diagonal} and Assumption~\ref{assmp:non_diagonal}, and by using Observation~\ref{obs:4x4_sufficient_condition}.
\end{proof}

\begin{lemma} \label{lem:b_35_sufficient_condition_2345}
\begin{enumerate}

\item If
\[ b_{35} \leq \frac{1}{2} \sqrt{ \left( 1 - 2 b_{33} \right) \left( 1 - 2 b_{55} \right) } \]
holds, then $ \lambda_3 \leq \frac{1}{2} $.

\item If \eqref{eq:b_24_lower_bound} and
\[ b_{35} \leq \frac{1}{2} \sqrt{ \left( 1 - 2 b_{33} \right) \left( 1 - 2 b_{55} + \frac{ 4 \left( 1 - 2 b_{22} \right) {b_{45}}^2 }{ 4 {b_{24}}^2 - \left( 1 - 2 b_{22} \right) \left( 1 - 2 b_{44} \right) } \right) } \]
hold, then $ \lambda_3 \leq \frac{1}{2} $.

\end{enumerate}
\end{lemma}

\begin{proof}
\begin{enumerate}

\item Let
\begin{align*}
P &= \left(
\begin{array}{ccccc}
0 & 0 & 0 & 1 \\ \noalign{\medskip}
0 & 0 & 1 & 0 \\ \noalign{\medskip}
0 & 1 & 0 & 0 \\ \noalign{\medskip}
1 & 0 & 0 & 0
\end{array}
\right) , \\
M &= P \cdot B[2,3,4,5] \cdot P^{-1} = \left(
\begin{array}{ccccc}
b_{55} & b_{45} & b_{35} & 0 \\ \noalign{\medskip}
b_{45} & b_{44} & 0 & b_{24} \\ \noalign{\medskip}
b_{35} & 0 & b_{33} & 0 \\ \noalign{\medskip}
0 & b_{24} & 0 & b_{22}
\end{array}
\right) .
\end{align*}
Then $ \tr \left( M \right) \leq \tr \left( B \right) = \frac{1}{2} $. The result follows by Lemma~\ref{lem:4x4_conditions}, as $M$ has the required pattern and due to Assumption~\ref{assmp:min_diagonal}, Assumption~\ref{assmp:max_diagonal} and Assumption~\ref{assmp:non_diagonal}, and by using Observation~\ref{obs:4x4_sufficient_condition}.

\item Let
\begin{align*}
P &= \left(
\begin{array}{ccccc}
0 & 0 & 1 & 0 \\ \noalign{\medskip}
0 & 0 & 0 & 1 \\ \noalign{\medskip}
1 & 0 & 0 & 0 \\ \noalign{\medskip}
0 & 1 & 0 & 0
\end{array}
\right) , \\
M &= P \cdot B[2,3,4,5] \cdot P^{-1} = \left(
\begin{array}{ccccc}
b_{44} & b_{45} & b_{24} & 0 \\ \noalign{\medskip}
b_{45} & b_{55} & 0 & b_{35} \\ \noalign{\medskip}
b_{24} & 0 & b_{22} & 0 \\ \noalign{\medskip}
0 & b_{35} & 0 & b_{33}
\end{array}
\right) .
\end{align*}
Then $ \tr \left( M \right) \leq \tr \left( B \right) = \frac{1}{2} $. The result follows by Lemma~\ref{lem:4x4_conditions}, as $M$ has the required pattern and due to Assumption~\ref{assmp:min_diagonal}, Assumption~\ref{assmp:max_diagonal} and Assumption~\ref{assmp:non_diagonal}, and by using Observation~\ref{obs:4x4_sufficient_condition}.

\end{enumerate}
\end{proof}

\begin{lemma} \label{lem:b_35_sufficient_condition_1235}
If \eqref{eq:b_12_lower_bound} and
\[ b_{35} \leq \frac{1}{2} \sqrt{ \left( 1 - 2 b_{55} \right) \left( 1 - 2 b_{33} + \frac{ 4 \left( 1 - 2 b_{22} \right) {b_{13}}^2 }{ 4 {b_{12}}^2 - \left( 1 - 2 b_{11} \right) \left( 1 - 2 b_{22} \right) } \right) } \]
hold, then $ \lambda_3 \leq \frac{1}{2} $.
\end{lemma}

\begin{proof}
Let
\[ P = \left(
\begin{array}{ccccc}
1 & 0 & 0 & 0 \\ \noalign{\medskip}
0 & 0 & 1 & 0 \\ \noalign{\medskip}
0 & 1 & 0 & 0 \\ \noalign{\medskip}
0 & 0 & 0 & 1
\end{array}
\right) , \quad
M = P \cdot B[1,2,3,5] \cdot P^{-1} = \left(
\begin{array}{ccccc}
b_{11} & b_{13} & b_{12} & 0 \\ \noalign{\medskip}
b_{13} & b_{33} & 0 & b_{35} \\ \noalign{\medskip}
b_{12} & 0 & b_{22} & 0 \\ \noalign{\medskip}
0 & b_{35} & 0 & b_{55}
\end{array}
\right) . \]
Then $ \tr \left( M \right) \leq \tr \left( B \right) = \frac{1}{2} $. The result follows by Lemma~\ref{lem:4x4_conditions}, as $M$ has the required pattern and due to Assumption~\ref{assmp:min_diagonal}, Assumption~\ref{assmp:max_diagonal} and Assumption~\ref{assmp:non_diagonal}, and by using Observation~\ref{obs:4x4_sufficient_condition}.
\end{proof}

\begin{lemma} \label{lem:b_45_better_upper_bound}
If $ \rho \left( B \right) \leq 1 $, \eqref{eq:b_24_geq_b_35}, \eqref{eq:b_24_lower_bound} and
\begin{equation}
b_{35} > \frac{1}{2} \sqrt{ \left( 1 - 2 b_{33} \right) \left( 1 - 2 b_{55} + \frac{ 4 \left( 1 - 2 b_{22} \right) {b_{45}}^2 }{ 4 {b_{24}}^2 - \left( 1 - 2 b_{22} \right) \left( 1 - 2 b_{44} \right) } \right) } \label{eq:b_35_lower_bound_2345}
\end{equation}
hold, then
\begin{equation}
b_{45} < \frac{1}{2} \sqrt{ \left( \frac{ 4 {b_{24}}^2 }{ 1 - 2 b_{22} } - \left( 1 - 2 b_{44} \right) \right) \left( \frac{ 4 {b_{24}}^2 }{ 1 - 2 b_{33} } - \left( 1 - 2 b_{55} \right) \right) } . \label{eq:b_45_better_upper_bound}
\end{equation}
\end{lemma}

\begin{proof}
Note that by \eqref{eq:b_24_lower_bound}, \eqref{eq:b_35_lower_bound_2345} is well-defined. By \eqref{eq:b_24_geq_b_35} and \eqref{eq:b_35_lower_bound_2345} we get
\[ b_{24} \geq b_{35} > \frac{1}{2} \sqrt{ \left( 1 - 2 b_{33} \right) \left( 1 - 2 b_{55} + \frac{ 4 \left( 1 - 2 b_{22} \right) {b_{45}}^2 }{ 4 {b_{24}}^2 - \left( 1 - 2 b_{22} \right) \left( 1 - 2 b_{44} \right) } \right) } \]
and so \eqref{eq:b_45_better_upper_bound} follows. Note that the value inside the square root of \eqref{eq:b_45_better_upper_bound} is well-defined by Assumption~\ref{assmp:min_diagonal} and Assumption~\ref{assmp:max_diagonal}, and nonnegative by \eqref{eq:b_24_geq_b_35}, \eqref{eq:b_24_lower_bound} and \eqref{eq:b_35_lower_bound_2345}.
\end{proof}

\begin{lemma} \label{lem:b_13_better_upper_bound}
If $ \rho \left( B \right) \leq 1 $, \eqref{eq:b_12_geq_b_35}, \eqref{eq:b_12_lower_bound} and
\begin{equation}
b_{35} > \frac{1}{2} \sqrt{ \left( 1 - 2 b_{55} \right) \left( 1 - 2 b_{33} + \frac{ 4 \left( 1 - 2 b_{22} \right) {b_{13}}^2 }{ 4 {b_{12}}^2 - \left( 1 - 2 b_{11} \right) \left( 1 - 2 b_{22} \right) } \right) } \label{eq:b_35_lower_bound_1235}
\end{equation}
hold, then
\begin{equation}
b_{13} < \frac{1}{2} \sqrt{ \left( \frac{ 4 {b_{12}}^2 }{ 1 - 2 b_{22} } - \left( 1 - 2 b_{11} \right) \right) \left( \frac{ 4 {b_{12}}^2 }{ 1 - 2 b_{55} } - \left( 1 - 2 b_{33} \right) \right) } . \label{eq:b_13_better_upper_bound}
\end{equation}
\end{lemma}

\begin{proof}
Note that by \eqref{eq:b_12_lower_bound}, \eqref{eq:b_35_lower_bound_1235} is well-defined. By \eqref{eq:b_12_geq_b_35} and \eqref{eq:b_35_lower_bound_1235} we get
\[ b_{12} \geq b_{35} > \frac{1}{2} \sqrt{ \left( 1 - 2 b_{55} \right) \left( 1 - 2 b_{33} + \frac{ 4 \left( 1 - 2 b_{22} \right) {b_{13}}^2 }{ 4 {b_{12}}^2 - \left( 1 - 2 b_{11} \right) \left( 1 - 2 b_{22} \right) } \right) } \]
and so \eqref{eq:b_13_better_upper_bound} follows. Note that the value inside the square root of \eqref{eq:b_13_better_upper_bound} is well-defined by Assumption~\ref{assmp:min_diagonal} and Assumption~\ref{assmp:max_diagonal}, and nonnegative by \eqref{eq:b_12_geq_b_35}, \eqref{eq:b_12_lower_bound} and \eqref{eq:b_35_lower_bound_1235}.
\end{proof}

\begin{lemma} \label{lem:b_24_better_lower_bound}
If $ \rho \left( B \right) \leq 1 $, \eqref{eq:b_24_geq_b_35}, \eqref{eq:b_12_lower_bound}, \eqref{eq:b_24_lower_bound}, \eqref{eq:b_45_better_upper_bound},
\begin{equation}
b_{35} > \frac{1}{2} \sqrt{ \left( 1 - 2 b_{33} \right) \left( 1 - 2 b_{55} \right) } \label{eq:b_35_lower_bound}
\end{equation}
and
\begin{align}
b_{45} &> \frac{1}{2} \sqrt{ \left( 1 - 2 b_{55} \right) \left( 1 - 2 b_{44} + \frac{ 4 \left( 1 - 2 b_{11} \right) {b_{24}}^2 }{ 4 {b_{12}}^2 - \left( 1 - 2 b_{11} \right) \left( 1 - 2 b_{22} \right) } \right) } \label{eq:b_45_lower_bound}
\end{align}
hold, then
\begin{equation}
b_{24} > \frac{1}{2} \sqrt{ 1 + 2 b_{11} + 4 b_{22} b_{44} + 4 b_{33} b_{55} } . \label{eq:b_24_better_lower_bound}
\end{equation}
\end{lemma}

\begin{proof}
Note that \eqref{eq:b_45_lower_bound} is well-defined by \eqref{eq:b_12_lower_bound}. Also note that by \eqref{eq:b_24_geq_b_35}, \eqref{eq:b_24_lower_bound} and \eqref{eq:b_35_lower_bound}, \eqref{eq:b_45_better_upper_bound} is well-defined. By \eqref{eq:b_45_better_upper_bound} and \eqref{eq:b_45_lower_bound} we have
\begin{align*}
\frac{1}{2} \sqrt{ \left( 1 - 2 b_{55} \right) \left( 1 - 2 b_{44} + \frac{ 4 \left( 1 - 2 b_{11} \right) {b_{24}}^2 }{ 4 {b_{12}}^2 - \left( 1 - 2 b_{11} \right) \left( 1 - 2 b_{22} \right) } \right) } < b_{45} \qquad \\
 < \frac{1}{2} \sqrt{ \left( \frac{ 4 {b_{24}}^2 }{ 1 - 2 b_{22} } - \left( 1 - 2 b_{44} \right) \right) \left( \frac{ 4 {b_{24}}^2 }{ 1 - 2 b_{33} } - \left( 1 - 2 b_{55} \right) \right) } ,
\end{align*}
which, as $ \sum_{i=1}^{5} b_{ii} = \tr \left( B \right) = \frac{1}{2} $, is equivalent to
\begin{equation}
\frac{ 1 - 2 b_{11} }{ 4 {b_{12}}^2 - \left( 1 - 2 b_{11} \right) \left( 1 - 2 b_{22} \right) } < \frac{ 4 {b_{24}}^2 - 1 - 2 b_{11} - 4 b_{22} b_{44} - 4 b_{33} b_{55} }{ \left( 1 - 2 b_{22} \right) \left( 1 - 2 b_{33} \right) \left( 1 - 2 b_{55} \right) } . \label{eq:b_12_b_24_relation}
\end{equation}
The left-hand side of \eqref{eq:b_12_b_24_relation} is nonnegative by \eqref{eq:b_12_lower_bound}. Therefore, \eqref{eq:b_24_better_lower_bound} follows.
\end{proof}

\begin{lemma} \label{lem:b_12_better_lower_bound}
If $ \rho \left( B \right) \leq 1 $, \eqref{eq:b_12_geq_b_35}, \eqref{eq:b_12_lower_bound}, \eqref{eq:b_24_lower_bound}, \eqref{eq:b_13_better_upper_bound}, \eqref{eq:b_35_lower_bound} and
\begin{align}
b_{13} &> \frac{1}{2} \sqrt{ \left( 1 - 2 b_{33} \right) \left( 1 - 2 b_{11} + \frac{ 4 \left( 1 - 2 b_{44} \right) {b_{12}}^2 }{ 4 {b_{24}}^2 - \left( 1 - 2 b_{22} \right) \left( 1 - 2 b_{44} \right) } \right) } \label{eq:b_13_lower_bound}
\end{align}
hold, then
\begin{equation}
b_{12} > \frac{1}{2} \sqrt{ 1 + 2 b_{44} + 4 b_{11} b_{22} + 4 b_{33} b_{55} } . \label{eq:b_12_better_lower_bound}
\end{equation}
\end{lemma}

\begin{proof}
Note that \eqref{eq:b_13_lower_bound} is well-defined by \eqref{eq:b_24_lower_bound}. Also note that by \eqref{eq:b_12_geq_b_35}, \eqref{eq:b_12_lower_bound} and \eqref{eq:b_35_lower_bound}, \eqref{eq:b_13_better_upper_bound} is well-defined. By \eqref{eq:b_13_better_upper_bound} and \eqref{eq:b_13_lower_bound} we have
\begin{align*}
\frac{1}{2} \sqrt{ \left( 1 - 2 b_{33} \right) \left( 1 - 2 b_{11} + \frac{ 4 \left( 1 - 2 b_{44} \right) {b_{12}}^2 }{ 4 {b_{24}}^2 - \left( 1 - 2 b_{22} \right) \left( 1 - 2 b_{44} \right) } \right) } < b_{13} \qquad \\
 < \frac{1}{2} \sqrt{ \left( \frac{ 4 {b_{12}}^2 }{ 1 - 2 b_{22} } - \left( 1 - 2 b_{11} \right) \right) \left( \frac{ 4 {b_{12}}^2 }{ 1 - 2 b_{55} } - \left( 1 - 2 b_{33} \right) \right) } ,
\end{align*}
which, as $ \sum_{i=1}^{5} b_{ii} = \tr \left( B \right) = \frac{1}{2} $, is equivalent to
\begin{equation}
\frac{ 1 - 2 b_{44} }{ 4 {b_{24}}^2 - \left( 1 - 2 b_{22} \right) \left( 1 - 2 b_{44} \right) } < \frac{ 4 {b_{12}}^2 - 1 - 2 b_{44} - 4 b_{11} b_{22} - 4 b_{33} b_{55} }{ \left( 1 - 2 b_{22} \right) \left( 1 - 2 b_{33} \right) \left( 1 - 2 b_{55} \right) } . \label{eq:b_24_b_12_relation}
\end{equation}
The left-hand side of \eqref{eq:b_24_b_12_relation} is nonnegative by \eqref{eq:b_24_lower_bound}. Therefore, \eqref{eq:b_12_better_lower_bound} follows.
\end{proof}

\begin{lemma} \label{lem:b_24_better_upper_bound}
If $ \rho \left( B \right) \leq 1 $, \eqref{eq:b_12_upper_bound}, \eqref{eq:b_45_upper_bound}, \eqref{eq:b_12_lower_bound} and \eqref{eq:b_45_lower_bound} hold, then
\begin{equation}
b_{24} < \frac{ \sqrt{ \left( 3 - 2 b_{11} - 2 b_{22} \right) \left( 3 - 2 b_{44} - 2 b_{55} \right) } }{ 2 \left( 2 \sqrt{ \left( 1 - b_{11} \right) \left( 1 - b_{55} \right) } + \sqrt{ \left( 1 - 2 b_{11} \right) \left( 1 - 2 b_{55} \right) } \right) } . \label{eq:b_24_better_upper_bound}
\end{equation}
\end{lemma}

\begin{proof}
First note that \eqref{eq:b_45_upper_bound} is well-defined by \eqref{eq:b_12_upper_bound} and due to $ \rho \left( B \right) \leq 1 $, and that \eqref{eq:b_45_lower_bound} is well-defined by \eqref{eq:b_12_lower_bound} and due to Assumption~\ref{assmp:max_diagonal}. By \eqref{eq:b_45_upper_bound} and \eqref{eq:b_45_lower_bound} we get
\begin{align*}
\frac{1}{2} \sqrt{ \left( 1 - 2 b_{55} \right) \left( 1 - 2 b_{44} + \frac{ 4 \left( 1 - 2 b_{11} \right) {b_{24}}^2 }{ 4 {b_{12}}^2 - \left( 1 - 2 b_{11} \right) \left( 1 - 2 b_{22} \right) } \right) } < b_{45} \quad \\
\leq \sqrt{ \left( 1 - b_{55} \right) \left( 1 - b_{44} - \frac{ \left( 1 - b_{11} \right) {b_{24}}^2 }{ \left( 1 - b_{11}\right) \left( 1 - b_{22} \right) - {b_{12}}^2 } \right) } .
\end{align*}
Equivalently,
\begin{equation}
{b_{24}}^2 < f_4 \left( {b_{12}}^2 \right) , \label{eq:f_4_relation}
\end{equation}
where
\[ f_4 : E \to \mathbb{R} , \quad E = \left( \frac{1}{4} R_{1,2}, Q_{1,2} \right) , \]
\[ f_4 \left( x \right) = \frac{ Q_{4,5} - \frac{1}{4} R_{4,5} }{ \frac{ Q_{1,5} }{ Q_{1,2} - x } + \frac{ R_{1,5} }{ 4 x - R_{1,2} } } \]
and
\begin{align*}
Q_{i,j} &= \left( 1 - b_{ii} \right) \left( 1 - b_{jj} \right) , \\
R_{i,j} &= \left( 1 - 2 b_{ii} \right) \left( 1 - 2 b_{jj} \right) .
\end{align*}
Note that
\[ Q_{1,2} - \frac{1}{4} R_{1,2} = \frac{1}{4} \left( 3 - 2 b_{11} - 2 b_{22} \right) > 0 \]
so $E$ is well-defined. Also note that by \eqref{eq:b_12_upper_bound} and \eqref{eq:b_12_lower_bound} $ {b_{12}}^2 \in E $.

We have
\[ \frac{d}{dx} f_4 \left( x \right) = - \frac{ \left( Q_{4,5} - \frac{1}{4} R_{4,5} \right) \left( \frac{ Q_{1,5} }{ \left( Q_{1,2} - x \right)^2 } - \frac{ 4 R_{1,5} }{ \left( 4 x - R_{1,2} \right)^2 } \right) }{ \left( \frac{ Q_{1,5} }{ Q_{1,2} - x } + \frac{ R_{1,5} }{ 4 x - R_{1,2} } \right)^2 } , \]
which is well-defined for $ x \in E $. Note that $ Q_{4,5} - \frac{1}{4} R_{4,5} > 0 $. As $ \frac{ Q_{1,5} }{ \left( Q_{1,2} - x \right)^2 } $ is strictly increasing and $ \frac{ 4 R_{1,5} }{ \left( 4 x - R_{1,2} \right)^2 } $ is strictly decreasing in $ x \in E $, and as
\[ \lim_{ x \to Q_{1,2} } \frac{ Q_{1,5} }{ \left( Q_{1,2} - x \right)^2 } = \lim_{ x \to \frac{1}{4} R_{1,2} } \frac{ 4 R_{1,5} }{ \left( 4 x - R_{1,2} \right)^2 } = \infty , \]
there is exactly one $ x_0 \in E $ for which $ \frac{d}{dx} f_4 \left( x_0 \right) = 0 $. Also, $ f_4 \left( x \right) $ is increasing when $ x \in \left( \frac{1}{4} R_{1,2} , x_0 \right) $ and decreasing when $ x \in \left( x_0, Q_{1,2} \right) $, and $ f_4 \left( x \right) \leq f_4 \left( x_0 \right) $ for every $ x \in E $. In particular, $ f_4 \left( {b_{12}}^2 \right) \leq f_4 \left( x_0 \right) $. As we are looking for $ x_0 \in E $, a solution of
\[ \frac{ Q_{1,5} }{ \left( Q_{1,2} - x_0 \right)^2 } = \frac{ 4 R_{1,5} }{ \left( 4 x_0 - R_{1,2} \right)^2 } \]
is also a solution of
\[ \frac{ \sqrt{ Q_{1,5} } }{ Q_{1,2} - x_0 } = \frac{ 2 \sqrt{ R_{1,5} } }{ 4 x_0 - R_{1,2} } . \]
Solving, we get
\[ x_0 = \frac{ \sqrt{ Q_{1,5} } R_{1,2} + 2 \sqrt{ R_{1,5} } Q_{1,2} }{ 4 \sqrt{ Q_{1,5} } + 2 \sqrt{ R_{1,5} } } \]
and
\begin{equation}
f_4 \left( x_0 \right) = \frac{ \left( 4 Q_{4,5} - R_{4,5} \right) \left( 4 Q_{1,2} - R_{1,2} \right) }{ 4 \left( 2 \sqrt{ Q_{1,5} } + \sqrt{ R_{1,5} } \right)^2 } \label{eq:max_f_4} .
\end{equation}
By \eqref{eq:f_4_relation} and \eqref{eq:max_f_4} we get \eqref{eq:b_24_better_upper_bound}.
\end{proof}

\begin{lemma} \label{lem:b_12_better_upper_bound}
If $ \rho \left( B \right) \leq 1 $, \eqref{eq:b_24_upper_bound}, \eqref{eq:b_13_upper_bound}, \eqref{eq:b_24_lower_bound} and \eqref{eq:b_13_lower_bound} hold, then
\begin{equation}
b_{12} < \frac{ \sqrt{ \left( 3 - 2 b_{11} - 2 b_{33} \right) \left( 3 - 2 b_{22} - 2 b_{44} \right) } }{ 2 \left( 2 \sqrt{ \left( 1 - b_{33} \right) \left( 1 - b_{44} \right) } + \sqrt{ \left( 1 - 2 b_{33} \right) \left( 1 - 2 b_{44} \right) } \right) } . \label{eq:b_12_better_upper_bound}
\end{equation}
\end{lemma}

\begin{proof}
We obtain the result from Lemma~\ref{lem:b_24_better_upper_bound} in the following way:

We formally replace the symbols $ b_{44} $, $ b_{24} $, $ b_{12} $, $ b_{11} $, $ b_{13} $, $ b_{33} $ by $ b_{11} $, $ b_{12} $, $ b_{24} $, $ b_{44} $, $ b_{45} $, $ b_{55} $ respectively (keeping $ b_{22} $ unchanged) into \eqref{eq:b_24_upper_bound}, \eqref{eq:b_13_upper_bound}, \eqref{eq:b_24_lower_bound} and \eqref{eq:b_13_lower_bound}, which are the assumptions of this Lemma. We then get respectively \eqref{eq:b_12_upper_bound}, \eqref{eq:b_45_upper_bound}, \eqref{eq:b_12_lower_bound} and \eqref{eq:b_45_lower_bound}, which are the assumptions of Lemma~\ref{lem:b_24_better_upper_bound}. Therefore, we get \eqref{eq:b_24_better_upper_bound}. By doing the opposite replacement we get \eqref{eq:b_12_better_upper_bound}.
\end{proof}

\begin{theorem} \label{th:pattern_C}
If $ \rho \left( B \right) \leq 1 $ then $ \lambda_3 \leq \frac{1}{2} $
\end{theorem}

\begin{proof}
By Lemma~\ref{lem:not_similar_to_positive_C} we have \eqref{eq:b_12_geq_b_35}, \eqref{eq:b_24_geq_b_35}. By Lemma~\ref{lem:b_12_and_b_45_conditions} we have \eqref{eq:b_12_upper_bound} and \eqref{eq:b_45_upper_bound}. Also, if \eqref{eq:b_12_sufficient_upper_bound} holds then we are done. Otherwise, we have \eqref{eq:b_12_lower_bound}. If \eqref{eq:b_45_sufficient_upper_bound} holds we are done. Otherwise, we have \eqref{eq:b_45_lower_bound}. By Lemma~\ref{lem:b_24_and_b_13_conditions} we have \eqref{eq:b_24_upper_bound} and \eqref{eq:b_13_upper_bound}. Also, if \eqref{eq:b_24_sufficient_upper_bound} holds then we are done. Otherwise, we have \eqref{eq:b_24_lower_bound}. If \eqref{eq:b_13_sufficient_upper_bound} holds we are done. Otherwise, we have \eqref{eq:b_13_lower_bound}. If the conditions of Lemma~\ref{lem:b_35_sufficient_condition_2345} hold or the conditions of Lemma~\ref{lem:b_35_sufficient_condition_1235} hold we are done. Otherwise, we have \eqref{eq:b_35_lower_bound}, \eqref{eq:b_35_lower_bound_2345} and \eqref{eq:b_35_lower_bound_1235}. By Lemma~\ref{lem:b_45_better_upper_bound} we have \eqref{eq:b_45_better_upper_bound} and by Lemma~\ref{lem:b_13_better_upper_bound} we have \eqref{eq:b_13_better_upper_bound}. By Lemma~\ref{lem:b_24_better_lower_bound} we have \eqref{eq:b_24_better_lower_bound} and by Lemma~\ref{lem:b_12_better_lower_bound} we have \eqref{eq:b_12_better_lower_bound}. By Lemma~\ref{lem:b_24_better_upper_bound} we have \eqref{eq:b_24_better_upper_bound} and by Lemma~\ref{lem:b_12_better_upper_bound} we have \eqref{eq:b_12_better_upper_bound}. If $ b_{33} \geq \frac{26}{100} $ or $ b_{55} \geq \frac{26}{100} $ then by Appendix~\ref{app:b_33_or_b_55_geq_0.26} we get a contradiction. Else, by Appendix~\ref{app:b_33_and_b_55_leq_0.26} we have $ \rho \left( B \right) > 1 $, which contradicts our assumption.
\end{proof}

%% file: new_necessary_condition_appendices.tex
\section{Proving $ x_{min} \left( s, t \right) \geq q_2^- $} \label{app:xmin_geq_q2}

\begin{lemma}
Let $ 0 \leq s < \frac{1}{2} $, $ 0 \leq t \leq \frac{1}{4} - \frac{1}{2} s $, $ t + s > 0 $ and let $ x_{min} \left( s, t \right) $ and $ q_2^- $ be as defined in Lemma~\ref{lem:spectral_radius_of_A[1,2,3,5]}. Then $ x_{min} \left( s, t \right) \geq q_2^- $.
\end{lemma}

\begin{proof}
We first note that by the proof of \eqref{eq:a_45_better_lower_bound} in Lemma~\ref{lem:a_45_better_lower_bound} we know that
\[ \frac{ \left( 1 - 2 s \right) \left( 3 - 2 s \right) \left( t + s \right) }{ 1 + t + s } \leq 1 . \]
As $ \sqrt{ 1 - \alpha } \leq 1 - \frac{1}{2} \alpha $ for $ 0 \leq \alpha \leq 1 $ we get
\[ 1 - \frac{1}{2} \frac{ \left( 1 - 2 s \right) \left( 3 - 2 s \right) \left( t + s \right) }{ 1 + t + s } \geq \sqrt{ 1 - \frac{ \left( 1 - 2 s \right) \left( 3 - 2 s \right) \left( t + s \right) }{ 1 + t + s } } . \]
Hence,
\begin{align}
x_{min} \left( s, t \right) &= \frac{1}{2} \left( 1 - s \right) - \frac{1}{4} \sqrt{ 1 - \frac{ \left( 1 - 2 s \right) \left( 3 - 2 s \right) \left( t + s \right) }{ 1 + t + s } } \nonumber \\
 &\geq \frac{1}{2} \left( 1 - s \right) - \frac{1}{4} \left( 1 - \frac{1}{2} \frac{ \left( 1 - 2 s \right) \left( 3 - 2 s \right) \left( t + s \right) }{ 1 + t + s } \right) \nonumber \\
 &= \frac{ \left( 1 - 2 s \right) \left( \left( 5 - 2 s \right) \left( t + s \right) + 2 \right) }{ 8 \left( 1 + t + s \right) } . \label{eq:xmin_geq_xmin2}
\end{align}

Our goal is to show that
\begin{equation}
\frac{ \left( 1 - 2 s \right) \left( \left( 5 - 2 s \right) \left( t + s \right) + 2 \right) }{ 8 \left( 1 + t + s \right) } \geq q_2^- . \label{eq:xmin2_geq_q2}
\end{equation}
This is equivalent to showing that
\begin{equation}
8 \left( 1 - 2 t \right) \sqrt{ \left( 3 - 2 t \right) \left( 1 + s \right) \left( t + s \right)^3 } \geq \frac{ \left( t + s \right) h_1 \left( s, t \right) }{ 2 \left( 1 + t + s \right) } , \label{eq:xmin2_geq_q2_alt1}
\end{equation}
where
\begin{align*}
 h_1 \left( s, t \right) &= \frac{ 2 \left( 1 + t + s \right) }{ \left( t + s \right) } \left( C_5 - \frac{ \left( 3 + 2 s \right) \left( \left( 5 - 2 s \right) \left( t + s \right) + 2 \right) }{ 2 \left( 1 + t + s \right) } C_4 \right) \\
 &= 64 s^2 t^2 + 64 t s^3 - 24 s^3 + 32 t^3 - 16 t s^2 + 48 s t^2 \\
 & \quad + 12 s^2 - 152 t^2 - 168 t s + 78 s + 56 t - 3 .
\end{align*}
Note that $ h_1 $ is defined as a rational function in $s$ and $t$, but the terms $ t + s $ and $ 1 + t + s $ are both positive and appear in the numerator and denominator of $ h_1 $, so can be reduced.

If $ h_1 \left( s, t \right) < 0 $ then obviously \eqref{eq:xmin2_geq_q2} holds, so we want to learn more about the sign of $ h_1 $. We have
\begin{align*}
 \frac{ \partial^2 }{ \partial t^2 } h_1 \left( s, t \right) &= 128 s^2 + 96 s + 192 t - 304 \\
 &\leq 128 s^2 + 96 s + 192 \cdot \left( \frac{1}{4} - \frac{1}{2} s \right) - 304 \\
 &= 128 s^2 - 256 < 128 \cdot \frac{1}{4} - 256 = -224 < 0 ,
\end{align*}
so $ h_1 $ is concave in $t$ (for any fixed $s$ in its domain). Therefore,
\[ \min_{ 0 \leq t \leq \frac{1}{4} - \frac{1}{2} s } h_1 \left( s, t \right) = \min \left\{ h_1 \left( s, 0 \right), h_1 \left( s, \frac{1}{4} - \frac{1}{2} s \right) \right\} . \]
The roots of
\[ h_1 \left( s, 0 \right) = -24 s^3 + 12 s^2 + 78 s - 3 \]
are $ -1.591478567, 0.03825363319, 2.053224934 $. Let $ s_0 $ be the second largest root. Then $ h_1 \left( s, 0 \right) $ is negative when $ 0 \leq s < s_0 $ and nonnegative when $ s_0 \leq s < \frac{1}{2} $. The roots of
\[ h_1 \left( s, \frac{1}{4} - \frac{1}{2} s \right) = -2 \left( 1 + s \right) \left( 8 s^3 - 4 s^2 - 22 s - 1 \right) \]
are $ -1.400220700, -1, -0.04587223942, 1.946092939 $, so $ h_1 \left( s, \frac{1}{4} - \frac{1}{2} s \right) > 0 $ for $ 0 \leq s < \frac{1}{2} $. Therefore, for $ s_0 < s < \frac{1}{2} $ and $ 0 \leq t \leq \frac{1}{4} - \frac{1}{2} s $ we have $ h_1 \left( s, t \right) > 0 $.

Another observation is that if $ h_1 \left( s, 0 \right) < 0 $, then as $ h_1 $ is concave in $t$ and as $ h_1 \left( s, \frac{1}{4} - \frac{1}{2} s \right) > 0 $ there must be a value $ t_s $ with the following properties:
\begin{enumerate}
\item $ h_1 \left( s, t_s \right) = 0 $,
\item $ h_1 \left( s, t \right) < 0 $ for $ 0 \leq t < t_s $,
\item $ h_1 \left( s, t \right) > 0 $ for $ t_s < t \leq \frac{1}{4} - \frac{1}{2} s $,
\item $ h_1 $ is monotonically increasing in $t$ for $ 0 \leq t \leq t_s $.
\end{enumerate}
This means that if $ h_1 \left( s, t \right) < 0 $ then $ h_1 \left( s, \tilde{t} \right) < 0 $ for any $ 0 \leq \tilde{t} \leq t $.

If $ 0 \leq s \leq s_0 $ then, as $ t \leq \frac{1}{4} - \frac{1}{2} s < \frac{3}{8} $, we have
\begin{align*}
 \frac{ \partial^2 }{ \partial s^2 } h_1 \left( s, t \right) &= 128 t^2 - 32 t + 48 \left( 8 t - 3 \right) s + 24 \\
 &\geq 128 t^2 - 32 t + 48 \left( 8 t - 3 \right) s_0 + 24 \\
 &> 128 t^2 - 32 t + 48 \left( 8 t - 3 \right) \cdot \frac{4}{100} + 24 = 128 t^2 - \frac{416}{25} t + \frac{456}{25} > 0 .
\end{align*}
Hence, $ h_1 $ is convex in $s$ when $ 0 \leq s \leq s_0 $. Also, as $ 0 \leq t \leq \frac{1}{4} $,
\[ \left. \frac{ \partial }{ \partial s } h_1 \left( s, t \right) \right|_{ s = 0 } = 48 t^2 - 168 t + 78 \geq 0 - 168 \cdot \frac{1}{4} + 78 = 36 > 0 \]
and so $ h_1 $ is monotonically increasing in $s$ when $ 0 \leq s \leq s_0 $. Therefore,
\[ \min_{ 0 \leq s \leq s_0 } h_1 \left( s, t \right) = h_1 \left( 0, t \right) = 32 t^3 - 152 t^2 + 56 t - 3 , \]
whose roots are $ 0.06482035236, 0.3322609755, 4.352918672 $. Let $ t_0 $ be the smallest root. Therefore, for $ 0 \leq s \leq s_0 $ and $ t_0 < t \leq \frac{1}{4} - \frac{1}{2} s $ we have $ h_1 \left( s, t \right) > 0 $. We conclude that $ h_1 \left( s, t \right) $ can be negative only when $ 0 \leq s < s_0 $ and $ 0 \leq t < t_0 $. Also, if $ h_1 \left( s, t \right) < 0 $ then for any $ 0 \leq \tilde{s} \leq s $ and $ 0 \leq \tilde{t} \leq t $ we have $ h_1 \left( \tilde{s}, \tilde{t} \right) < 0 $.

We now check when
\begin{equation}
8 \left( 1 - 2 t \right) \sqrt{ \left( 3 - 2 t \right) \left( 1 + s \right) \left( t + s \right)^3 } \geq \left| \frac{ \left( t + s \right) h_1 \left( s, t \right) }{ 2 \left( 1 + t + s \right) } \right| \label{eq:xmin2_geq_q2_alt2}
\end{equation}
holds. Note that if \eqref{eq:xmin2_geq_q2_alt2} holds then \eqref{eq:xmin2_geq_q2_alt1} holds as well, and that \eqref{eq:xmin2_geq_q2_alt2} is equivalent to
\begin{align*}
 0 &\leq \left( 8 \left( 1 - 2 t \right) \sqrt{ \left( 3 - 2 t \right) \left( 1 + s \right) \left( t + s \right)^3 } \right)^2 - \left( \frac{ \left( t + s \right) h_1 \left( s, t \right) }{ 2 \left( 1 + t + s \right) } \right)^2 \\
 &= - \frac{ \left( 3 + 2 s \right) \left( t + s \right)^2 C_4 h_2 \left( s, t \right) }{ 4 \left( 1 + t + s \right)^2 } ,
\end{align*}
where
\begin{align*}
 h_2 \left( s, t \right) &= 128 s^4 t + 128 t^2 s^3 - 48 s^4 - 192 t s^3 + 64 s^2 t^2 + 256 s t^3 + 64 t^4 \\
 & \quad + 96 s^3 - 512 t s^2 - 736 s t^2 - 192 t^3 + 232 s^2 + 528 t s + 352 t^2 \\
 & \quad - 136 s - 120 t + 1 .
\end{align*}
In Lemma~\ref{lem:spectral_radius_of_A[1,2,3,5]} we showed that $ C_4 > 0 $. If $ h_2 \left( s, t \right) < 0 $ then obviously \eqref{eq:xmin2_geq_q2_alt2} holds, so we want to learn more about the sign of $ h_2 $. We have
\begin{align*}
 \frac{ \partial^2 }{ \partial t^2 } h_2 \left( s, t \right) &= 256 s^3 + 128 s^2 + 1536 t s + 768 t^2 - 1472 s - 1152 t + 704 , \\
 \frac{ \partial^3 }{ \partial t^3 } h_2 \left( s, t \right) &= 1536 t + 1536 s - 1152 \\
 &\leq 1536 \left( \frac{1}{4} - \frac{1}{2} s \right) + 1536 s - 1152 = -768 \left( 1 - s \right) < 0 .
\end{align*}
Therefore,
\[ \min_{ 0 \leq t \leq \frac{1}{4} - \frac{1}{2} s } \frac{ \partial^2 }{ \partial t^2 } h_2 \left( s, t \right) = \left. \frac{ \partial^2 }{ \partial t^2 } h_2 \left( s, t \right) \right|_{ t = \frac{1}{4} - \frac{1}{2} s } = 256 s^3 - 448 s^2 - 704 s + 464 , \]
which has roots $ -1.333013968, 0.5332693631, 2.549744605 $, and so is positive for $ 0 \leq s < \frac{1}{2} $. We conclude that $ h_2 $ is convex in $t$ (for any fixed $s$ in its domain), which means that
\[ \max_{ 0 \leq t \leq \frac{1}{4} - \frac{1}{2} s } h_2 \left( s, t \right) = \max \left\{ h_2 \left( s, 0 \right), h_2 \left( s, \frac{1}{4} - \frac{1}{2} s \right) \right\} . \]
The roots of
\[ h_2 \left( s, 0 \right) = \left( 1 - 2 s \right) \left( 24 s^3 - 36 s^2 - 134 s + 1 \right) \]
are $ -1.733921023, 0.007447858016, 0.5, 3.226473165 $. Let $ s_1 $ be the second smallest root. Then $ h_2 \left( s, 0 \right) $ is negative when $ s_1 < s < \frac{1}{2} $ and nonnegative when $ 0 \leq s \leq s_1 $. The roots of
\[ h_2 \left( s, \frac{1}{4} - \frac{1}{2} s \right) = \frac{1}{4} \left( 1 - 2 s \right) \left( 16 s^2 + 22 s + 3 \right) \left( 4 s^2 - 8 s - 13 \right) \]
are $ -1.221500234, -1.061552813, -0.1534997659, 0.5, 3.061552813 $. This means that $ h_2 \left( s, \frac{1}{4} - \frac{1}{2} s \right) < 0 $ for $ 0 \leq s < \frac{1}{2} $. Hence, for $ s_1 < s < \frac{1}{2} $ and $ 0 \leq t \leq \frac{1}{4} - \frac{1}{2} s $ we have $ h_2 \left( s, t \right) < 0 $, so \eqref{eq:xmin2_geq_q2} is met.

Also, as $ 0 \leq s < \frac{1}{2} $, we have
\begin{align*}
 \frac{ \partial^2 }{ \partial s^2 } h_2 \left( s, t \right) &= \left( 128 + 768 s \right) t^2 + \left( 1536 s^2 - 1152 s -1024 \right) t \\
& \quad - 576 s^2 + 576 s + 464 , \\
 \frac{ \partial }{ \partial t } \frac{ \partial^2 }{ \partial s^2 } h_2 \left( s, t \right) &= \left( 256 + 1536 s \right) t + 1536 s^2 - 1152 s - 1024 \\
 &\leq \left( 256 + 1536 s \right) \left( \frac{1}{4} - \frac{1}{2} s \right) + 1536 s^2 - 1152 s - 1024 \\
 &= 768 s^2 - 896 s - 960 \\
 &\leq 768 \cdot \left( \frac{1}{2} \right)^2 - 0 - 960 = -768 < 0 .
\end{align*}
Therefore,
\[ \min_{ 0 \leq t \leq \frac{1}{4} - \frac{1}{2} s } \frac{ \partial^2 }{ \partial s^2 } h_2 \left( s, t \right) = \left. \frac{ \partial^2 }{ \partial s^2 } h_2 \left( s, t \right) \right|_{ t = \frac{1}{4} - \frac{1}{2} s } = 8 \left( 3 - 2 s \right) \left( 36 s^2 + 40 s + 9 \right) , \]
which is positive when $ 0 \leq s < \frac{1}{2} $. We conclude that $ h_2 $ is convex in $s$ (for any fixed $t$ in its domain), which means that
\[ \max_{ 0 \leq s \leq \frac{1}{2} - 2 t } h_2 \left( s, t \right) = \max \left\{ h_2 \left( 0, t \right), h_2 \left( \frac{1}{2} - 2 t, t \right) \right\} . \]
The real roots of
\[ h_2 \left( 0, t \right) = 64 t^4 - 192 t^3 + 352 t^2 - 120 t + 1 \]
are $ 0.008546600862, 0.4150148497 $. Let $ t_1 $ be the smaller root. Then $ h_2 \left( 0, t \right) $ is negative when $ t_1 < t \leq \frac{1}{4} $ and nonnegative when $ 0 \leq t \leq t_1 $. The roots of
\[ h_2 \left( \frac{1}{2} - 2 t, t \right) = 16 t \left( 2 t^2 + t - 2 \right) \left( 32 t^2 - 38 t + 9 \right) \]
are $ -1.280776406, 0, 0.3267498830, 0.7807764064, 0.8607501170 $. This means we have $ h_2 \left( \frac{1}{2} - 2 t, t \right) < 0 $ for $ 0 < t \leq \frac{1}{4} $. Hence, for $ 0 \leq s \leq s_1 $ and $ t_1 < t \leq \frac{1}{4} - \frac{1}{2} s $ we have $ h_2 \left( s, t \right) < 0 $, so \eqref{eq:xmin2_geq_q2} is met.

Finally, as $ h_1 \left( \frac{1}{100}, \frac{1}{100} \right) = - \frac{5,283,621}{3,125,000} < 0 $, then for $ 0 \leq s \leq s_1 < \frac{1}{100} $ and $ 0 \leq t \leq t_1 < \frac{1}{100} $ we have $ h_1
\left( s, t \right) < 0 $, so \eqref{eq:xmin2_geq_q2} is met. Therefore, \eqref{eq:xmin2_geq_q2} holds for $ 0 \leq s < \frac{1}{2} $ and $ 0 \leq t \leq \frac{1}{4} - \frac{1}{2} s $. Together with \eqref{eq:xmin_geq_xmin2} we have $ x_{min} \left( s, t \right) \geq q_2^- $.
\end{proof}

\section{Proving $ x_{max} \left( s, t \right) \leq q_2^+ $} \label{app:q2_qeq_xmax}

\begin{lemma}
Let $ 0 \leq s < \frac{1}{2} $, $ 0 \leq t \leq \frac{1}{4} - \frac{1}{2} s $, $ t + s > 0 $ and let $ x_{max} \left( s, t \right) $ and $ q_2^+ $ be as defined in Lemma~\ref{lem:spectral_radius_of_A[1,2,3,5]}. Then $ x_{max} \left( s, t \right) \leq q_2^+ $.
\end{lemma}

\begin{proof}
First we want to show that
\begin{equation}
x_{max} \left( s, t \right) = \frac{ \left( 3 - 2 s \right) \left( 5 - 6 t - 4 \sqrt{ \left( 1 - t \right) \left( 1 - 2 t \right) } \right) }{ 4 \left( 3 - 2 t \right) } \leq \frac{1}{4} + \frac{1}{3} t - \frac{1}{6} s . \label{eq:xmax_leq_xmax2}
\end{equation}
This is equivalent to showing that
\begin{equation}
3 \left( 3 - 2 s \right) \sqrt{ \left( 1 - t \right) \left( 1 - 2 t \right) } \geq h_3 \left( s, t \right) , \label{eq:xmax_leq_xmax2_alt}
\end{equation}
where
\[ h_3 \left( s, t \right) = - \left( 6 - 8 t \right) s + 2 t^2 - 15 t + 9 . \]
As the coefficient of $s$ in $ h_3 $ is negative and $ 0 \leq t \leq \frac{1}{4} $ we have
\begin{align*}
h_3 \left( s, t \right) &\geq - \left( 6 - 8 t \right) \left( \frac{1}{2} - 2 t \right) + 2 t^2 - 15 t + 9 = -14 t^2 + t + 6 \\
 &\geq -14 \cdot \left( \frac{1}{4} \right)^2 + 0 + 6 = \frac{41}{8} > 0 .
\end{align*}
Therefore, \eqref{eq:xmax_leq_xmax2_alt} is equivalent to
\[ 0 \leq \left( 3 \left( 3 - 2 s \right) \sqrt{ \left( 1 - t \right) \left( 1 - 2 t \right) } \right)^2 - \left( h_3 \left( s, t \right) \right)^2 = - t \left( 3 - 2 t \right) h_4 \left( s, t \right) , \]
where
\[ h_4 \left( s, t \right) = 4 s^2 - 16 t s - 2 t^2 + 27 t - 9 . \]
As $ 0 \leq s < \frac{1}{2} $ and $ 0 \leq t \leq \frac{1}{4} $, we have
\[ h_4 \left( s, t \right) \leq 4 s^2 + 27 t - 9 < 4 \cdot \left( \frac{1}{2} \right)^2 + 27 \cdot \frac{1}{4} - 9 = -\frac{5}{4} < 0 , \]
so \eqref{eq:xmax_leq_xmax2} follows.

Our next goal is to show that
\begin{equation}
\frac{1}{4} + \frac{1}{3} t - \frac{1}{6} s \leq q_2^+ . \label{eq:xmax2_leq_q2}
\end{equation}
This is equivalent to showing that
\begin{equation}
6 \left( 1 - 2 t \right) \left( 1 - 2 s \right) \sqrt{ \left( 3 - 2 t \right) \left( 1 + s \right) \left( t + s \right) } \geq h_5 \left( s, t \right) , \label{eq:xmax2_leq_q2_alt1}
\end{equation}
where
\begin{align*}
h_5 \left( s, t \right) &= \frac{ 3 \left( 1 - 2 s \right) }{ 4 \left( t + s \right) } \left( \frac{ \left( 3 + 2 s \right) \left( 3 + 4 t - 2 s \right) }{ 3 \left( 1 - 2 s \right) } C_4 - C_5 \right) \\
 &= 44 t s^2 + 56 s t^2 - 12 s^2 + 36 t^2 - 4 t s + 6 s - 9 t .
\end{align*}
Instead, we check when
\begin{equation}
6 \left( 1 - 2 t \right) \left( 1 - 2 s \right) \sqrt{ \left( 3 - 2 t \right) \left( 1 + s \right) \left( t + s \right) } \geq \left| h_5 \left( s, t \right) \right| \label{eq:xmax2_leq_q2_alt2}
\end{equation}
holds. Note that if \eqref{eq:xmax2_leq_q2_alt2} holds then \eqref{eq:xmax2_leq_q2_alt1} holds as well, and that \eqref{eq:xmax2_leq_q2_alt2} is equivalent to
\begin{align*}
 0 &\leq \left( 6 \left( 1 - 2 t \right) \left( 1 - 2 s \right) \sqrt{ \left( 3 - 2 t \right) \left( 1 + s \right) \left( t + s \right) } \right)^2 - \left( h_5 \left( s, t \right) \right)^2 \\
 &= - \left( 3 + 2 s \right) C_4 h_6 \left( s, t \right) ,
\end{align*}
where
\[ h_6 \left( s, t \right) = \left( 36 s^2 + 44 s + 33 \right) t^2 + \left( - 16 s^2 + 32 s - 12 \right) t + 24 s^2 - 12 s . \]
The coefficient of $ t^2 $ in $ h_6 $ is positive so $ h_6 $ is convex in $t$ (for any fixed $s$ in its domain). Therefore,
\[ \max_{ 0 \leq t \leq \frac{1}{4} - \frac{1}{2} s } h_6 \left( s, t \right) = \max \left\{ h_6 \left( s, 0 \right), h_6 \left( s, \frac{1}{4} - \frac{1}{2} s \right) \right\} . \]
We have
\[ h_6 \left( s, 0 \right) = -12 s \left( 1 - 2 s \right) < 0 \]
and
\[ h_6 \left( s, \frac{1}{4} - \frac{1}{2} s \right) = - \frac{1}{16} \left( 1 - 2 s \right) \left( 72 s^3 + 116 s^2 + 86 s + 15 \right) < 0 , \]
so $ h_6 \left( s, t \right) < 0 $ for $ 0 \leq s < \frac{1}{2} $ and $ 0 \leq t \leq \frac{1}{4} - \frac{1}{2} s $. As we already know that $ C_4 > 0 $, \eqref{eq:xmax2_leq_q2} follows and together with \eqref{eq:xmax_leq_xmax2} we have $ x_{max} \left( s, t \right) \leq q_2^+ $.
\end{proof}

\section{$ b_{33} \geq \frac{26}{100} $ or $ b_{55} \geq \frac{26}{100} $} \label{app:b_33_or_b_55_geq_0.26}

\begin{lemma} \label{lem:b_24_contradiction}
If $ b_{33} \geq \frac{26}{100} $ then \eqref{eq:b_24_better_lower_bound} and \eqref{eq:b_24_better_upper_bound} cannot both hold.
\end{lemma}

\begin{proof}
Assume $ b_{33} \geq \frac{26}{100} $ and that \eqref{eq:b_24_better_lower_bound} and \eqref{eq:b_24_better_upper_bound} both hold. Let
\[ F = \left\{ \mathbf{x} = \left( x_1, x_2, x_3, x_4, x_5 \right) \in {\mathbb{R}}^5 \mid \mathbf{x} \geq 0, x_3 \geq \frac{26}{100}, \sum_{i=1}^{5} x_i = \frac{1}{2} \right\} \]
and note that $ \left( b_{11}, b_{22}, b_{33}, b_{44}, b_{55} \right) \in F $.

We find a lower bound for the denominator of \eqref{eq:b_24_better_upper_bound}. Note that $ 1 - \frac{39}{25} x \leq \sqrt{ \left( 1 - x \right) \left( 1 - 2 x \right) } $ for $ 0 \leq x \leq \frac{1}{4} $. As $ b_{33} \geq \frac{26}{100} $ we must have $ b_{11} \leq \frac{24}{100} \leq \frac{1}{4} $ and $ b_{55} \leq \frac{24}{100} \leq \frac{1}{4} $. Therefore,
\begin{align*}
& \left( 2 \sqrt{ \left( 1 - b_{11} \right) \left( 1 - b_{55} \right) } + \sqrt{ \left( 1 - 2 b_{11} \right) \left( 1 - 2 b_{55} \right) } \right)^2 \\
& \qquad = 4 \left( 1 - b_{11} \right) \left( 1 - b_{55} \right) + \left( 1 - 2 b_{11} \right) \left( 1 - 2 b_{55} \right) \\
& \qquad \quad + 4 \sqrt{ \left( 1 - b_{11} \right) \left( 1 - 2 b_{11} \right) } \sqrt{ \left( 1 - b_{55} \right) \left( 1 - 2 b_{55} \right) } \\
& \qquad \geq 4 \left( 1 - b_{11} \right) \left( 1 - b_{55} \right) + \left( 1 - 2 b_{11} \right) \left( 1 - 2 b_{55} \right) \\
& \qquad \quad + 4 \left( 1 - \frac{39}{25} b_{11} \right) \left( 1 - \frac{39}{25} b_{55} \right) \\
& \qquad = \frac{1}{625} \left( 5625 - 7650 b_{11} - 7650 b_{55} + 11084 b_{11} b_{55} \right) > 0
\end{align*}
and so by \eqref{eq:b_24_better_upper_bound}
\begin{align}
{b_{24}}^2 &< \frac{ \left( 3 - 2 b_{11} - 2 b_{22} \right) \left( 3 - 2 b_{44} - 2 b_{55} \right) }{ 4 \left( 2 \sqrt{ \left( 1 - b_{11} \right) \left( 1 - b_{55} \right) } + \sqrt{ \left( 1 - 2 b_{11} \right) \left( 1 - 2 b_{55} \right) } \right)^2 } \nonumber \\
& \leq h_7 \left( b_{11}, b_{22}, b_{33}, b_{44}, b_{55} \right) , \label{eq:b_24_sq_leq_h_7}
\end{align}
where $ h_7 : F \to \mathbb{R} $ is defined by
\[ h_7 \left( x, y, z, u, v \right) = \frac{ 625 \left( 3 - 2 x - 2 y \right) \left( 3 - 2 u - 2 v \right) }{ 4 \left( 5625 - 7650 x - 7650 v + 11084 x v \right) } . \]

By \eqref{eq:b_24_better_lower_bound} and as $ \sum_{i=1}^{5} b_{ii} = \tr \left( B \right) = \frac{1}{2} $
\begin{align}
{b_{24}}^2 > \frac{1}{4} \left( 1 + 2 b_{11} + 4 b_{22} b_{44} + 4 b_{33} b_{55} \right) = h_8 \left( b_{11}, b_{22}, b_{33}, b_{44}, b_{55} \right) , \label{eq:b_24_sq_geq_h_8}
\end{align}
where $ h_8 : F \to \mathbb{R} $ is defined by
\[ h_8 \left( x, y, z, u, v \right) = \frac{1}{4} \left( \left( 1 - 2 y \right) \left( 1 - 2 u \right) + \left( 1 - 2 z \right) \left( 1 - 2 v \right) \right) . \]

By \eqref{eq:b_24_sq_leq_h_7} and \eqref{eq:b_24_sq_geq_h_8} we get that
\[ h_7 \left( b_{11}, b_{22}, b_{33}, b_{44}, b_{55} \right) > h_8 \left( b_{11}, b_{22}, b_{33}, b_{44}, b_{55} \right) . \]
This is equivalent to
\begin{equation}
h_9 \left( b_{11}, b_{22}, b_{33}, b_{44}, b_{55} \right) > 0 , \label{eq:h_9_positive}
\end{equation}
where $ h_9 : F \to \mathbb{R} $ is defined by
\begin{align*}
& h_9 \left( x, y, z, u, v \right) = 625 \left( 3 - 2 x - 2 y \right) \left( 3 - 2 u - 2 v \right) \\
 & \quad - \left( \left( 1 - 2 y \right) \left( 1 - 2 u \right) + \left( 1 - 2 z \right) \left( 1 - 2 v \right) \right) \left( 5625 - 7650 x - 7650 v + 11084 x v \right) .
\end{align*}
By the definition of $F$ we can replace $v$ by $ \frac{1}{2} - x - y - z - u $. Note that
\begin{align*}
\frac{ \partial^2 }{ \partial x^2 } h_9 & \left( x, y, z, u, \frac{1}{2} - x - y - z - u \right) = 17168 + 133008 \left( 1 - 2 z \right) x \\
 & \quad + 88672 y u + 22168 \left( 4 z - 1 \right) \left( 1 - 2 \left( y + z + u \right) \right) \\
 & \geq 17168 + 0 + 0 + 22168 \left( 4 \cdot \frac{26}{100} - 1 \right) \left( 1 - 2 \cdot \frac{1}{2} \right) \\
 & = 17168 > 0 ,
\end{align*}
and so $ h_9 \left( x, y, z, u, \frac{1}{2} - x - y - z - u \right) $ is convex in $x$ (keeping the other variables fixed). Further note that
\begin{align*}
\frac{ \partial^2 }{ \partial u^2 } h_9 & \left( x, y, z, u, \frac{1}{2} - x - y - z - u \right) = 272 \left( 225 - 326 x \right) \left( z - y \right) \\
 &\geq 272 \left( 225 - 326 \cdot \frac{24}{100} \right) \left( \frac{26}{100} - \frac{24}{100} \right) = \frac{498984}{625} > 0
\end{align*}
and so $ h_9 \left( x, y, z, u, \frac{1}{2} - x - y - z - u \right) $ is convex in $u$ (keeping the other variables fixed). Moreover,
\begin{align*}
\frac{ \partial^2 }{ \partial u^2 } h_9 \left( \frac{1}{2} - y - z - u, y, z, u, 0 \right) &= 400 \left( 64 - 153 y \right) \\
 &\geq 400 \left( 64 - 153 \cdot \frac{24}{100} \right) = 10912 > 0
\end{align*}
and so $ h_9 \left( \frac{1}{2} - y - z - u, y, z, u, 0 \right) $ is convex in $u$ (keeping the other variables fixed). Therefore,
\begin{align*}
\max_F & h_9 \left( x, y, z, u, v \right) = \max_{ y, z, u } \max_{ 0 \leq x \leq \frac{1}{2} - y - z - u } h_9 \left( x, y, z, u, \frac{1}{2} - x - y - z - u \right) = \\
& \max_{ y, z, u } \max \left\{ h_9 \left( 0, y, z, u, \frac{1}{2} - y - z - u \right), h_9 \left( \frac{1}{2} - y - z - u, y, z, u, 0 \right) \right\} = \\
& \max \left\{ \max_{ y, z, u } h_9 \left( 0, y, z, u, \frac{1}{2} - y - z - u \right), \right. \\
 & \left. \qquad \, \, \, \, \max_{ y, z, u } h_9 \left( \frac{1}{2} - y - z - u, y, z, u, 0 \right) \right\} = \\
 & \max \left\{ \max_{ y, z } \max_{ 0 \leq u \leq \frac{1}{2} - y - z } h_9 \left( 0, y, z, u, \frac{1}{2} - y - z - u \right), \right. \\
 & \left. \qquad \, \, \, \, \max_{ y, z } \max_{ 0 \leq u \leq \frac{1}{2} - y - z } h_9 \left( \frac{1}{2} - y - z - u, y, z, u, 0 \right) \right\} = \\
 & \max \left\{ \max_{ y, z } \max \left\{ h_9 \left( 0, y, z, 0, \frac{1}{2} - y - z \right), h_9 \left( 0, y, z, \frac{1}{2} - y - z, 0 \right) \right\} \right. , \\
 & \left. \qquad \, \, \, \, \max_{ y, z } \max \left\{ h_9 \left( \frac{1}{2} - y - z, y, z, 0, 0 \right), h_9 \left( 0, y, z, \frac{1}{2} - y - z, 0 \right) \right\} \right\} = \\
 & \max \left\{ \max_{ y, z } h_9 \left( 0, y, z, 0, \frac{1}{2} - y - z \right), \max_{ y, z } h_9 \left( \frac{1}{2} - y - z, y, z, 0, 0 \right) \right. , \\
 & \left. \qquad \, \, \, \, \max_{ y, z } h_9 \left( 0, y, z, \frac{1}{2} - y - z, 0 \right) \right\} .
\end{align*}
We check each expression.

\begin{enumerate}

\item{$ x = 0 $ and $ u = 0 $.}

We have
\begin{align*}
\frac{ \partial^2 }{ \partial y^2 } h_9 \left( 0, y, z, 0, \frac{1}{2} - y - z \right) &= 200 \left( 306 z - 25 \right) \\
 & \geq 200 \left( 306 \cdot \frac{26}{100} - 25 \right) = 10912 > 0
\end{align*}
and so $ h_9 \left( 0, y, z, 0, \frac{1}{2} - y - z \right) $ is convex in $y$ (keeping $z$ fixed). Therefore,
\begin{align*}
 & \max_{ y, z } h_9 \left( 0, y, z, 0, \frac{1}{2} - y - z \right) = \max_z \max_{ 0 \leq y \leq \frac{1}{2} - z } h_9 \left( 0, y, z, 0, \frac{1}{2} - y - z \right) \\
 & \qquad \quad = \max_z \max \left\{ h_9 \left( 0, 0, z, 0, \frac{1}{2} - z \right), h_9 \left( 0, \frac{1}{2} - z, z, 0, 0 \right) \right\} .
\end{align*}

\begin{enumerate}

\item{$ y = 0 $.}

\[ h_9 \left( 0, 0, z, 0, \frac{1}{2} - z \right) = - 150 \left( 1 - 2 z \right) \left( 102 z^2 + 24 z - 13 \right) \leq 0 \]
as it has roots $ -0.4935350885, 0.2582409708, 0.5 $ and $ \frac{26}{100} \leq z \leq \frac{1}{2} $. 
\item{$ y = \frac{1}{2} - z $.} \label{case:x=u=0_y=0.5-z}
\[ h_9 \left( 0, \frac{1}{2} - z, z, 0, 0 \right) = - 1875 \left( 1 - 2 z \right) \leq 0 \]
as $ z \leq \frac{1}{2} $.

\end{enumerate}

\item{$ x = \frac{1}{2} - y - z $ and $ u = 0 $.}

We have
\[ \frac{ \partial^2 }{ \partial y^2 } h_9 \left( \frac{1}{2} - y - z, y, z, 0, 0 \right) = 30600 > 0 \]
and so $ h_9 \left( \frac{1}{2} - y - z, y, z, 0, 0 \right) $ is convex in $y$ (keeping $z$ fixed). Therefore,
\begin{align*}
 & \max_{ y, z } h_9 \left( \frac{1}{2} - y - z, y, z, 0, 0 \right) = \max_z \max_{ 0 \leq y \leq \frac{1}{2} - z } h_9 \left( \frac{1}{2} - y - z, y, z, 0, 0 \right) \\
 & \qquad \quad = \max_z \max \left\{ h_9 \left( \frac{1}{2} - z, 0, z, 0, 0 \right), h_9 \left( 0, \frac{1}{2}- z, z, 0, 0 \right) \right\} .
\end{align*}

\begin{enumerate}

\item{$ y = 0 $.}

\[ h_9 \left( \frac{1}{2} - z, 0, z, 0, 0 \right) = 150 \left( 1 - 2 z \right) \left( 1 - 51 z \right) \leq 0 \]
as $ \frac{1}{51} < \frac{26}{100} \leq z \leq \frac{1}{2} $.

\item{$ y = \frac{1}{2} - z $.}

Same as case \ref{case:x=u=0_y=0.5-z}.

\end{enumerate}

\item{$ x = 0 $ and $ u = \frac{1}{2} - y - z $.}

\[ \frac{ \partial^2 }{ \partial y^2 } h_9 \left( 0, y, z, \frac{1}{2} - y - z, 0 \right) = 40000 > 0 \]
and so $ h_9 \left( 0, y, z, \frac{1}{2} - y - z, 0 \right) $ is convex in $y$ (keeping $z$ fixed). Therefore,
\begin{align*}
 & \max_{ y, z } h_9 \left( 0, y, z, \frac{1}{2} - y - z, 0 \right) = \max_z \max_{ 0 \leq y \leq \frac{1}{2} - z } h_9 \left( 0, y, z, \frac{1}{2} - y - z, 0 \right) \\
 & \qquad \quad = \max_z \max \left\{ h_9 \left( 0, 0, z, \frac{1}{2} - z, 0 \right), h_9 \left( 0, \frac{1}{2} - z, z, 0, 0 \right) \right\} .
\end{align*}

\begin{enumerate}

\item{$ y = 0 $.}

\[ h_9 \left( 0, 0, z, \frac{1}{2} - z, 0 \right) = - 1875 \left( 1 - 2 z \right) \leq 0 \]
as $ z \leq \frac{1}{2} $.

\item{$ y = \frac{1}{2} - z $.}

Same as case \ref{case:x=u=0_y=0.5-z}.

\end{enumerate}

\end{enumerate}

Therefore, we got that $ h_9 \left( x, y, z, u, v \right) \leq 0 $ over $F$, which contradicts \eqref{eq:h_9_positive}.
\end{proof}

\begin{lemma} \label{lem:b_12_contradiction}
If $ b_{55} \geq \frac{26}{100} $ then \eqref{eq:b_12_better_lower_bound} and \eqref{eq:b_12_better_upper_bound} cannot both hold.
\end{lemma}

\begin{proof}
The proof is similar to the proof of Lemma~\ref{lem:b_24_contradiction}. Note, in particular, that when we formally replace the symbols $ b_{12} $, $ b_{11} $, $ b_{33} $, $ b_{44} $, $ b_{55} $ by $ b_{24} $, $ b_{44} $, $ b_{55} $, $ b_{11} $, $ b_{33} $ respectively (keeping $ b_{22} $ unchanged) into $ b_{55} \geq \frac{26}{100} $, \eqref{eq:b_12_better_lower_bound} and \eqref{eq:b_12_better_upper_bound} we get $ b_{33} \geq \frac{26}{100} $, \eqref{eq:b_24_better_lower_bound} and \eqref{eq:b_24_better_upper_bound} that are the assumptions of Lemma~\ref{lem:b_24_contradiction}.
\end{proof}

\section{$ b_{33} < \frac{26}{100} $ and $ b_{55} < \frac{26}{100} $} \label{app:b_33_and_b_55_leq_0.26}

In this section we assume \eqref{eq:b_12_geq_b_35}, \eqref{eq:b_24_geq_b_35}, \eqref{eq:b_12_upper_bound}, \eqref{eq:b_45_upper_bound}, \eqref{eq:b_12_lower_bound}, \eqref{eq:b_24_upper_bound}, \eqref{eq:b_13_upper_bound}, \eqref{eq:b_24_lower_bound}, \eqref{eq:b_35_lower_bound_2345}, \eqref{eq:b_45_better_upper_bound}, \eqref{eq:b_35_lower_bound_1235}, \eqref{eq:b_13_better_upper_bound}, \eqref{eq:b_45_lower_bound}, \eqref{eq:b_24_better_lower_bound}, \eqref{eq:b_13_lower_bound}, \eqref{eq:b_12_better_lower_bound}, \eqref{eq:b_24_better_upper_bound}, \eqref{eq:b_12_better_upper_bound}, $ b_{33} < \frac{26}{100} $ and $ b_{55} < \frac{26}{100} $.

Our goal is to show that under the above conditions we have $ \rho \left( B \right) > 1 $. We do this by defining a matrix
\[ B_{min} = \left(
\begin{array}{ccccc}
m_{11} & m_{12} & m_{13} & 0 & 0 \\ \noalign{\medskip}
m_{12} & m_{22} & 0 & m_{24} & 0 \\ \noalign{\medskip}
m_{13} & 0 & m_{33} & 0 & m_{35} \\ \noalign{\medskip}
0 & m_{24} & 0 & m_{44} & m_{45} \\ \noalign{\medskip}
0 & 0 & m_{35} & m_{45} & m_{55}
\end{array}
\right) , \]
such that $ 0 \leq m_{ij} \leq b_{ij} $ for $ 1 \leq i \leq j \leq 5 $, and showing that $ P_{B_{min}} \left( 1 \right) < 0 $. By Observation~\ref{obs:negative_cp} we get that $ \rho \left( B_{min} \right) > 1 $. By the Perron-Frobenius Theorem $ \rho \left( B \right) \geq \rho \left( B_{min} \right) $ and we are done.

In Section~\ref{app:bounds_for_b_ij}, given lower and upper bounds for the main diagonal elements of $B$, we find lower and upper bounds for off-diagonal elements of $B$ over the given range of the main diagonal elements. In Section~\ref{app:negative_cp_B_min} we show that $ P_{B_{min}} \left( 1 \right) < 0 $ by dividing the problem into four main cases and further sub-dividing each case into sub-cases. In each sub-case we define $ B_{min} $ based on the lower bounds derived using the results of Section~\ref{app:bounds_for_b_ij}.

\subsection{Bounds for $ b_{ij} $} \label{app:bounds_for_b_ij}

Let $ 0 \leq m_{ii} \leq M_{ii} \leq \frac{1}{2} $ for $ i = 1, 2, \dots, 5 $ be given and define
\begin{align*}
G &= G \left( m_{11}, \dots, m_{55}, M_{11}, \dots, M_{55} \right) \\
 &= \left\{ \left( x_1, x_2, x_3, x_4, x_5 \right) \in {\mathbb{R}}^5 \mid m_{ii} \leq x_i \leq M_{ii}, i = 1, 2, \dots, 5 \right\} .
\end{align*}
Assume $ \left( b_{11}, b_{22}, b_{33}, b_{44}, b_{55} \right) \in G $. We denote by $ m_{ij} $ ($ M_{ij} $) a lower (upper) bound for $ b_{ij} $ ($ i \neq j$) over $G$. In the following sub-sections we derive such bounds.

\subsubsection{$ b_{12} $ and $ b_{24} $}

Note that the right-hand side of \eqref{eq:b_24_better_lower_bound} and the right-hand side of \eqref{eq:b_12_better_lower_bound} are monotonically increasing in the $ b_{ii} $'s. Therefore,
\begin{align}
m_{12} &= \frac{1}{2} \sqrt{ 1 + 2 m_{44} + 4 m_{11} m_{22} + 4 m_{33} m_{55} } , \label{eq:m_12} \\
m_{24} &= \frac{1}{2} \sqrt{ 1 + 2 m_{11} + 4 m_{22} m_{44} + 4 m_{33} m_{55} } \label{eq:m_24} .
\end{align}

Note also that the right-hand side of \eqref{eq:b_24_better_upper_bound} is monotonically decreasing in $ b_{22} $ and in $ b_{44} $, and the right-hand side of \eqref{eq:b_12_better_upper_bound} is monotonically decreasing in $ b_{11} $ and in $ b_{22} $.

We look at the right-hand side of \eqref{eq:b_24_better_upper_bound} as a function of $ b_{11} $ treating the other $ b_{ii} $'s as constants. By Assumption~\ref{assmp:max_diagonal}, $ b_{11} < \frac{1}{2} $, so we define $ h_{10} : \left[ 0, \frac{1}{2} \right) \to \mathbb{R} $ by
\[ h_{10} \left( x \right) = \frac{ C_7 \sqrt{ C_8 - 2 x } }{ C_9 \sqrt{ 1 - x } + C_{10} \sqrt{ 1 - 2 x } } , \]
where
\begin{align*}
C_7 &= \sqrt{ 3 - 2 b_{44} - 2 b_{55} } \geq \sqrt{2} , \\
C_8 &= 3 - 2 b_{22} \geq 2 , \\
C_9 &= 4 \sqrt{ 1 - b_{55} } \geq \sqrt{8} , \\
C_{10} &= 2 \sqrt{ 1 - 2 b_{55} } \geq 0 .
\end{align*}
Therefore,
\[ \frac{ d }{ d x } h_{10} \left( x \right) = \frac{ C_7 \left( C_9 \left( C_8 - 2 \right) \sqrt{ 1 - 2 x } + 2 C_{10} \left( C_8 - 1 \right) \sqrt{ 1 - x } \right) }{ 2 \sqrt{ \left( C_8 - 2 x \right) \left( 1 - x \right) \left( 1 - 2 x \right) } \left( C_9 \sqrt{ 1 - x } + C_{10} \sqrt{ 1 - 2 x } \right)^2 } \]
is nonnnegative. Note that the denominator is positive as $ x < \frac{1}{2} $. Hence, the right-hand side of \eqref{eq:b_24_better_upper_bound} is increasing in $ b_{11} $. By symmetry, it is also increasing in $ b_{55} $. Similarly, by symmetry, the right-hand side of \eqref{eq:b_12_better_upper_bound} is increasing in $ b_{33} $ and in $ b_{44} $. Therefore,
\begin{align}
M_{12} &= \frac{ \sqrt{ \left( 3 - 2 m_{11} - 2 M_{33} \right) \left( 3 - 2 m_{22} - 2 M_{44} \right) } }{ 2 \left( 2 \sqrt{ \left( 1 - M_{33} \right) \left( 1 - M_{44} \right) } + \sqrt{ \left( 1 - 2 M_{33} \right) \left( 1 - 2 M_{44} \right) } \right) } , \label{eq:M_12} \\
M_{24} &= \frac{ \sqrt{ \left( 3 - 2 M_{11} - 2 m_{22} \right) \left( 3 - 2 m_{44} - 2 M_{55} \right) } }{ 2 \left( 2 \sqrt{ \left( 1 - M_{11} \right) \left( 1 - M_{55} \right) } + \sqrt{ \left( 1 - 2 M_{11} \right) \left( 1 - 2 M_{55} \right) } \right) } \label{eq:M_24} .
\end{align}

\subsubsection{$ b_{13} $ and $ b_{45} $}

Note that the right-hand side of \eqref{eq:b_45_lower_bound} is monotonically decreasing in $ b_{11} $, $ b_{22} $, $ b_{44} $, $ b_{55} $, $ b_{12} $ and monotonically increasing in $ b_{24} $. Similarly, the right-hand side of \eqref{eq:b_13_lower_bound} is monotonically decreasing in $ b_{11} $, $ b_{22} $, $ b_{33} $, $ b_{44} $, $ b_{24} $ and monotonically increasing in $ b_{12} $. Therefore,
\begin{align}
m_{13} &= \frac{1}{2} \sqrt{ \left( 1 - 2 M_{33} \right) \left( 1 - 2 M_{11} + \frac{ 4 \left( 1 - 2 M_{44} \right) {m_{12}}^2 }{ 4 {M_{24}}^2 - \left( 1 - 2 M_{22} \right) \left( 1 - 2 M_{44} \right) } \right) } , \label{eq:m_13} \\
m_{45} &= \frac{1}{2} \sqrt{ \left( 1 - 2 M_{55} \right) \left( 1 - 2 M_{44} + \frac{ 4 \left( 1 - 2 M_{11} \right) {m_{24}}^2 }{ 4 {M_{12}}^2 - \left( 1 - 2 M_{11} \right) \left( 1 - 2 M_{22} \right) } \right) } . \label{eq:m_45}
\end{align}

\subsubsection{$ b_{35} $}

Note that the right-hand side of \eqref{eq:b_35_lower_bound_2345} is monotonically decreasing in $ b_{22} $, $ b_{33} $, $ b_{44} $, $ b_{55} $, $ b_{24} $ and monotonically increasing in $ b_{45} $. Similarly, the right-hand side of \eqref{eq:b_35_lower_bound_1235} is monotonically decreasing in $ b_{11} $, $ b_{22} $, $ b_{33} $, $ b_{55} $, $ b_{12} $ and monotonically increasing in $ b_{13} $. Therefore,
\begin{equation}
m_{35} = \max \left\{ m_{35}^{13}, m_{35}^{45} \right\} , \label{eq:m_35}
\end{equation}
where
\begin{align}
m_{35}^{13} &= \frac{1}{2} \sqrt{ \left( 1 - 2 M_{55} \right) \left( 1 - 2 M_{33} + \frac{ 4 \left( 1 - 2 M_{22} \right) {m_{13}}^2 }{ 4 {M_{12}}^2 - \left( 1 - 2 M_{11} \right) \left( 1 - 2 M_{22} \right) } \right) } , \label{eq:m_35_13} \\
m_{35}^{45} &= \frac{1}{2} \sqrt{ \left( 1 - 2 M_{33} \right) \left( 1 - 2 M_{55} + \frac{ 4 \left( 1 - 2 M_{22} \right) {m_{45}}^2 }{ 4 {M_{24}}^2 - \left( 1 - 2 M_{22} \right) \left( 1 - 2 M_{44} \right) } \right) } . \label{eq:m_35_45}
\end{align}

\subsection{Calculating $ P_{B_{min}} \left( 1 \right) $} \label{app:negative_cp_B_min}

In this section we divide the problem into four main cases based on the relation between some main diagonal elements of $B$. This gives rise to relations between some off-diagonal elements. When needed, we further divide each case into sub-cases based on ranges of some main diagonal elements of $B$, taking into account the relations between some main diagonal elements. In each sub-case we calculate $ m_{ij} $ based on the equations of Section~\ref{app:bounds_for_b_ij}, but in order to avoid complicated expressions and to simplify future calculations, we use a rational lower bound approximation of the $ m_{ij} $'s. Next, if possible, we use the relations among the off-diagonal elements to improve (enlarge) some $ m_{ij} $'s. Finally, we calculate $ P_{B_{min}} \left( 1 \right) = \det \left( I_5 - B_{min} \right) $, where $ B_{min} $ is defined using the (possibly improved) rational lower bounds.

The rational approximation of the $ m_{ij} $'s deserves a more detailed discussion. In the following discussion we assume the $ m_{ii} $'s and the $ M_{ii} $'s in Section~\ref{app:bounds_for_b_ij} are rational. Then the lower bounds of $ m_{12} $ and $ m_{24} $, given by \eqref{eq:m_12} and \eqref{eq:m_24} respectively, are square roots of rational numbers. Checking whether $ \frac{r}{10^n} $ is a rational lower bound approximation of $ \sqrt{ \frac{p}{q} } $ with an accuracy of $ 10^{-n} $, where $q$, $n$ are positive integers and $p$, $r$ are nonnegative integers, is an easy task. We simply need to check that $ \frac{p}{q} - \frac{r^2}{10^{2n}} \geq 0 $ and that $ \frac{\left( r + 1 \right) ^2}{10^{2n}} - \frac{p}{q} \geq 0 $, or equivalently, that $ 10^{2n} p - q r^2 \geq 0 $ and that $ q \left( r + 1 \right)^2 - 10^{2n} p \geq 0 $. This provides a way to check the rational lower bound approximation of $ m_{12} $ and $ m_{24} $ mentioned in the sections to come.

Next, we find an upper bound approximation of $ M_{12} $, such that this upper bound will be a square root of a rational number. Note that by \eqref{eq:M_12} we have
\begin{align*}
{M_{12}}^2 &= \frac{ \left( 3 - 2 m_{11} - 2 M_{33} \right) \left( 3 - 2 m_{22} - 2 M_{44} \right) }{ \left( 2 \left( 2 \sqrt{ \left( 1 - M_{33} \right) \left( 1 - M_{44} \right) } + \sqrt{ \left( 1 - 2 M_{33} \right) \left( 1 - 2 M_{44} \right) } \right) \right)^2 } \\
 &= \frac{ \left( 3 - 2 m_{11} - 2 M_{33} \right) \left( 3 - 2 m_{22} - 2 M_{44} \right) }{ 4 \left( 4 \left( 1 - M_{33} \right) \left( 1 - M_{44} \right) + \sqrt{ r_{12} } + \left( 1 - 2 M_{33} \right) \left( 1 - 2 M_{44} \right) \right) } ,
\end{align*}
where
\begin{align}
r_{12} = 16 \left( 1 - M_{33} \right) \left( 1 - M_{44} \right) \left( 1 - 2 M_{33} \right) \left( 1 - 2 M_{44} \right) \label{eq:r_12}
\end{align}
is a rational number. Let $ \tilde{r}_{12} $ be a rational lower bound approximation of $ \sqrt{ r_{12} } $ with a yet to be determined accuracy. As before, we can easily verify this approximation in the sections to come. Then
\begin{align}
\tilde{M}_{12} = \frac{1}{2} \sqrt{ \frac{ \left( 3 - 2 m_{11} - 2 M_{33} \right) \left( 3 - 2 m_{22} - 2 M_{44} \right) }{ 4 \left( 1 - M_{33} \right) \left( 1 - M_{44} \right) + \tilde{r}_{12} + \left( 1 - 2 M_{33} \right) \left( 1 - 2 M_{44} \right) } } \label{eq:tilde_M_12}
\end{align}
is a square root of a rational number and $ \tilde{M}_{12} \geq M_{12} $. In a similar fashion, let
\begin{align}
r_{24} = 16 \left( 1 - M_{11} \right) \left( 1 - M_{55} \right) \left( 1 - 2 M_{11} \right) \left( 1 - 2 M_{55} \right) \label{eq:r_24}
\end{align}
and let $ \tilde{r}_{24} $ be a rational lower bound approximation of $ \sqrt{ r_{24} } $ with a yet to be determined accuracy. Then
\begin{align}
\tilde{M}_{24} = \frac{1}{2} \sqrt{ \frac{ \left( 3 - 2 M_{11} - 2 m_{22} \right) \left( 3 - 2 m_{44} - 2 M_{55} \right) }{ 4 \left( 1 - M_{11} \right) \left( 1 - M_{55} \right) + \tilde{r}_{24} + \left( 1 - 2 M_{11} \right) \left( 1 - 2 M_{55} \right) } } \label{eq:tilde_M_24}
\end{align}
is a square root of a rational number and $ \tilde{M}_{24} \geq M_{24} $.

Substituting $ \tilde{M}_{24} $ for $ M_{24} $ in \eqref{eq:m_13} we get an even smaller $ m_{13} $, which is square root of a rational number. Similarly, substituting $ \tilde{M}_{12} $ for $ M_{12} $ in \eqref{eq:m_45} we get an even smaller $ m_{45} $, which is square root of a rational number. As before, we can easily verify a rational lower bound approximation of $ m_{13} $ and $ m_{45} $ mentioned in the sections to come.

Finally, using the above mentioned $ m_{13} $ and substituting $ \tilde{M}_{12} $ for $ M_{12} $ in \eqref{eq:m_35_13} we get an even smaller $ m_{35}^{13} $, which is square root of a rational number. Similarly, using the above mentioned $ m_{45} $ and substituting $ \tilde{M}_{24} $ for $ M_{24} $ in \eqref{eq:m_35_45} we get an even smaller $ m_{35}^{45} $, which is square root of a rational number. As before, we can easily verify a rational lower bound approximation of $ m_{35}^{13} $ and $ m_{35}^{45} $ mentioned in the sections to come.

In the following subsections we do some computations using Maple, with the aim of using symbolic calculations whenever possible. We perform the calculation in phases:
\begin{enumerate}

\item We select rational values for the $ m_{ii} $'s and the $ M_{ii} $'s.

\item The formulas for $ m_{12} $ and $ m_{24} $, given by \eqref{eq:m_12} and \eqref{eq:m_24} respectively, involve only the $ m_{ii} $'s. Hence $ m_{12} $ and $ m_{24} $ are square roots of rational numbers. Therefore, $ { m_{12} }^2 $ and $ { m_{24} }^2 $ are rational and are computed precisely. Those exact values are substituted into the formulas of $ m_{13} $, $ m_{45} $, given by \eqref{eq:m_13} and \eqref{eq:m_45} respectively. $ { \tilde{M}_{12} }^2 $ and $ { \tilde{M}_{24} }^2 $ are rational numbers and they are substituted for $ { M_{12} }^2 $ and $ { M_{24} }^2 $ respectively into formulas \eqref{eq:m_13}, \eqref{eq:m_45}, \eqref{eq:m_35_13}, \eqref{eq:m_35_45}. Consequently, $ m_{13} $, $ m_{45} $, $ m_{35}^{13} $ and $ m_{35}^{45} $ are square roots of rational numbers.

\item We calculate lower bound approximations of the $ m_{ij} $'s with an accuracy of $ 10^{-2} $, as indicated before.

\item We potentially improve some $ m_{ij} $'s by using relations between the off-diagonal elements of $B$. Obviously, this keeps the $ m_{ij} $'s rational.

\item We build the matrix $ B_{min} $ using the $ m_{ij} $'s from the previous phase. The computation of $ P_{B_{min}} \left( 1 \right) = \det \left( I_5 - B_{min} \right) $ involves only rational numbers so is computed precisely.

\end{enumerate}

For example, let
\begin{center}
\begin{tabular}{ *{5}{ >{$} l <{$} } }
m_{11} = \frac{1}{100} , & m_{22} = 0 , & m_{33} = \frac{3}{25} , & m_{44} = 0 , & m_{55} = \frac{9}{50} , \\
\addlinespace
M_{11} = \frac{19}{100} , & M_{22} = \frac{1}{15} , & M_{33} = \frac{13}{50} , & M_{44} = \frac{1}{10} , & M_{55} = \frac{21}{100} . \\
\end{tabular}
\end{center}
Taking rational lower bound approximation of $ \sqrt{ r_{12} } $ and $ \sqrt{ r_{24} } $ with an accuracy of $ 10^{-2} $ we get
\begin{center}
\begin{tabular}{ *{3}{ >{$} l <{$} } }
r_{12} = \frac{63,936}{15,625} , & \tilde{r}_{12} = \frac{101}{50} , & \tilde{M}_{12} = \frac{ \sqrt{44,526} }{362} , \\
\addlinespace
r_{24} = \frac{5,752,701}{1,562,500} , & \tilde{r}_{24} = \frac{191}{100} , & \tilde{M}_{24} = \frac{ \sqrt{204,021,627} }{24,146} . \\
\end{tabular}
\end{center}
Next,
\begin{center}
\begin{tabular}{ *{2}{ >{$} l <{$} } }
m_{12} = \frac{ \sqrt{679} }{50} , & m_{13} = \frac{ \sqrt{5,677,681,587,533,610} }{159,907,250} , \\
\addlinespace
m_{24} = \frac{ \sqrt{2766} }{100} , & m_{35}^{13} = \frac{ 3 \sqrt{1,815,792,772,377,013,574,394,490} }{8,919,386,544,125} , \\
\addlinespace
m_{45} = \frac{ \sqrt{737,476,403,275,835} }{55,778,500} , & m_{35}^{45} = \frac{ 3 \sqrt{7,396,556,511,412,986,644,789,619} }{17,838,773,088,250} . \\
\end{tabular}
\end{center}
Taking rational lower bound approximation of the $ m_{ij} $'s with an accuracy of $ 10^{-2} $ we get
\begin{center}
\begin{tabular}{ *{3}{ >{$} l <{$} } }
m_{12} = \frac{13}{25} , & m_{13} = \frac{47}{100} , & m_{24} = \frac{13}{25} , \\
\addlinespace
m_{35}^{13} = \frac{9}{20} , & m_{35}^{45} = \frac{9}{20} , & m_{45} = \frac{12}{25} . \\
\end{tabular}
\end{center}
Therefore, the lower bound matrix is
\[ B_{min} = \left(
\begin{array}{ccccc}
\frac{1}{100} & \frac{13}{25} & \frac{47}{100} & 0 & 0 \\ \noalign{\medskip}
\frac{13}{25} & 0 & 0 & \frac{13}{25} & 0 \\ \noalign{\medskip}
\frac{47}{100} & 0 & \frac{3}{25} & 0 & \frac{9}{20} \\ \noalign{\medskip}
0 & \frac{13}{25} & 0 & 0 & \frac{12}{25} \\ \noalign{\medskip}
0 & 0 & \frac{9}{20} & \frac{12}{25} & \frac{9}{50}
\end{array}
\right) , \]
and
\[ P_{B_{min}} \left( 1 \right) = -\frac{7,419,049}{156,250,000} . \]

\subsubsection{$ b_{55} \geq b_{11} $ and $ b_{44} \geq b_{33} $} \label{sec:main_diag_case_1}

By Assumption~\ref{assmp:min_diagonal}, Assumption~\ref{assmp:b_11_geq_b_44} and by Lemma~\ref{lem:not_similar_to_positive_C} we have
\begin{align}
b_{55} \geq b_{11} \geq b_{44} \geq b_{33} \geq b_{22} , \nonumber \\
b_{12} \geq b_{35}, \quad b_{24} \geq b_{35}, \quad b_{45} \geq b_{13} . \label{eq:b_55_geq_b11_and_b44_geq_b33}
\end{align}
Let $ m_{55} \leq b_{55} \leq M_{55} $. We derive upper and lower bounds for the $ b_{ii} $'s using \eqref{eq:b_55_geq_b11_and_b44_geq_b33}.

\paragraph{Bounds for $ b_{11} $.}

We have
\begin{align*}
b_{11} & \leq b_{11} + b_{22} + b_{33} + b_{44} = \frac{1}{2} - b_{55} \leq \frac{1}{2} - m_{55} , \\
2 b_{11} & \leq b_{11} + b_{55} \leq b_{11} + b_{22} + b_{33} + b_{44} + b_{55} = \frac{1}{2} , \\
4 b_{11} & \geq b_{11} + b_{22} + b_{33} + b_{44} = \frac{1}{2} - b_{55} \geq \frac{1}{2} - M_{55} ,
\end{align*}
and so $ m_{11} \leq b_{11} \leq M_{11} $, where
\begin{align*}
m_{11} &= \max \left\{ 0, \frac{1}{8} - \frac{1}{4} M_{55} \right\} , \\
M_{11} &= \min \left\{ \frac{1}{4}, M_{55}, \frac{1}{2} - m_{55} \right\} .
\end{align*}

\paragraph{Bounds for $ b_{44} $.}

We have
\begin{align*}
2 b_{44} & \leq b_{11} + b_{44} \leq b_{11} + b_{22} + b_{33} + b_{44} = \frac{1}{2} - b_{55} \leq \frac{1}{2} - m_{55} , \\
3 b_{44} & \leq b_{11} + b_{44} + b_{55} \leq b_{11} + b_{22} + b_{33} + b_{44} + b_{55} = \frac{1}{2} ,
\end{align*}
and so, $ m_{44} \leq b_{44} \leq M_{44} $, where
\begin{align*}
m_{44} &= 0 , \\
M_{44} &= \min \left\{ \frac{1}{6}, M_{55}, \frac{1}{4} - \frac{1}{2} m_{55} \right\} .
\end{align*}

\paragraph{Bounds for $ b_{33} $.}

We have
\begin{align*}
3 b_{33} & \leq b_{11} + b_{33} + b_{44} \leq b_{11} + b_{22} + b_{33} + b_{44} = \frac{1}{2} - b_{55} \leq \frac{1}{2} - m_{55} , \\
4 b_{33} & \leq b_{11} + b_{33} + b_{44} + b_{55} \leq b_{11} + b_{22} + b_{33} + b_{44} + b_{55} = \frac{1}{2} ,
\end{align*}
and so $ m_{33} \leq b_{33} \leq M_{33} $, where
\begin{align*}
m_{33} &= 0 , \\
M_{33} &= \min \left\{ \frac{1}{8}, M_{55}, \frac{1}{6} - \frac{1}{3} m_{55} \right\} .
\end{align*}

\paragraph{Bounds for $ b_{22} $.}

We have
\begin{align*}
4 b_{22} & \leq b_{11} + b_{22} + b_{33} + b_{44} = \frac{1}{2} - b_{55} \leq \frac{1}{2} - m_{55} ,
\end{align*}
and so, by Assumption~\ref{assmp:min_diagonal}, $ m_{22} \leq b_{22} \leq M_{22} $, where
\begin{align*}
m_{22} &= 0 , \\
M_{22} &= \min \left\{ \frac{1}{10}, M_{55}, \frac{1}{8} - \frac{1}{4} m_{55} \right\} .
\end{align*}

\paragraph{Sub-ranges.}

We divide the range $ 0 \leq b_{55} \leq \frac{26}{100} $ into sub-ranges.

Table~\ref{tbl:main_diag_case1_sub_range} lists the sub-ranges.
\begin{table}
\begin{center}
\begin{tabular}{ c *{2}{ >{$} c <{$} } }
\toprule
Sub-range & m_{55} & M_{55} \\
\midrule
\addlinespace
1 & 0 & \frac{20}{100} \\
\addlinespace
2 & \frac{20}{100} & \frac{26}{100} \\
\addlinespace
\bottomrule
\end{tabular}
\caption{Case 1 sub-ranges}
\label{tbl:main_diag_case1_sub_range}
\end{center}
\end{table}

Table~\ref{tbl:main_diag_case1_m_ii} lists the $ m_{ii} $'s for each sub-range.
\begin{table}
\begin{center}
\begin{tabular}{ c *{5}{ >{$} c <{$} } }
\toprule
Sub-range & m_{11} & m_{22} & m_{33} & m_{44} & m_{55} \\
\midrule
\addlinespace
1 & \frac{3}{40} & 0 & 0 & 0 & 0 \\
\addlinespace
2 & \frac{3}{50} & 0 & 0 & 0 & \frac{1}{5} \\
\addlinespace
\bottomrule
\end{tabular}
\caption{Case 1 $ m_{ii} $'s}
\label{tbl:main_diag_case1_m_ii}
\end{center}
\end{table}

Table~\ref{tbl:main_diag_case1_M_ii} lists the $ M_{ii} $'s for each sub-range.
\begin{table}
\begin{center}
\begin{tabular}{ c *{5}{ >{$} c <{$} } }
\toprule
Sub-range & M_{11} & M_{22} & M_{33} & M_{44} & M_{55} \\
\midrule
\addlinespace
1 & \frac{1}{5} & \frac{1}{10} & \frac{1}{8} & \frac{1}{6} & \frac{1}{5} \\
\addlinespace
2 & \frac{1}{4} & \frac{3}{40} & \frac{1}{10} & \frac{3}{20} & \frac{13}{50} \\
\addlinespace
\bottomrule
\end{tabular}
\caption{Case 1 $ M_{ii} $'s}
\label{tbl:main_diag_case1_M_ii}
\end{center}
\end{table}

Table~\ref{tbl:main_diag_case1_intermediates} lists the values of intermediate calculations based on \eqref{eq:r_12}, \eqref{eq:tilde_M_12}, \eqref{eq:r_24}, \eqref{eq:tilde_M_24}, where $ \tilde{r}_{12} $ and $ \tilde{r}_{24} $ are calculated with an accuracy of $ 10^{-2} $, for each sub-range.
\begin{table}
\begin{center}
\begin{tabular}{ c *{6}{ >{$} c <{$} } }
\toprule
Sub-range & r_{12} & \tilde{r}_{12} & \tilde{M}_{12} & r_{24} & \tilde{r}_{24} & \tilde{M}_{24} \\
\midrule
\addlinespace
1 & \frac{35}{6} & \frac{241}{100} & \frac{ \sqrt{56,810} }{437} & \frac{2,304}{625} & \frac{48}{25} & \frac{13}{22} \\
\addlinespace
2 & \frac{4,284}{625} & \frac{261}{100} & \frac{ 3 \sqrt{1,252,230} }{6,230} & \frac{1,332}{625} & \frac{29}{20} & \frac{ \sqrt{60,605} }{391} \\
\addlinespace
\bottomrule
\end{tabular}
\caption{Case 1 intermediates}
\label{tbl:main_diag_case1_intermediates}
\end{center}
\end{table}

Table~\ref{tbl:main_diag_case1_m_ij} lists rational lower bound approximation of $ m_{ij} $'s with an accuracy of $ 10^{-2} $ for each sub-range, taking into account the relations in \eqref{eq:b_55_geq_b11_and_b44_geq_b33}. When the relations improve an $ m_{ij} $ the original rational approximation is given in parentheses. We denote by $ * $ the maximal lower bound among $ m_{35}^{13} $ and $ m_{35}^{45} $, which is taken to be $ m_{35} $.
\begin{table}
\begin{center}
\begin{tabular}{ c *{6}{ >{$} c <{$} } }
\toprule
Sub-range & m_{12} & m_{13} & m_{24} & m_{35}^{13} & m_{35}^{45} & m_{45} \\
\midrule
\addlinespace
1 & \frac{53}{100} \left( \frac{1}{2} \right) & \frac{1}{2} & \frac{53}{100} & \frac{53}{100} * & \frac{53}{100} & \frac{1}{2} \left( \frac{49}{100} \right) \\
\addlinespace
2 & \frac{1}{2} & \frac{49}{100} & \frac{13}{25} & \frac{47}{100} * & \frac{23}{50} & \frac{49}{100} \left( \frac{41}{100} \right) \\
\addlinespace
\bottomrule
\end{tabular}
\caption{Case 1 $ m_{ij} $'s}
\label{tbl:main_diag_case1_m_ij}
\end{center}
\end{table}

Table~\ref{tbl:main_diag_case1_P_B_min} lists $ P_{B_{min}} \left( 1 \right) $ for each sub-range.
\begin{table}
\begin{center}
\begin{tabular}{ c *{1}{ >{$} c <{$} } }
\toprule
Sub-range & P_{B_{min}} \left( 1 \right) \\
\midrule
\addlinespace
1 & -\frac{305,646,963}{4,000,000,000} \\
\addlinespace
2 & -\frac{136,143,299}{2,500,000,000} \\
\addlinespace
\bottomrule
\end{tabular}
\caption{Case 1 $ P_{B_{min}} \left( 1 \right) $ }
\label{tbl:main_diag_case1_P_B_min}
\end{center}
\end{table}

\subsubsection{$ b_{11} \geq b_{55} $ and $ b_{44} \geq b_{33} $}

By Assumption~\ref{assmp:min_diagonal}, Assumption~\ref{assmp:b_11_geq_b_44} and by Lemma~\ref{lem:not_similar_to_positive_C} we have
\begin{align}
b_{11} \geq b_{55} \geq b_{22}, \quad b_{11} \geq b_{44} \geq b_{33} \geq b_{22} , \nonumber \\
b_{12} \geq b_{35}, \quad b_{24} \geq b_{35}, \quad b_{45} \geq b_{13} . \label{eq:b_11_geq_b55_and_b44_geq_b33}
\end{align}
Let $ m_{11} \leq b_{11} \leq M_{11} $. We derive upper and lower bounds for the $ b_{ii} $'s using \eqref{eq:b_11_geq_b55_and_b44_geq_b33}.

\paragraph{Bounds for $ b_{44} $.}

We have
\begin{align*}
b_{44} & \leq b_{22} + b_{33} + b_{44} + b_{55} = \frac{1}{2} - b_{11} \leq \frac{1}{2} - m_{11} , \\
2 b_{44} & \leq b_{11} + b_{44} \leq b_{11} + b_{22} + b_{33} + b_{44} + b_{55} = \frac{1}{2} ,
\end{align*}
and so, $ m_{44} \leq b_{44} \leq M_{44} $, where
\begin{align*}
m_{44} &= 0 , \\
M_{44} &= \min \left\{ \frac{1}{4}, M_{11}, \frac{1}{2} - m_{11} \right\} .
\end{align*}

\paragraph{Bounds for $ b_{55} $.}

We have
\begin{align*}
b_{55} & \leq b_{22} + b_{33} + b_{44} + b_{55} = \frac{1}{2} - b_{11} \leq \frac{1}{2} - m_{11} , \\
2 b_{55} & \leq b_{11} + b_{55} \leq b_{11} + b_{22} + b_{33} + b_{44} + b_{55} = \frac{1}{2} ,
\end{align*}
and so $ m_{55} \leq b_{55} \leq M_{55} $, where
\begin{align*}
m_{55} &= 0 , \\
M_{55} &= \min \left\{ \frac{1}{4}, M_{11}, \frac{1}{2} - m_{11} \right\} .
\end{align*}

\paragraph{Bounds for $ b_{33} $.}

We have
\begin{align*}
2 b_{33} & \leq b_{33} + b_{44} \leq b_{22} + b_{33} + b_{44} + b_{55} = \frac{1}{2} - b_{11} \leq \frac{1}{2} - m_{11} , \\
3 b_{33} & \leq b_{11} + b_{33} + b_{44} \leq b_{11} + b_{22} + b_{33} + b_{44} + b_{55} = \frac{1}{2} ,
\end{align*}
and so $ m_{33} \leq b_{33} \leq M_{33} $, where
\begin{align*}
m_{33} &= 0 , \\
M_{33} &= \min \left\{ \frac{1}{6}, M_{11}, \frac{1}{4} - \frac{1}{2} m_{11} \right\} .
\end{align*}

\paragraph{Bounds for $ b_{22} $.}

We have
\begin{align*}
4 b_{22} & \leq b_{22} + b_{33} + b_{44} + b_{55} = \frac{1}{2} - b_{11} \leq \frac{1}{2} - m_{11} ,
\end{align*}
and so, by Assumption~\ref{assmp:min_diagonal}, $ m_{22} \leq b_{22} \leq M_{22} $, where
\begin{align*}
m_{22} &= 0 , \\
M_{22} &= \min \left\{ \frac{1}{10}, M_{11}, \frac{1}{8} - \frac{1}{4} m_{11} \right\} .
\end{align*}

\paragraph{Sub-ranges.}

We divide the range $ 0 \leq b_{11} \leq \frac{1}{2} $ into sub-ranges. Tables~\ref{tbl:main_diag_case2_sub_range}, \ref{tbl:main_diag_case2_m_ii}, \ref{tbl:main_diag_case2_M_ii}, \ref{tbl:main_diag_case2_intermediates}, \ref{tbl:main_diag_case2_m_ij} and \ref{tbl:main_diag_case2_P_B_min} are arranged as in Section~\ref{sec:main_diag_case_1}, except that in Table~\ref{tbl:main_diag_case2_m_ij} we take into account the relations in \eqref{eq:b_11_geq_b55_and_b44_geq_b33}.

\begin{table}
\begin{center}
\begin{tabular}{ c *{2}{ >{$} c <{$} } }
\toprule
Sub-range & m_{11} & M_{11} \\
\midrule
\addlinespace
1 & 0 & \frac{14}{100} \\
\addlinespace
2 & \frac{14}{100} & \frac{20}{100} \\
\addlinespace
3 & \frac{20}{100} & \frac{24}{100} \\
\addlinespace
4 & \frac{24}{100} & \frac{28}{100} \\
\addlinespace
5 & \frac{28}{100} & \frac{36}{100} \\
\addlinespace
6 & \frac{36}{100} & \frac{50}{100} \\
\addlinespace
\bottomrule
\end{tabular}
\caption{Case 2 sub-ranges}
\label{tbl:main_diag_case2_sub_range}
\end{center}
\end{table}

\begin{table}
\begin{center}
\begin{tabular}{ c *{5}{ >{$} c <{$} } }
\toprule
Sub-range & m_{11} & m_{22} & m_{33} & m_{44} & m_{55} \\
\midrule
\addlinespace
1 & 0 & 0 & 0 & 0 & 0 \\
\addlinespace
2 & \frac{7}{50} & 0 & 0 & 0 & 0 \\
\addlinespace
3 & \frac{1}{5} & 0 & 0 & 0 & 0 \\
\addlinespace
4 & \frac{6}{25} & 0 & 0 & 0 & 0 \\
\addlinespace
5 & \frac{7}{25} & 0 & 0 & 0 & 0 \\
\addlinespace
6 & \frac{9}{25} & 0 & 0 & 0 & 0 \\
\addlinespace
\bottomrule
\end{tabular}
\caption{Case 2 $ m_{ii} $'s}
\label{tbl:main_diag_case2_m_ii}
\end{center}
\end{table}

\begin{table}
\begin{center}
\begin{tabular}{ c *{5}{ >{$} c <{$} } }
\toprule
Sub-range & M_{11} & M_{22} & M_{33} & M_{44} & M_{55} \\
\midrule
\addlinespace
1 & \frac{7}{50} & \frac{1}{10} & \frac{7}{50} & \frac{7}{50} & \frac{7}{50} \\
\addlinespace
2 & \frac{1}{5} & \frac{9}{100} & \frac{1}{6} & \frac{1}{5} & \frac{1}{5} \\
\addlinespace
3 & \frac{6}{25} & \frac{3}{40} & \frac{3}{20} & \frac{6}{25} & \frac{6}{25} \\
\addlinespace
4 & \frac{7}{25} & \frac{13}{200} & \frac{13}{100} & \frac{1}{4} & \frac{1}{4} \\
\addlinespace
5 & \frac{9}{25} & \frac{11}{200} & \frac{11}{100} & \frac{11}{50} & \frac{11}{50} \\
\addlinespace
6 & \frac{1}{2} & \frac{7}{200} & \frac{7}{100} & \frac{7}{50} & \frac{7}{50} \\
\addlinespace
\bottomrule
\end{tabular}
\caption{Case 2 $ M_{ii} $'s}
\label{tbl:main_diag_case2_M_ii}
\end{center}
\end{table}

\begin{table}
\begin{center}
\begin{tabular}{ c *{6}{ >{$} c <{$} } }
\toprule
Sub-range & r_{12} & \tilde{r}_{12} & \tilde{M}_{12} & r_{24} & \tilde{r}_{24} & \tilde{M}_{24} \\
\midrule
\addlinespace
1 & \frac{2,396,304}{390,625} & \frac{247}{100} & \frac{ 68 \sqrt{14,867} }{14,867} & \frac{2,396,304}{390,625} & \frac{247}{100} & \frac{ 68 \sqrt{14,867} }{14,867} \\
\addlinespace
2 & \frac{64}{15} & \frac{103}{50} & \frac{ \sqrt{17,894,630} }{7,690} & \frac{2,304}{625} & \frac{48}{25} & \frac{13}{22} \\
\addlinespace
3 & \frac{58,786}{15,625} & \frac{193}{100} & \frac{ \sqrt{87,262} }{542} & \frac{976,144}{390,625} & \frac{79}{50} & \frac{ 9 \sqrt{10,402} }{1,486} \\
\addlinespace
4 & \frac{9,657}{2,500} & \frac{49}{25} & \frac{ \sqrt{279,110} }{988} & \frac{1,188}{625} & \frac{137}{100} & \frac{ \sqrt{366} }{30} \\
\addlinespace
5 & \frac{1,895,166}{390,625} & \frac{11}{5} & \frac{ 4 \sqrt{751,137} }{6,767} & \frac{489,216}{390,625} & \frac{111}{100} & \frac{ 8 \sqrt{465,063} }{8,159} \\
\addlinespace
6 & \frac{3,095,226}{390,625} & \frac{281}{100} & \frac{ \sqrt{60,285,298} }{16,571} & 0 & 0 & \frac{ \sqrt{1,462} }{43} \\
\addlinespace
\bottomrule
\end{tabular}
\caption{Case 2 intermediates}
\label{tbl:main_diag_case2_intermediates}
\end{center}
\end{table}

\begin{table}
\begin{center}
\begin{tabular}{ c *{6}{ >{$} c <{$} } }
\toprule
Sub-range & m_{12} & m_{13} & m_{24} & m_{35}^{13} & m_{35}^{45} & m_{45} \\
\midrule
\addlinespace
1 & \frac{63}{100} \left( \frac{1}{2} \right) & \frac{14}{25} & \frac{63}{100} \left( \frac{1}{2} \right) & \frac{63}{100} * & \frac{63}{100} * & \frac{14}{25} \\
\addlinespace
2 & \frac{1}{2} & \frac{9}{20} & \frac{14}{25} & \frac{49}{100} & \frac{1}{2} * & \frac{1}{2} \\
\addlinespace
3 & \frac{1}{2} & \frac{41}{100} & \frac{59}{100} & \frac{11}{25} & \frac{11}{25} * & \frac{11}{25} \\
\addlinespace
4 & \frac{1}{2} & \frac{39}{100} & \frac{3}{5} & \frac{21}{50} & \frac{21}{50} * & \frac{41}{100} \\
\addlinespace
5 & \frac{1}{2} & \frac{37}{100} & \frac{31}{50} & \frac{11}{25} * & \frac{43}{100} & \frac{39}{100} \\
\addlinespace
6 & \frac{1}{2} & \frac{6}{25} & \frac{13}{20} & \frac{11}{25} * & \frac{11}{25} & \frac{9}{25} \\
\addlinespace
\bottomrule
\end{tabular}
\caption{Case 2 $ m_{ij} $'s}
\label{tbl:main_diag_case2_m_ij}
\end{center}
\end{table}

\begin{table}
\begin{center}
\begin{tabular}{ c *{1}{ >{$} c <{$} } }
\toprule
Sub-range & P_{B_{min}} \left( 1 \right) \\
\midrule
\addlinespace
1 & -\frac{390,487,023}{1,250,000,000} \\
\addlinespace
2 & -\frac{48,643}{1,000,000} \\
\addlinespace
3 & -\frac{5,699,171}{500,000,000} \\
\addlinespace
4 & -\frac{12,411}{4,000,000} \\
\addlinespace
5 & -\frac{46,207,263}{2,500,000,000} \\
\addlinespace
6 & -\frac{72}{15,625} \\
\addlinespace
\bottomrule
\end{tabular}
\caption{Case 2 $ P_{B_{min}} \left( 1 \right) $ }
\label{tbl:main_diag_case2_P_B_min}
\end{center}
\end{table}

\subsubsection{$ b_{55} \geq b_{11} $ and $ b_{33} \geq b_{44} $}

By Assumption~\ref{assmp:min_diagonal}, Assumption~\ref{assmp:b_11_geq_b_44} and by Lemma~\ref{lem:not_similar_to_positive_C} we have
\begin{align}
b_{55} \geq b_{11} \geq b_{44} \geq b_{22}, \quad b_{33} \geq b_{44} \geq b_{22} , \nonumber \\
b_{12} \geq b_{35}, \quad b_{24} \geq b_{35}, \quad b_{45} \geq b_{13}, \quad b_{24} \geq b_{13} . \label{eq:b_55_geq_b11_and_b33_geq_b44}
\end{align}
Let $ \tilde{m}_{33} \leq b_{33} \leq \tilde{M}_{33} $ and $ \tilde{m}_{55} \leq b_{55} \leq \tilde{M}_{55} $. We derive upper and lower bounds for the $ b_{ii} $'s using \eqref{eq:b_55_geq_b11_and_b33_geq_b44}.

\paragraph{Bounds for $ b_{33} $.}

We have
\begin{align*}
b_{33} &= \frac{1}{2} - b_{11} - b_{22} - b_{44} - b_{55} \leq \frac{1}{2} - b_{55} \leq \frac{1}{2} - \tilde{m}_{55} , \\
b_{33} &= \frac{1}{2} - b_{11} - b_{22} - b_{44} - b_{55} \geq \frac{1}{2} - 4 b_{55} \geq \frac{1}{2} - 4 \tilde{M}_{55} , \\
3 b_{33} & \geq b_{22} + b_{33} + b_{44} = \frac{1}{2} - b_{11} - b_{55} \geq \frac{1}{2} - 2 b_{55} \geq \frac{1}{2} - 2 \tilde{M}_{55} ,
\end{align*}
and so $ m_{33} \leq b_{33} \leq M_{33} $, where
\begin{align*}
m_{33} &= \max \left\{ \tilde{m}_{33}, \frac{1}{2} - 4 \tilde{M}_{55}, \frac{1}{6} - \frac{2}{3} \tilde{M}_{55} \right\} , \\
M_{33} &= \min \left\{ \tilde{M}_{33}, \frac{1}{2} - \tilde{m}_{55} \right\} .
\end{align*}

\paragraph{Bounds for $ b_{55} $.}

We have
\begin{align*}
b_{55} &= \frac{1}{2} - b_{11} - b_{22} - b_{33} - b_{44} \leq \frac{1}{2} - b_{33} \leq \frac{1}{2} - m_{33} , \\
2 b_{55} & \geq b_{11} + b_{55} = \frac{1}{2} - b_{22} - b_{33} - b_{44} \geq \frac{1}{2} - 3 b_{33} \geq \frac{1}{2} - 3 M_{33} , \\
4 b_{55} & \geq b_{11} + b_{22} + b_{44} + b_{55} = \frac{1}{2} - b_{33} \geq \frac{1}{2} - M_{33} ,
\end{align*}
and so, $ m_{55} \leq b_{55} \leq M_{55} $, where
\begin{align*}
m_{55} &= \max \left\{ \tilde{m}_{55}, \frac{1}{4} - \frac{3}{2} M_{33}, \frac{1}{8} - \frac{1}{4} M_{33} \right\} , \\
M_{55} &= \min \left\{ \tilde{M}_{55}, \frac{1}{2} - m_{33} \right\} .
\end{align*}

\paragraph{Bounds for $ b_{11} $.}

We have
\begin{align*}
b_{11} & \leq b_{11} + b_{22} + b_{44} = \frac{1}{2} - b_{33} - b_{55} \leq \frac{1}{2} - m_{33} - m_{55} , \\
2 b_{11} & \leq b_{11} + b_{55} = \frac{1}{2} - b_{22} - b_{33} - b_{44} \leq \frac{1}{2} - b_{33} \leq \frac{1}{2} - m_{33} , \\
3 b_{11} & \geq b_{11} + b_{22} + b_{44} = \frac{1}{2} - b_{33} - b_{55} \geq \frac{1}{2} - M_{33} - M_{55} ,
\end{align*}
and so $ m_{11} \leq b_{11} \leq M_{11} $, where
\begin{align*}
m_{11} &= \max \left\{ 0, \frac{1}{6} - \frac{1}{3} M_{33} - \frac{1}{3} M_{55} \right\} , \\
M_{11} &= \min \left\{ M_{55}, \frac{1}{2} - m_{33} - m_{55}, \frac{1}{4} - \frac{1}{2} m_{33} \right\} .
\end{align*}

\paragraph{Bounds for $ b_{44} $.}

We have
\begin{align*}
2 b_{44} & \leq b_{11} + b_{44} \leq b_{11} + b_{22} + b_{44} = \frac{1}{2} - b_{33} - b_{55} \leq \frac{1}{2} - m_{33} - m_{55} , \\
3 b_{44} & \leq b_{11} + b_{44} + b_{55} \leq b_{11} + b_{22} + b_{33} + b_{55} = \frac{1}{2} - b_{33} \leq \frac{1}{2} - m_{33} , \\
3 b_{44} & \leq b_{11} + b_{33} + b_{44} \leq b_{11} + b_{22} + b_{33} + b_{44} = \frac{1}{2} - b_{55} \leq \frac{1}{2} - m_{55} , \\
4 b_{44} & \leq b_{11} + b_{33} + b_{44} + b_{55} \leq b_{11} + b_{22} + b_{33} + b_{44} + b_{55} = \frac{1}{2} ,
\end{align*}
and so, $ m_{44} \leq b_{44} \leq M_{44} $, where
\begin{align*}
m_{44} &= 0 , \\
M_{44} &= \min \left\{ \frac{1}{8}, M_{33}, M_{55}, \frac{1}{4} - \frac{1}{2} m_{33} - \frac{1}{2} m_{55}, \frac{1}{6} - \frac{1}{3} m_{33}, \frac{1}{6} - \frac{1}{3} m_{55} \right\} .
\end{align*}

\paragraph{Bounds for $ b_{22} $.}

We have
\begin{align*}
3 b_{22} & \leq b_{11} + b_{22} + b_{44} = \frac{1}{2} - b_{33} - b_{55} \leq \frac{1}{2} - m_{33} - m_{55} , \\
4 b_{22} & \leq b_{11} + b_{22} + b_{44} + b_{55} = \frac{1}{2} - b_{33} \leq \frac{1}{2} - m_{33} , \\
4 b_{22} & \leq b_{11} + b_{22} + b_{33} + b_{44} = \frac{1}{2} - b_{55} \leq \frac{1}{2} - m_{55} ,
\end{align*}
and so, by Assumption~\ref{assmp:min_diagonal}, $ m_{22} \leq b_{22} \leq M_{22} $, where
\begin{align*}
m_{22} &= 0 , \\
M_{22} &= \min \left\{ \frac{1}{10}, M_{33}, M_{55}, \frac{1}{6} - \frac{1}{3} m_{33} - \frac{1}{3} m_{55}, \frac{1}{8} - \frac{1}{4} m_{33}, \frac{1}{8} - \frac{1}{4} m_{55} \right\} .
\end{align*}

\paragraph{Sub-ranges.}

We divide the range $ 0 \leq b_{33} \leq \frac{26}{100} $ and $ 0 \leq b_{55} \leq \frac{26}{100} $ into sub-ranges. Tables~\ref{tbl:main_diag_case3_sub_range}, \ref{tbl:main_diag_case3_m_ii}, \ref{tbl:main_diag_case3_M_ii}, \ref{tbl:main_diag_case3_intermediates}, \ref{tbl:main_diag_case3_m_ij} and \ref{tbl:main_diag_case3_P_B_min} are arranged as in Section~\ref{sec:main_diag_case_1}, except that in Table~\ref{tbl:main_diag_case3_m_ij} we take into account the relations in \eqref{eq:b_55_geq_b11_and_b33_geq_b44}.

\begin{table}
\begin{center}
\begin{tabular}{ c *{4}{ >{$} c <{$} } }
\toprule
Sub-range & \tilde{m}_{33} & \tilde{M}_{33} & \tilde{m}_{55} & \tilde{M}_{55} \\
\midrule
\addlinespace
1 & 0 & \frac{12}{100} & 0 & \frac{26}{100} \\
\addlinespace
2 & \frac{12}{100} & \frac{26}{100} & 0 & \frac{18}{100} \\
\addlinespace
3 & \frac{12}{100} & \frac{26}{100} & \frac{18}{100} & \frac{21}{100} \\
\addlinespace
4 & \frac{12}{100} & \frac{26}{100} & \frac{21}{100} & \frac{26}{100} \\
\addlinespace
\bottomrule
\end{tabular}
\caption{Case 3 sub-ranges}
\label{tbl:main_diag_case3_sub_range}
\end{center}
\end{table}

\begin{table}
\begin{center}
\begin{tabular}{ c *{5}{ >{$} c <{$} } }
\toprule
Sub-range & m_{11} & m_{22} & m_{33} & m_{44} & m_{55} \\
\midrule
\addlinespace
1 & \frac{1}{25} & 0 & 0 & 0 & \frac{19}{200} \\
\addlinespace
2 & \frac{1}{50} & 0 & \frac{3}{25} & 0 & \frac{3}{50} \\
\addlinespace
3 & \frac{1}{100} & 0 & \frac{3}{25} & 0 & \frac{9}{50} \\
\addlinespace
4 & 0 & 0 & \frac{3}{25} & 0 & \frac{21}{100} \\
\addlinespace
\bottomrule
\end{tabular}
\caption{Case 3 $ m_{ii} $'s}
\label{tbl:main_diag_case3_m_ii}
\end{center}
\end{table}

\begin{table}
\begin{center}
\begin{tabular}{ c *{5}{ >{$} c <{$} } }
\toprule
Sub-range & M_{11} & M_{22} & M_{33} & M_{44} & M_{55} \\
\midrule
\addlinespace
1 & \frac{1}{4} & \frac{1}{10} & \frac{3}{25} & \frac{3}{25} & \frac{13}{50} \\
\addlinespace
2 & \frac{9}{50} & \frac{19}{200} & \frac{13}{50} & \frac{1}{8} & \frac{9}{50} \\
\addlinespace
3 & \frac{19}{100} & \frac{1}{15} & \frac{13}{50} & \frac{1}{10} & \frac{21}{100} \\
\addlinespace
4 & \frac{17}{100} & \frac{17}{300} & \frac{13}{50} & \frac{17}{200} & \frac{13}{50} \\
\addlinespace
\bottomrule
\end{tabular}
\caption{Case 3 $ M_{ii} $'s}
\label{tbl:main_diag_case3_M_ii}
\end{center}
\end{table}

\begin{table}
\begin{center}
\begin{tabular}{ c *{6}{ >{$} c <{$} } }
\toprule
Sub-range & r_{12} & \tilde{r}_{12} & \tilde{M}_{12} & r_{24} & \tilde{r}_{24} & \tilde{M}_{24} \\
\midrule
\addlinespace
1 & \frac{2,795,584}{390,625} & \frac{267}{100} & \frac{ \sqrt{73,334,649} }{15,863} & \frac{1,332}{625} & \frac{29}{20} & \frac{ \sqrt{60,605} }{391} \\
\addlinespace
2 & \frac{2,331}{625} & \frac{193}{100} & \frac{ \sqrt{22} }{8} & \frac{1,721,344}{390,625} & \frac{209}{100} & \frac{ 66 \sqrt{12,973} }{12,973} \\
\addlinespace
3 & \frac{63,936}{15,625} & \frac{101}{50} & \frac{ \sqrt{44,526} }{362} & \frac{5,752,701}{1,562,500} & \frac{191}{100} & \frac{ \sqrt{204,021,627} }{24,146} \\
\addlinespace
4 & \frac{1,685,979}{390,625} & \frac{207}{100} & \frac{ \sqrt{6,307,787} }{4,314} & \frac{1,216,116}{390,625} & \frac{44}{25} & \frac{ \sqrt{46,730,082} }{11,334} \\
\addlinespace
\bottomrule
\end{tabular}
\caption{Case 3 intermediates}
\label{tbl:main_diag_case3_intermediates}
\end{center}
\end{table}

\begin{table}
\begin{center}
\begin{tabular}{ c *{6}{ >{$} c <{$} } }
\toprule
Sub-range & m_{12} & m_{13} & m_{24} & m_{35}^{13} & m_{35}^{45} & m_{45} \\
\midrule
\addlinespace
1 & \frac{1}{2} & \frac{49}{100} & \frac{51}{100} & \frac{23}{50} * & \frac{11}{25} & \frac{49}{100} \left( \frac{41}{100} \right) \\
\addlinespace
2 & \frac{1}{2} & \frac{9}{20} & \frac{51}{100} & \frac{11}{25} & \frac{9}{20} * & \frac{49}{100} \\
\addlinespace
3 & \frac{13}{25} & \frac{47}{100} & \frac{13}{25} & \frac{9}{20} & \frac{9}{20} * & \frac{12}{25} \\
\addlinespace
4 & \frac{13}{25} & \frac{12}{25} & \frac{13}{25} & \frac{43}{100} * & \frac{21}{50} & \frac{12}{25} \left( \frac{23}{50} \right) \\
\addlinespace
\bottomrule
\end{tabular}
\caption{Case 3 $ m_{ij} $'s}
\label{tbl:main_diag_case3_m_ij}
\end{center}
\end{table}

\begin{table}
\begin{center}
\begin{tabular}{ c *{1}{ >{$} c <{$} } }
\toprule
Sub-range & P_{B_{min}} \left( 1 \right) \\
\midrule
\addlinespace
1 & -\frac{214,977,947}{20,000,000,000} \\
\addlinespace
2 & -\frac{335,607}{1,250,000,000} \\
\addlinespace
3 & -\frac{7,419,049}{156,250,000} \\
\addlinespace
4 & -\frac{7,326,711}{156,250,000} \\
\addlinespace
\bottomrule
\end{tabular}
\caption{Case 3 $ P_{B_{min}} \left( 1 \right) $ }
\label{tbl:main_diag_case3_P_B_min}
\end{center}
\end{table}

\subsubsection{$ b_{11} \geq b_{55} $ and $ b_{33} \geq b_{44} $}

By Assumption~\ref{assmp:min_diagonal}, Assumption~\ref{assmp:b_11_geq_b_44} and by Lemma~\ref{lem:not_similar_to_positive_C} we have
\begin{align}
b_{11} \geq b_{55} \geq b_{22}, \quad b_{11} \geq b_{44} \geq b_{22}, \quad b_{33} \geq b_{44} \geq b_{22} , \nonumber \\
b_{12} \geq b_{35}, \quad b_{24} \geq b_{35}, \quad b_{45} \geq b_{13}, \quad b_{24} \geq b_{13} . \label{eq:b_11_geq_b55_and_b33_geq_b44}
\end{align}
Let $ \tilde{m}_{11} \leq b_{11} \leq \tilde{M}_{11} $ and $ \tilde{m}_{33} \leq b_{33} \leq \tilde{M}_{33} $. We derive upper and lower bounds for the $ b_{ii} $'s using \eqref{eq:b_11_geq_b55_and_b33_geq_b44}.

\paragraph{Bounds for $ b_{11} $.}

We have
\begin{align*}
2 b_{11} & \geq b_{11} + b_{55} = \frac{1}{2} - b_{22} - b_{33} - b_{44} \geq \frac{1}{2} - 3 b_{33} \geq \frac{1}{2} - 3 \tilde{M}_{33} , \\
4 b_{11} & \geq b_{11} + b_{22} + b_{44} + b_{55} = \frac{1}{2} - b_{33} \geq \frac{1}{2} - \tilde{M}_{33} ,
\end{align*}
and so $ m_{11} \leq b_{11} \leq M_{11} $, where
\begin{align*}
m_{11} &= \max \left\{ \tilde{m}_{11}, \frac{1}{4} - \frac{3}{2} \tilde{M}_{33}, \frac{1}{8} - \frac{1}{4} \tilde{M}_{33} \right\} , \\
M_{11} &= \tilde{M}_{11} .
\end{align*}

\paragraph{Bounds for $ b_{33} $.}

We have
\begin{align*}
b_{33} &= \frac{1}{2} - b_{11} - b_{22} - b_{44} - b_{55} \leq \frac{1}{2} - b_{11} \leq \frac{1}{2} - m_{11} , \\
b_{33} &= \frac{1}{2} - b_{11} - b_{22} - b_{44} - b_{55} \geq \frac{1}{2} - 4 b_{11} \geq \frac{1}{2} - 4 M_{11} , \\
3 b_{33} & \geq b_{22} + b_{33} + b_{44} = \frac{1}{2} - b_{11} - b_{55} \geq \frac{1}{2} - 2 b_{11} \geq \frac{1}{2} - 2 M_{11} ,
\end{align*}
and so $ m_{33} \leq b_{33} \leq M_{33} $, where
\begin{align*}
m_{33} &= \max \left\{ \tilde{m}_{33}, \frac{1}{2} - 4 M_{11}, \frac{1}{6} - \frac{2}{3} M_{11} \right\} , \\
M_{33} &= \min \left\{ \tilde{M}_{33}, \frac{1}{2} - m_{11} \right\} .
\end{align*}

\paragraph{Bounds for $ b_{44} $.}

We have
\begin{align*}
b_{44} &= \frac{1}{2} - b_{11} - b_{22} - b_{33} - b_{55} \leq \frac{1}{2} - b_{11} - b_{33} \leq \frac{1}{2} - m_{11} - m_{33} , \\
2 b_{44} & \leq b_{33} + b_{44} \leq b_{22} + b_{33} + b_{44} + b_{55} = \frac{1}{2} - b_{11} \leq \frac{1}{2} - m_{11} , \\
2 b_{44} & \leq b_{11} + b_{44} \leq b_{11} + b_{22} + b_{44} + b_{55} = \frac{1}{2} - b_{33} \leq \frac{1}{2} - m_{33} , \\
3 b_{44} & \leq b_{11} + b_{33} + b_{44} \leq b_{11} + b_{22} + b_{33} + b_{44} + b_{55} = \frac{1}{2} ,
\end{align*}
and so, $ m_{44} \leq b_{44} \leq M_{44} $, where
\begin{align*}
m_{44} &= 0 , \\
M_{44} &= \min \left\{ \frac{1}{6}, M_{11}, M_{33}, \frac{1}{2} - m_{11} - m_{33}, \frac{1}{4} - \frac{1}{2} m_{11}, \frac{1}{4} - \frac{1}{2} m_{33} \right\} .
\end{align*}

\paragraph{Bounds for $ b_{55} $.}

We have
\begin{align*}
b_{55} &= \frac{1}{2} - b_{11} - b_{22} - b_{33} - b_{44} \leq \frac{1}{2} - b_{11} - b_{33} \leq \frac{1}{2} - m_{11} - m_{33} , \\
2 b_{55} & \leq b_{11} + b_{55} \leq b_{11} + b_{22} + b_{44} + b_{55} = \frac{1}{2} - b_{33} \leq \frac{1}{2} - m_{33} ,
\end{align*}
and so, $ m_{55} \leq b_{55} \leq M_{55} $, where
\begin{align*}
m_{55} &= 0 , \\
M_{55} &= \min \left\{ M_{11}, \frac{1}{2} - m_{11} - m_{33}, \frac{1}{4} - \frac{1}{2} m_{33} \right\} .
\end{align*}

\paragraph{Bounds for $ b_{22} $.}

We have
\begin{align*}
3 b_{22} & \leq b_{22} + b_{44} + b_{55} = \frac{1}{2} - b_{11} - b_{33} \leq \frac{1}{2} - m_{11} - m_{33} , \\
4 b_{22} & \leq b_{22} + b_{33} + b_{44} + b_{55} = \frac{1}{2} - b_{11} \leq \frac{1}{2} - m_{11} , \\
4 b_{22} & \leq b_{11} + b_{22} + b_{44} + b_{55} = \frac{1}{2} - b_{33} \leq \frac{1}{2} - m_{33} ,
\end{align*}
and so, by Assumption~\ref{assmp:min_diagonal}, $ m_{22} \leq b_{22} \leq M_{22} $, where
\begin{align*}
m_{22} &= 0 , \\
M_{22} &= \min \left\{ \frac{1}{10}, M_{11}, M_{33}, \frac{1}{6} - \frac{1}{3} m_{11} - \frac{1}{3} m_{33}, \frac{1}{8} - \frac{1}{4} m_{11}, \frac{1}{8} - \frac{1}{4} m_{33} \right\} .
\end{align*}

\paragraph{Sub-ranges.}

We divide the range $ 0 \leq b_{11} \leq \frac{1}{2} $ and $ 0 \leq b_{33} \leq \frac{26}{100} $ into sub-ranges. Tables~\ref{tbl:main_diag_case4_sub_range}, \ref{tbl:main_diag_case4_m_ii}, \ref{tbl:main_diag_case4_M_ii}, \ref{tbl:main_diag_case4_intermediates}, \ref{tbl:main_diag_case4_m_ij} and \ref{tbl:main_diag_case4_P_B_min} are arranged as in Section~\ref{sec:main_diag_case_1}, except that in Table~\ref{tbl:main_diag_case4_m_ij} we take into account the relations in \eqref{eq:b_11_geq_b55_and_b33_geq_b44}.

\begin{table}
\begin{center}
\begin{tabular}{ c *{4}{ >{$} c <{$} } }
\toprule
Sub-range & \tilde{m}_{11} & \tilde{M}_{11} & \tilde{m}_{33} & \tilde{M}_{33} \\
\midrule
\addlinespace
1 & 0 & \frac{15}{100} & 0 & \frac{26}{100} \\
\addlinespace
2 & \frac{15}{100} & \frac{25}{100} & 0 & \frac{15}{100} \\
\addlinespace
3 & \frac{15}{100} & \frac{25}{100} & \frac{15}{100} & \frac{26}{100} \\
\addlinespace
4 & \frac{25}{100} & \frac{40}{100} & 0 & \frac{10}{100} \\
\addlinespace
5 & \frac{25}{100} & \frac{40}{100} & \frac{10}{100} & \frac{20}{100} \\
\addlinespace
6 & \frac{25}{100} & \frac{40}{100} & \frac{20}{100} & \frac{26}{100} \\
\addlinespace
7 & \frac{40}{100} & \frac{50}{100} & 0 & \frac{26}{100} \\
\addlinespace
\bottomrule
\end{tabular}
\caption{Case 4 sub-ranges}
\label{tbl:main_diag_case4_sub_range}
\end{center}
\end{table}

\begin{table}
\begin{center}
\begin{tabular}{ c *{5}{ >{$} c <{$} } }
\toprule
Sub-range & m_{11} & m_{22} & m_{33} & m_{44} & m_{55} \\
\midrule
\addlinespace
1 & \frac{3}{50} & 0 & \frac{1}{15} & 0 & 0 \\
\addlinespace
2 & \frac{3}{20} & 0 & 0 & 0 & 0 \\
\addlinespace
3 & \frac{3}{20} & 0 & \frac{3}{20} & 0 & 0 \\
\addlinespace
4 & \frac{1}{4} & 0 & 0 & 0 & 0 \\
\addlinespace
5 & \frac{1}{4} & 0 & \frac{1}{10} & 0 & 0 \\
\addlinespace
6 & \frac{1}{4} & 0 & \frac{1}{5} & 0 & 0 \\
\addlinespace
7 & \frac{2}{5} & 0 & 0 & 0 & 0 \\
\addlinespace
\bottomrule
\end{tabular}
\caption{Case 4 $ m_{ii} $'s}
\label{tbl:main_diag_case4_m_ii}
\end{center}
\end{table}

\begin{table}
\begin{center}
\begin{tabular}{ c *{5}{ >{$} c <{$} } }
\toprule
Sub-range & M_{11} & M_{22} & M_{33} & M_{44} & M_{55} \\
\midrule
\addlinespace
1 & \frac{3}{20} & \frac{1}{10} & \frac{13}{50} & \frac{3}{20} & \frac{3}{20} \\
\addlinespace
2 & \frac{1}{4} & \frac{7}{80} & \frac{3}{20} & \frac{3}{20} & \frac{1}{4} \\
\addlinespace
3 & \frac{1}{4} & \frac{1}{15} & \frac{13}{50} & \frac{1}{6} & \frac{7}{40} \\
\addlinespace
4 & \frac{2}{5} & \frac{1}{16} & \frac{1}{10} & \frac{1}{10} & \frac{1}{4} \\
\addlinespace
5 & \frac{2}{5} & \frac{1}{20} & \frac{1}{5} & \frac{1}{8} & \frac{3}{20} \\
\addlinespace
6 & \frac{2}{5} & \frac{1}{60} & \frac{1}{4} & \frac{1}{20} & \frac{1}{20} \\
\addlinespace
7 & \frac{1}{2} & \frac{1}{40} & \frac{1}{10} & \frac{1}{20} & \frac{1}{10} \\
\addlinespace
\bottomrule
\end{tabular}
\caption{Case 4 $ M_{ii} $'s}
\label{tbl:main_diag_case4_M_ii}
\end{center}
\end{table}

\begin{table}
\begin{center}
\begin{tabular}{ c *{6}{ >{$} c <{$} } }
\toprule
Sub-range & r_{12} & \tilde{r}_{12} & \tilde{M}_{12} & r_{24} & \tilde{r}_{24} & \tilde{M}_{24} \\
\midrule
\addlinespace
1 & \frac{52,836}{15,625} & \frac{183}{100} & \frac{ 3 \sqrt{828,714} }{4,682} & \frac{14,161}{2,500} & \frac{119}{50} & \frac{9}{16} \\
\addlinespace
2 & \frac{14,161}{2,500} & \frac{119}{50} & \frac{ 3 \sqrt{2} }{8} & \frac{9}{4} & \frac{3}{2} & \frac{5}{8} \\
\addlinespace
3 & \frac{1,184}{375} & \frac{177}{100} & \frac{ 2 \sqrt{149,003} }{1,367} & \frac{1,287}{400} & \frac{179}{100} & \frac{ 5 \sqrt{5,406} }{612} \\
\addlinespace
4 & \frac{5,184}{625} & \frac{72}{25} & \frac{ \sqrt{161} }{26} & \frac{18}{25} & \frac{21}{25} & \frac{ 5 \sqrt{1,507} }{274} \\
\addlinespace
5 & \frac{126}{25} & \frac{56}{25} & \frac{ \sqrt{140,910} }{732} & \frac{714}{625} & \frac{53}{50} & \frac{ \sqrt{66} }{12} \\
\addlinespace
6 & \frac{513}{100} & \frac{113}{50} & \frac{ \sqrt{20,155} }{278} & \frac{1,026}{625} & \frac{32}{25} & \frac{ \sqrt{493} }{34} \\
\addlinespace
7 & \frac{6,156}{625} & \frac{313}{100} & \frac{ \sqrt{105,415} }{727} & 0 & 0 & \frac{ \sqrt{7} }{3} \\
\addlinespace
\bottomrule
\end{tabular}
\caption{Case 4 intermediates}
\label{tbl:main_diag_case4_intermediates}
\end{center}
\end{table}

\begin{table}
\begin{center}
\begin{tabular}{ c *{6}{ >{$} c <{$} } }
\toprule
Sub-range & m_{12} & m_{13} & m_{24} & m_{35}^{13} & m_{35}^{45} & m_{45} \\
\midrule
\addlinespace
1 & \frac{1}{2} & \frac{9}{20} & \frac{13}{25} & \frac{47}{100} & \frac{49}{100} * & \frac{27}{50} \\
\addlinespace
2 & \frac{1}{2} & \frac{23}{50} & \frac{57}{100} & \frac{9}{20} * & \frac{9}{20} & \frac{23}{50} \left( \frac{11}{25} \right) \\
\addlinespace
3 & \frac{1}{2} & \frac{39}{100} & \frac{57}{100} & \frac{21}{50} & \frac{43}{100} * & \frac{12}{25} \\
\addlinespace
4 & \frac{1}{2} & \frac{2}{5} & \frac{61}{100} & \frac{43}{100} * & \frac{21}{50} & \frac{2}{5} \left( \frac{19}{50} \right) \\
\addlinespace
5 & \frac{1}{2} & \frac{7}{20} & \frac{61}{100} & \frac{11}{25} * & \frac{11}{25} & \frac{43}{100} \\
\addlinespace
6 & \frac{13}{25} \left( \frac{1}{2} \right) & \frac{39}{100} & \frac{61}{100} & \frac{13}{25} * & \frac{13}{25} & \frac{53}{100} \\
\addlinespace
7 & \frac{1}{2} & \frac{7}{25} & \frac{67}{100} & \frac{12}{25} * & \frac{23}{50} & \frac{21}{50} \\
\addlinespace
\bottomrule
\end{tabular}
\caption{Case 4 $ m_{ij} $'s}
\label{tbl:main_diag_case4_m_ij}
\end{center}
\end{table}

\begin{table}
\begin{center}
\begin{tabular}{ c *{1}{ >{$} c <{$} } }
\toprule
Sub-range & P_{B_{min}} \left( 1 \right) \\
\midrule
\addlinespace
1 & -\frac{17,003,473}{468,750,000} \\
\addlinespace
2 & -\frac{1,684,287}{80,000,000} \\
\addlinespace
3 & -\frac{31,285,863}{2,000,000,000} \\
\addlinespace
4 & -\frac{6,702,413}{400,000,000} \\
\addlinespace
5 & -\frac{136,283}{6,250,000} \\
\addlinespace
6 & -\frac{153,180,277}{1,250,000,000} \\
\addlinespace
7 & -\frac{2,089,397}{31,250,000} \\
\addlinespace
\bottomrule
\end{tabular}
\caption{Case 4 $ P_{B_{min}} \left( 1 \right) $ }
\label{tbl:main_diag_case4_P_B_min}
\end{center}
\end{table}